\documentclass[11pt]{article}

\usepackage{amsmath, amsthm}
\usepackage{amsmath, amsfonts}
\usepackage{amsmath, amssymb}
\usepackage{amsmath}
\usepackage{graphics}
\usepackage{graphicx}
\usepackage{color}
\usepackage{hyperref}
\usepackage[all]{xy}

\newtheorem{theorem}{Theorem}[subsection]
\newtheorem{definition}[theorem]{Definition}
\newtheorem{definition-lemma}[theorem]{Definition/Lemma}
\newtheorem{definition-explanation}[theorem]{Definition/Explanation}
\newtheorem{explanation}[theorem]{Explanation}
\newtheorem{explanation-definition}[theorem]{Explanation/Definition}
\newtheorem{definition-fact}[theorem]{Definition/Fact}
\newtheorem{definition-notation}[theorem]{Definition/Notation}
\newtheorem{definition-conjecture}[theorem]{Definition/Conjecture}
\newtheorem{lemma}[theorem]{Lemma}
\newtheorem{lemma-definition}[theorem]{Lemma/Definition}
\newtheorem{proposition}[theorem]{Proposition}
\newtheorem{corollary}[theorem]{Corollary}
\newtheorem{remark}[theorem]{\it Remark}
\newtheorem{remark-notation}[theorem]{\it Remark/Notation}

\newtheorem{application-lemma}[theorem]{Application/Lemma}

\newtheorem{assumption}[theorem]{\it Assumption}

\newtheorem{example}[theorem]{Example}
\newtheorem{example-definition}[theorem]{Example/Definition}

\newtheorem{definition-prototype}[theorem]{Definition-Prototype}
\newtheorem{question}[theorem]{Question}

\numberwithin{equation}{subsection}

\newtheorem{stheorem}{Theorem}[section]
\newtheorem{sdefinition}[stheorem]{Definition}
\newtheorem{sdefinition-lemma}[stheorem]{Definition/Lemma}
\newtheorem{sdefinition-explanation}[stheorem]{Definition/Explanation}

\newtheorem{sexplanation-definition}[stheorem]{Explanation/Definition}
\newtheorem{sdefinition-fact}[stheorem]{Definition/Fact}
\newtheorem{sdefinition-notation}[stheorem]{Definition/Notation}
\newtheorem{sdefinition-conjecture}[stheorem]{Definition/Conjecture}
\newtheorem{slemma}[stheorem]{Lemma}
\newtheorem{slemma-definition}[stheorem]{Lemma/Definition}
\newtheorem{sproposition}[stheorem]{Proposition}
\newtheorem{scorollary}[stheorem]{Corollary}

\newtheorem{sremark-notation}[stheorem]{\it Remark/Notation}

\newtheorem{sapplication-lemma}[stheorem]{Application/Lemma}

\newtheorem{sexample}[stheorem]{Example}
\newtheorem{sexample-definition}[stheorem]{Example/Definition}

\newtheorem{sdefinition-prototype}[stheorem]{Definition-Prototype}
\newtheorem{squestion}[stheorem]{Question}

\newtheorem{ssdefinition-lemma}[sstheorem]{Definition/Lemma}
\newtheorem{ssdefinition-explanation}[sstheorem]{Definition/Explanation}

\newtheorem{ssexplanation-definition}[sstheorem]{Explanation/Definition}
\newtheorem{ssdefinition-fact}[sstheorem]{Definition/Fact}
\newtheorem{ssdefinition-notation}[sstheorem]{Definition/Notation}
\newtheorem{ssdefinition-conjecture}[sstheorem]{Definition/Conjecture}

\newtheorem{sslemma-definition}[sstheorem]{Lemma/Definition}

\newtheorem{ssremark-notation}[sstheorem]{\it Remark/Notation}

\newtheorem{ssapplication-lemma}[sstheorem]{Application/Lemma}

\newtheorem{ssexample-definition}[sstheorem]{Example/Definition}

\newtheorem{ssdefinition-prototype}[sstheorem]{Definition-Prototype}

 \newcommand{\dashAlgscriptsize}{{\mbox{\scriptsize\it -Alg}\,}}

\newcommand{\Ann}{{\mbox{\it Ann}\,}}

\newcommand{\Az}{{\mbox{\it Az}}}
 
 \newcommand{\Azscriptsize}{{\mbox{\scriptsize\it A$\!$z}}}
 \newcommand{\Aztiny}{{\mbox{\tiny\it A$\!$z}}}

\newcommand{\Cayleyscriptsize}{{\mbox{\it\scriptsize Cayley}\,}}

\newcommand{\Center}{{\mbox{\it Center}\,}}
 \newcommand{\scriptsizeCenter}{{\mbox{\scriptsize\it Center}\,}}
\newcommand{\Centralizer}{{\mbox{\it Centralizer}\,}}

\newcommand{\End}{\mbox{\it End}\,}

\newcommand{\Endsheaf}{{\mbox{\it ${\cal E}\!$nd}\,}}

\newcommand{\Hom}{\mbox{\it Hom}\,}
\newcommand{\Homsheaf}{\mbox{\it ${\cal H}$om}\,}

\newcommand{\Image}{\mbox{\it Im}\,}

\newcommand{\Isom}{\mbox{\it Isom}\,}

\newcommand{\Ker}{\mbox{\it Ker}\,}

\newcommand{\Morphism}{\mbox{\it Morphism}\,}

\newcommand{\Nil}{\mbox{\it Nil}\,}

\newcommand{\Object}{\mbox{\it Object}\,}
\newcommand{\ObjectCategory}{\mbox{\it ${\cal O}$bject}\,}

\newcommand{\Scheme}{\mbox{\it ${\cal S}\!$cheme}\,}

\newcommand{\Space}{\mbox{\it Space}\,}

\newcommand{\Spec}{\mbox{\it Spec}\,}
 \newcommand{\boldSpec}{\mbox{\it\bf Spec}\,}

\newcommand{\Supp}{\mbox{\it Supp}\,}
 \newcommand{\scriptsizeSupp}{\mbox{\scriptsize\it Supp}\,}

\newcommand{\anticommuting}{{\mbox{\scriptsize\it anti-c}}}

\newcommand{\scriptsizelower}{{\mbox{\scriptsize\it lower}}}

\newcommand{\nc}{\mbox{\it nc}}

\newcommand{\scriptsizeop}{{\mbox{\scriptsize\it op}}}

\newcommand{\pr}{\mbox{\it pr}}

\newcommand{\pt}{\mbox{\it pt}}

\newcommand{\scriptsizeupper}{{\mbox{\scriptsize\it upper}}}

\newcommand{\tinyclubsuit}{{\mbox{\tiny $\clubsuit$}}}

\newcommand{\longrightaarrow}{\longrightarrow\hspace{-3ex}\longrightarrow}
\newcommand{\rightaarrow}{\rightarrow\hspace{-2ex}\rightarrow}

\newcommand{\tinybullet}{{\raisebox{.2ex}{\tiny $\bullet$}}}

\begin{document}

\enlargethispage{24cm}

\begin{titlepage}

$ $

\vspace{-1.5cm} 

\noindent\hspace{-1cm}
\parbox{6cm}{\small June 2021}\
   \hspace{7cm}\
   \parbox[t]{6cm}{\small
                arXiv:yymm.nnnnn [math.AG] \\
                D(15.1), NCS(1):\\ $\mbox{\hspace{2em}}$  	
				D-brane on soft NC space
				}

\vspace{2cm}

\centerline{\large\bf
 Soft noncommutative schemes via toric geometry and}
\vspace{1ex}
\centerline{\large\bf  
 morphisms from an Azumaya scheme with a fundamental module thereto}
\vspace{1ex}
\centerline{\large\bf  
 ---  (Dynamical, complex algebraic) D-branes on a soft noncommutative space}

\bigskip

\vspace{3em}

\centerline{\large
  Chien-Hao Liu   
            \hspace{1ex} and \hspace{1ex}
  Shing-Tung Yau
}

\vspace{4em}

\begin{quotation}
\centerline{\bf Abstract}
\vspace{0.3cm}

\baselineskip 12pt  
{\small
 A class of noncommutative spaces, named {\it soft noncommutative schemes via toric geometry},
   are constructed and 
 the mathematical model for (dynamical/nonsolitonic, complex algebraic) {\it D-branes} on such a noncommutative space,
    following arXiv:0709.1515 [math.AG] (D(1)), is given.
 Any algebraic Calabi-Yau space that arises from a complete intersection in a smooth toric variety
   can embed as a commutative closed subscheme of some soft noncommutative scheme.	
 Along the study, the notion of
  {\it soft noncommutative toric schemes associated to a} (simplicial, maximal cone of index $1$) {\it fan},
  {\it invertible sheaves} on such a noncommutative space, and
  {\it twisted sections} of an invertible sheaf
   are developed    and
 Azumaya schemes with a fundamental module as the world-volumes of  D-branes are reviewed.
 Two guiding questions,
        Question~3.12 (soft noncommutative Calabi-Yau spaces and their mirror) and
	    Question~4.2.14 (generalized matrix models), 	
   are presented.  				
 } 
\end{quotation}

\vspace{13em}

\baselineskip 12pt
{\footnotesize
} 

 \bigskip

\noindent {\small MSC number 2020:  14A22; 14A15, 14M25; 81T30    
} 

\bigskip

\baselineskip 10pt
{\scriptsize
\noindent{\bf Acknowledgements.}
 We thank
  Andrew Strominger, Cumrun Vafa
    for influence to our understanding of strings, branes, and gravity.
 C.-H.L.\ thanks in addition
  Chin-Lung Wang
    for communications on Noncommutative Algebraic Geometry during the brewing years;
  Enno Ke{\ss}ler, Carlos Shahbazi, Peter West
    for communications related to a concurrent in-progress work during the COVID-19 lockdown year
    that will cross the current work in the future;
  Xi Yin and Sebastien Picard
    for topic courses, spring 2020;
  Dror Bar-Natan, Galia Dafni, Ayelet Lindenstrauss
    for a surprise contact;
  Ren-De Hwa
    for bringing MuseScore to his attention, spring 2020, just before the lockdown;
  Pei-Jung Chen
    for the biweekly communications on J.S.\ Bach's sonata  and
  Ling-Miao Chou
    for the daily exchange of progress and
         comments that improve the illustrations and the tremendous moral support,
    which together boost the morale during the lockdown;
    and
  the numerous people behind the scene
    who make the continuation of projects possible during the COVID-19 pandemic times
 	 (cf.\ dedication \& continuing acknowledgements).
 The project is supported by NSF grants DMS-9803347 and DMS-0074329.
} 

\end{titlepage}

\newpage

\enlargethispage{24cm}
\begin{titlepage}

$ $

\vspace{-2em}

\centerline{\small\it
 Chien-Hao Liu dedicates this first work in the second spin-off NCS of the D-project to}
 \centerline{\small\it
    a group of familiar strangers, including}
 \centerline{\small\it
    the building common-area cleaners and waste collectors,}
 \centerline{\small\it post-and-parcel delivery-service persons,}
 \centerline{\small\it
    and workers and staff in the nearby}
 \centerline{\small\it
     Star Market/Shaw's, Walgreens, Reliable Market, Inman Square Hardware,}
 \centerline{\small\it
      CVS Pharmacy, Skenderian Apothecary,}	
 \centerline{\small\it	
 	 post office at Harvard Square, Kinko's/FedEx at Harvard Square,
      Bob Slate Stationer,}
 \centerline{\small\it
      Bank of America at Harvard Square, Cambridge Savings Bank at Harvard Square and Inman Square,}
 \centerline{\small\it	
      Sullivan Tires and Auto Service, and Cambridge Honda,}
 \centerline{\small\it
      and Christine Rishebarger, Joseph Yager, and their team members,}
 \centerline{\small\it
   and many more people behind the scene whom he would likely never have a chance to meet}
 \centerline{\small\it
   for their indispensable work and services in the COVID-19 pandemic times}
 \centerline{\small\it
   that make his life going on as normally as it can be;}
 \centerline{\small\it
   and to Lynda Azar,
      who served in the Condominium Board for over a decade to promote a lot of improvement}
 \centerline{\small\it
 	 but sadly had to change her course of life due to the COVID-19 pandemic.}
 
 \vspace{2em}
 
 \noindent
 {\footnotesize
 {\it Continuing acknowledges from C.-H.L.}:\hspace{1em}
  Since ancient times death has been a taboo in many cultures.
  Yet, no matter how glorious one may be through one's life,
   at the end one still has to face one's final fate and last journey: death, which one can only face alone.
  Having lived a very simple life in the United States for decades,
  naturally I would not expect this to be a pressing issue that I have to think about right now,
   despite that every incident of death of people I had been close to and cherished in my life
      always provoked some deep reflections.
  Then came the COVID-19 pandemic outbreak, spring 2020, in the United States.
  When one read the daily statistics of deaths and
                  lived in a situation that essential items in grocery stores had to be rationed,
    one was immediately submerged into an atmosphere of a seemingly wartime and
    the issue of death looked indeed very near and real.
  What work or deed or achievement or honor would still retain its meaning in front of death?
  This is a counter-question to the question of the meaning of life,
    which we mostly don't bother to think about in a peaceful time.
   I would like to thank in addition the following
    authors, lecturers, musicians, developers, creators, artists, and hosts
   \begin{itemize}
    \item[$\boldsymbol{\cdot}$]  
     Kelli Francis-Staite, Dominic Joyce; \hspace{.6em}
 	 Daniel S.\ Freed; \hspace{.6em}
 	 Mark Gross; \hspace{.6em}
	 Arthur Ogus; \\
     Costas Bachas; \hspace{.6em}
	 Sergio Ferrara; \hspace{.6em}	
 	 Amihay Hanany; \hspace{.6em}
     Kentaro Hori, Amer Iqbal, Cumrun Vafa;\\  
 	 Gregory W.\ Moore; \hspace{.6em}
 	 Joseph Polchinski$^\dagger$; \hspace{.6em}
 	 Leonard Susskind; \hspace{.6em}
 	 Steven Weinberg$^\dagger$
   
    \item[$\boldsymbol{\cdot}$]  
     William Christ, Richard DeLone$^\dagger$, Vernon Kliewer$^\dagger$,
  	 Lewis Rowell, William Thomson;\hspace{.6em}
 	 Craig Wright
    	
    \item[$\cdot$]  
     Arthur H.\ Benade$^\dagger$;\hspace{.6em}
     Michel Debost, Jeanne Debost-Roth;\hspace{.6em}
  	 Seta Der Hohannesian;\hspace{.6em}
     John C.\ Krell$^\dagger$;\\
     Hui-Hsuan Liang;\hspace{.6em}
     Jia-Hua Lin;\hspace{.6em}
 	 Yi-Hui Lin;\hspace{.6em}
 	 David McGill;\hspace{.6em}
 	 Emmanuel Pahud;\hspace{.6em}
 	 Amy Porter;\\
 	 Erwin Stein$^\dagger$;\hspace{.6em} 	
 	 William Vennard$^\dagger$;\hspace{.6em}
 	 Xiao-Jen Wu
    
    \item[$\cdot$]  
     \makebox[6em][l]{[{\sl MuseScore}]}
 	 Thomas Bonte, Nicolas Froment, Werner Schweer, plus volunteers;\\
 	 \makebox[6em][l]{[{\sl Audacity}]}
 	 Roger Dannenberg, Dominic Mazzoni, plus volunteers
 	
    \item[$\boldsymbol{\cdot}$]  
     Michael Schuenke,  Eric Schulte, Udo Schumacher:
     Markus Voll, Karl Wesker:  \\
     Anne M.\ Gilroy, Brian R.\ MacPherson, Lawrence M.\ Ross
    
    \item[$\cdot$]
     Arthur C.\ Guyton$^\dagger$, John E.\ Hall, Michael E.\ Hall
	
    \item[$\boldsymbol{\cdot}$]  
     Charles Kuralt$^\dagger$;\hspace{.6em}
 	 Juan Li;\hspace{.6em}
     Ying Liu;\hspace{.6em}
 	 Quan-Chong Luo$^\dagger$;\hspace{.6em}
 	 Qingchun Shu (Lao She)$^\dagger$;\hspace{.6em}
 	 Ann Voskamp;\\    
 	 Rui-Hong Xia;\hspace{.6em}
 	 Malala Yousafzai
 	
    \item[$\boldsymbol{\cdot}$]  
     Robert Barrett;\hspace{.6em}
 	 George Bridgman$^\dagger$;\hspace{.6em}
     Andrew Loomis$^\dagger$;\hspace{.6em}
 	 Frank Reilly$^\dagger$;\hspace{.6em}
 	 John Singer Sargent$^\dagger$;\\
 	 Mau-Kun Yim, Iris Yim\:\:\:/\!/\hspace{2em}				
 	 Danqing Chen;\hspace{.6em}
     Susie Hodge\:\:\:/\!/\hspace{2em}	
     Sarah Simblet;\hspace{.6em}
     Uldis Zarins 	
 	
    \item[$\boldsymbol{\cdot}$]  
     Yu-Han Ho;\hspace{.6em}
     Sarah Ivanhoe;\hspace{.6em}
     Jennifer Kries;\hspace{.6em}
 	 Yu-Han Liu;\hspace{.6em}
 	 Yo-Tse Tsai
 	
    \item[$\boldsymbol{\cdot}$]  
     Cheng Hsuan;\hspace{.6em}
     Brij Kothari;\hspace{.6em}
     Wei-San Wang	
 	
    \item[$\boldsymbol{\cdot}$]  
     Waldemar Januszczak;\hspace{.6em}
     Ziqi Li;\hspace{.6em}
 	 program hosts of 99.5 WCRB - Classical Radio Boston
   \end{itemize}
   
  \noindent
   for their book or book series, works, lecture series, master classes,
        demonstrations, freewares, creations, life stories, programs, or intellectual legacies
   that provided
              intellectual and physical challenges,
 			 inner growth,
 			 a tool to access some of the most creative minds crossing over eight hundred years
			      from Bernart de Ventadorn (1135-1194) to Aaron Copland  (1900-1990),
			 and the joy of learning and reading
          for a life in social distancing and isolation,
      immersed me in the various categories of arts and beauties humankind had ever created,
 	 and made this accidental soul-searching journey that accompanied the brewing and preparation
 	         of the current and other in-progress work less scary while under the shadow of death.
 			 \hfill ($^\dagger$: deceased; with salute from CHL)
  %
 } 
 
 %
 %

\end{titlepage}


\newpage
$ $

\vspace{-3em}

\centerline{\sc
 Soft Noncommutative Schemes and Morphisms from Azumaya Schemes Thereto
 } %

\vspace{1.2em}


\begin{flushleft}
{\Large\bf 0. Introduction and outline}
\end{flushleft}
The realization of the noncommutative feature of D-brane world-volumes ([H-W], [Po1], [Wi]; [L-Y1], [Liu]) makes D-branes
 a good candidate as a probe to noncommutative geometry.
Unfortunately, all the three building blocks of modern Commutative Algebraic Geometry ---
  the notion of localizations of a ring,
  the method of associating a topology to a ring via Spec,    and
  the approach to understand a ring by studying its category of modules  ---
 have their limitation when extending to noncommutative rings and Noncommutative Algebraic Geometry.
(See, e.g., [Ro1], [Ro2], [B-R-S-S-W] to get a feel.)
Unlike a (fundamental) superstring world-sheet, which can have a very abundant class of geometry as its target-spaces
   ---most notably Calabi-Yau spaces ---
 it is not immediate clear
   what class of noncommutative spaces can serve as the target-space of a dynamical super D-string world-sheet
  and how to construct them, let alone the even deeper issue of Mirror Symmetry phenomenon when two different target-spaces can
  give rise to isomorphic $2$-dimensional supersymmetric quantum field theories on the superstring world-sheet.

In this work we introduce a class of noncommutative spaces, named `{\it soft noncommutative schemes}' via toric geometry
  (Sec.\ 2 and Sec.\ 3),
 to partially take care of the persistent difficulty of gluing in general Noncommutative Algebraic Geometry.
This  by construction is only a very limited class of noncommutative spaces,
  yet abundant enough that all the algebraic Calabi-Yau spaces that arise from complete intersections in a smooth toric variety
  can show up as a commutative subscheme of some soft noncommutative scheme (Corollary~3.11).
A  mathematical model for {\it (dynamical, complex algebraic) D-brane on such a noncommutative space}
  is then given $\mbox{(Sec.\ 4)}$ as morphisms from Azumaya schemes with a fundamental module thereto.
From this, one is certainly very curious as to how Mirror Symmetry stands when these soft noncommutative spaces are taken
 to be the target-spaces of D-string world-sheets.
 

\bigskip

\noindent
{\bf Convention.}
 References for standard notations, terminology, operations and facts are\\
  (1) Azumaya/matrix  algebra: [Ar], [Az], [A-N-T];   \hfill    
  (2) commutative monoid: [Og];     \\     
  (3) toric geometry: [Fu];   \hfill  
  (4) aspects of noncommutative algebraic geometry: [B-R-S-S-W];   \\  
  (5) commutative algebra and  (commutative) algebraic geometry: [Ei], [E-H], [Ha]; \\    
  (6) string theory and D-branes: [Ba], [Jo], [Po1], [Po2], [Sz].
 \begin{itemize}
  %
  %
  %
  
  \item[$\cdot$]
  All commutative schemes are over ${\Bbb C}$ and Noetherian.

  \item[$\cdot$]
   For a sheaf ${\cal F}$ on a gluing system of charts,
   the notation `$s\in{\cal F}$' means a local section $s\in {\cal F}(U)$ for some chart $U$.

 \end{itemize}

\bigskip
   
\begin{flushleft}
{\bf Outline}
\end{flushleft}
\nopagebreak
{\small
\baselineskip 12pt  
\begin{itemize}
 \item[0]
  Introduction      \hfill 1

 \item[1]
  Monoids and monoid algebras      \hfill 2
  
 \item[2]
  \makebox[42.8em][s]{Soft noncommutative toric schemes\hfill 4}
  \vspace{-3.4ex}
  \begin{itemize}	 	
    \item[2.1]	
    The noncommutative affine $n$-space {\it nc}$\mathbf{A}^n_{\Bbb C}$ and its $0$-dimensional subschemes
	  \dotfill 5
	
	\item[2.2]
	Soft noncommutative toric schemes associated to a fan      \dotfill 8
  \end{itemize}
  
 \item[3]
  Invertible sheaves on a soft noncommutative toric scheme, twisted sections,     and\\
  soft noncommutative schemes via toric geometry      \hfill 15
  
 \item[4]
  \makebox[42.8em][s]{(Dynamical) D-branes on a soft noncommutative space      \hfill 21}
  \vspace{-3.6ex}
  \begin{itemize}	 	
    \item[4.1]	
	Review of Azumaya/matrix schemes with a fundamental module over ${\Bbb C}$      \dotfill 22
	
	\item[4.2]
	Morphisms from an Azumaya scheme with a fundamental module to a soft\\
	noncommutative scheme $\breve{Y}$: (Dynamical) D-branes on $\breve{Y}$      \dotfill 23
	
    \item[4.3]	
	Two equivalent descriptions of a morphism $\varphi: X^{\!A\!z}\rightarrow \breve{Y}$    \dotfill 30
  \end{itemize}
  
 %
 %
 %
\end{itemize}
} 

\newpage

\section{Monoids and monoid algebras}

Basic definitions of monoids and monoid algebras we need for the current work are collected in this section
  to fix terminology and notation.
Readers are referred to [De], [Ge], [G-H-V] and [Og], from where these definitions are adapted, for more details.

\bigskip

\begin{sdefinition}{\bf [monoid]}\; {\rm
 A {\it monoid} is a triple $(M, \star, e_M)$, abbreviated by $(M,\star)$ or $M$, that consists of
    a set $M$,
	an associative binary operation $\star$, and
	a (two-sided) identity element $e_M$ of $(M, \star)$:
	  $e_M\star m = m\star e_M = m$ for all $m\in M$.
 An element $m\in M$ is {\it invertible} if there exists an element $m^\prime\in M$
   such that $m\star m^\prime = m^\prime\star m = e_M$.
 When $m$ is invertible, such $m^\prime$ is unique and is denoted $m^{-1}$, called the {\it inverse} of $m$.
 The {\it center} $\Center(M)$ of $M$
    is the submonoid  $\{c\in M\,|\, c\star m = m\star c,\, \mbox{for all $m\in M$} \}$ of $M$.
 A {\it homomorphism of monoids} (or {\it monoid homomorphism})
   is a function $\theta: M\rightarrow N$ between monoids such that
   $\theta(e_M)=e_N$ and $\theta(m\star m^\prime)=\theta(m)\star\theta(m^\prime)$.
 $\theta$ is called a {\it monomorphism} if it is injective, and an {\it epimorphism} if surjective.
   
 Given a set $S$, the {\it free monoid} $F(S)$ associated to $S$ is
  the monoid that consists of
   an element $e$ and
   all formal {\it words} $s_1\cdots s_k$ of finitely many elements of $S$,
    with the product rules
    $s_1\cdots s_k\star s_1^\prime\cdots s_l^\prime = s_1\cdots s_ks_1^\prime\cdots s_l^\prime$
	  and
	$e\star s_1\cdots s_k = s_1\cdots s_k \star e = s_1\cdots s_k$.
    
 Let $S\subset M$ be a subset of a monoid.
 We say that {\it $M$ is generated by $S$} if the monoid homomorphism
  $F(S)\rightarrow M$ specified by the inclusion $S\hookrightarrow M$,
     with $e\mapsto e_M$ and\\     $s_1\cdots s_k\mapsto s_1\star\cdots\star s_k$,
	is surjective.
 We say that a monoid $M$ is {\it finitely generated} if $M$ is generated by a finite subset $S\subset M$.
 We say that
   a generating set $S$ of $M$ is {\it inverse-closed} if $s\in S$ and $s$ is invertible in $M$, then $s^{-1}\in S$.
   
 By convention, when a monoid is written multiplicatively and the monoid operation $\star$ is clear from the content,
   $m_1\star m^\prime$ will be denoted simply $m\cdot m^\prime$ or $mm^\prime$.
 Commutative monoids are often written additively, with $\star$ denoted $+$ and $e_M$ denoted $0$.
}\end{sdefinition}

%

\medskip

\begin{sdefinition} {\bf [product and central extension]}\; {\rm
 Let $(C, \cdot, 1)$ and $(M, \star, e_M)$ be monoids.
 The {\it product} of $(C, \cdot, 1)$ and $(M, \star, e_M)$ is the monoid with elements $(c, m)\in C\times M$
   with $(c_1, m_1)\ast (c_2, m_2)= (c_1\cdot c_2, m_1\star m_2)$ and identity $(1, e_M)$.
 When $C$ is commutative, we'll denote the product $C\cdot M$ or $CM$ and elements $cm$ .
 In this case, $C\hookrightarrow C\cdot M\rightaarrow M$, with $c\mapsto c\,e_M$ and $cm\mapsto m$,
  is a {\it central extension} of $M$.
}\end{sdefinition}

\medskip

\begin{sdefinition} {\bf [Cayley graph of monoid with finite generating set]}\; {\rm
 Let $(M, \star, e_M)$ be a monoid with an inverse-closed, finite generating set $S$.
 Then the {\it Cayley graph} $\Gamma_\Cayleyscriptsize(M,S)$ of $(M,S)$ is a directed graph with
    the set of vertices $M$ and, for  every pair $(m_1, m_2)$ of vertices
    a directed edge labelled by $s\in S$ from $m_1$ to $m_2$ if $m_2 = m_1\star s$. 	
 We take the convention that if both $s$ and $s^{-1}$ are in $S$,
   then the edge associated to $s$ and that associated to $s^{-1}$ coincide with opposite directions;
 (i.e.\ the corresponding edge is bi-directed).
 Since every $m\in M$ can be expressed as a finite word in letters from $S$,
   $\Gamma_\Cayleyscriptsize(M,S)$ is a connected graph.
}\end{sdefinition}

\medskip

\begin{sdefinition} {\bf [edge-path monoid]}\; {\rm
 (Continuing Definition~1.3.)
 Endow the Cayley graph $\Gamma :=\Gamma_\Cayleyscriptsize(M,S)$ of $(M,S)$
    with the topology from the geometric realization of $\Gamma$ as a $1$-dimensional simplicial complex.
 An {\it edge-path at the vertex $m_0$} on $\Gamma$, with the {\it terminal vertex} $m_t$,
  is a direction-preserving continuous map $\gamma: [0,1]\rightarrow \Gamma$,
     with $\gamma(0)=m_0$ and $\gamma(1)=m_t$,
  that is piecewise linear after a subdivision of the interval $[0,1]$ as a $1$-dimensional simplicial complex.
 Thus, associated to each edge-path $\gamma$ at $m_0$  is a word $w_\gamma = s_1\cdots s_k$ in letters from $S$,
   where $s_1$ is the initial edge $\gamma$ takes from the vertex $m_0$ and $s_k$ is the final edge $\gamma$ takes
   to the terminal vertex $m_t$.
 Given an edge-path $\gamma$ at $m_0$ and a vertex $m\in M$,
  there is a {\it translation}  $T_m: \gamma \mapsto\,  _m\!\gamma$,
   where $_m\!\gamma$ is the edge-path at $m$ that takes the same word $w_\gamma$ as $\gamma$
    and is piecewise linear under the same subdivision of $[0,1]$.
 The constant edge-path $\gamma: [0,1]\rightarrow \Gamma$ such that $\gamma([0,1])=m\in M$
  will be denoted by the vertex $m$.
 
 There is an operation $\star$ on the set of edge-paths:
  \begin{itemize}
   \item[$\cdot$]
   For edge-paths $\gamma_1$ and $\gamma_2$ on $\Gamma$ such that $\gamma_1(1)=\gamma_2(0)$,
   define
   $$
    (\gamma_1\star\gamma_2)(t)\;
	    =\; \left\{
	          \begin{array}{ll}
			    \gamma_1(2t)\,,         & \mbox{for}\;\;  t \in [0,\, \frac{1}{2}]\,;  \\[1.2ex]
	            \gamma_2(2t-1)\,,     & \mbox{for}\;\;  t \in [\frac{1}{2},\, 1]\,.
	          \end{array}
             \right.
   $$
   
  \item[$\cdot$]
  For general edge-paths $\gamma_1$ and $\gamma_2$ on $\Gamma$, define
    $$
	  \gamma_1\star\gamma_2\; :=\; \gamma_1\star T_{\gamma_1(1)}(\gamma_2)\,.
	$$
 \end{itemize}
 In particular,
   $m\star \gamma = T_m(\gamma)$,
   $\gamma(0)\star\gamma = \gamma$,    and
   $\gamma\star m = \gamma$,
  after a subdivision of $[0,1]$ and a simplicial map on the subdivided $[0,1]$ that fixes $\{0,1\}$,
      for all vertex $m$ and edge-path $\gamma$.

 Consider the space $\breve{M}$ of homotopy classes of edge-paths at the vertex $e_M$ of $\Gamma$
  relative to the boundary $\{0,1\}$ of the interval $[0,1]$.
 The operation $\star$ on edge-paths on $\Gamma$ descends to an associative binary operation, still denoted by $\star$,
   on $\breve{M}$ with the identity the constant path $e_M$.
 The monoid $(\breve{M}, \star, e_M)$ will be called the {\it edge-path monoid} associated to $(M, S)$.
}\end{sdefinition}

\medskip

\begin{sdefinition} {\bf [monoid algebra]}\; {\rm
  Let $(M, \star, e_M)$ be a monoid.
  The ${\Bbb C}$-vector space\\     $\bigoplus_{m\in M} {\Bbb C}\cdot m$
    with the binary operation $\ast$ generated by ${\Bbb C}$-linear expansion of
     $$
	   (c_1m_1)\ast (c_2m_2) \; =\; (c_1c_2)\cdot (m_1\star m_2)
	 $$
  is called the {\it monoid algebra} over ${\Bbb C}$ associated to $M$,
  in notation ${\Bbb C}\langle M\rangle$ and the identity $1\cdot e_M=: 1$.
 When $M$ is commutative, we will denote ${\Bbb C}\langle M\rangle$ by ${\Bbb C}[M]$.
}\end{sdefinition}

\medskip

\begin{slemma}{\bf [natural homomorphism]}\;
 There is a built-in monoid epimorphism\\ $\pi: \breve{M}\rightarrow M$,
  which extends to an epimorphism
    $\pi: {\Bbb C}\langle \breve{M}\rangle\rightarrow {\Bbb C}\langle M\rangle$
    of monoid algebras over ${\Bbb C}$.
\end{slemma}

\medskip

\begin{proof}
 This follows from the surjective map
   $\pi: \{\mbox{edge-paths $\gamma$ at $e_M$}\}   \rightarrow M$,
   with $\gamma \mapsto \gamma(1)$.

\end{proof}

\medskip

\begin{sexample} {\bf [lattice of rank $n$]}\; {\rm
 Let $M$ be the lattice of rank $n$\, (i.e.\ $\simeq {\Bbb Z}^n$ as a ${\Bbb Z}$-module).
 As a commutative monoid written multiplicatively, $M\simeq \times_{i=1}^n z_i^{\Bbb Z}$ and
  generated by the inverse-closed subset $S= \{z_1, z_1^{-1}, \,\cdots,\, z_n, z_n^{-1}\}$.
 The associated ${\Bbb C}$-algebra is a polynomial ring
   $$
    {\Bbb C}[M]\;=\;
     {\Bbb C}[z_1, z^{-1},\, \cdots,\, z_n, z_n^{-1}]\;
        \simeq\; \frac{{\Bbb C}[z_1, z_1^\prime,\, \cdots,\, z_n, z_n^\prime]}
		                   {(z_1z_1^\prime-1,\,\cdots,\, z_nz_n^\prime-1)}\,.
   $$
 The Cayley diagram $\Gamma$ of $(M,S)$ is a homogeneous graph of constant valence $2n$ for all vertices,
   with the base vertex $1\in M$.
 Under the above presentation of $M$,
  the edge-path monoid $\breve{M}$  is
    $\langle z_1, z_1^{-1},\,\cdots,\, z_n, z_n^{-1} \rangle$
	(the noncommutative monoid generated by $z_1, z_1^{-1},\,\cdots\,, z_n, z_n^{-1}$
	   subject to relations $z_iz_i^{-1}= z_i^{-1}z_i=1$, $i=1,\,\ldots\,, n$         and
  the associated monoid algebra
   $$
    {\Bbb C}\langle\breve{M}\rangle\; =\;
	  {\Bbb C}\langle z_1,z_1^{-1},\,\cdots\,, z_n, z_n^{-1} \rangle
	     \simeq  \frac{{\Bbb C}\langle z_1, z_1^\prime, \,\cdots\,, z_n, z_n^\prime\rangle}
		                  {(z_1z_1^\prime-1, z_1^\prime z_1-1,\,\cdots\,,\, z_nz_n^\prime-1, z_n^\prime z_n-1)}\,.
   $$
 Here,
  ${\Bbb C}\langle z_1, z_1^\prime, \,\cdots\,, z_n, z_n^\prime\rangle$
     is the free associative ${\Bbb C}$-algebra with $2n$ generators and\\
  $(z_1z_1^\prime-1, z_1^\prime z_1-1,\,\cdots\,,\, z_nz_n^\prime-1, z_n^\prime z_n-1)$
    is the two-sided ideal generated by the elements indicated.
 The monoid homomorphism $\pi: \breve{M}\rightarrow M$ from the terminal-vertex-of-edge-path-at-$1$ map    and
    the associated monoid-algebra homomorphism
        $\pi: {\Bbb C}\langle\breve{M}\rangle \rightarrow {\Bbb C}[M]$
   coincide with the quotient by the commutator ideal $[\breve{M}, \breve{M}]$.
 Cf.\ {\sc Figure}~1-1.
    %
     	
\begin{figure}[htbp]
 \bigskip
  \centering
  \includegraphics[width=0.80\textwidth]{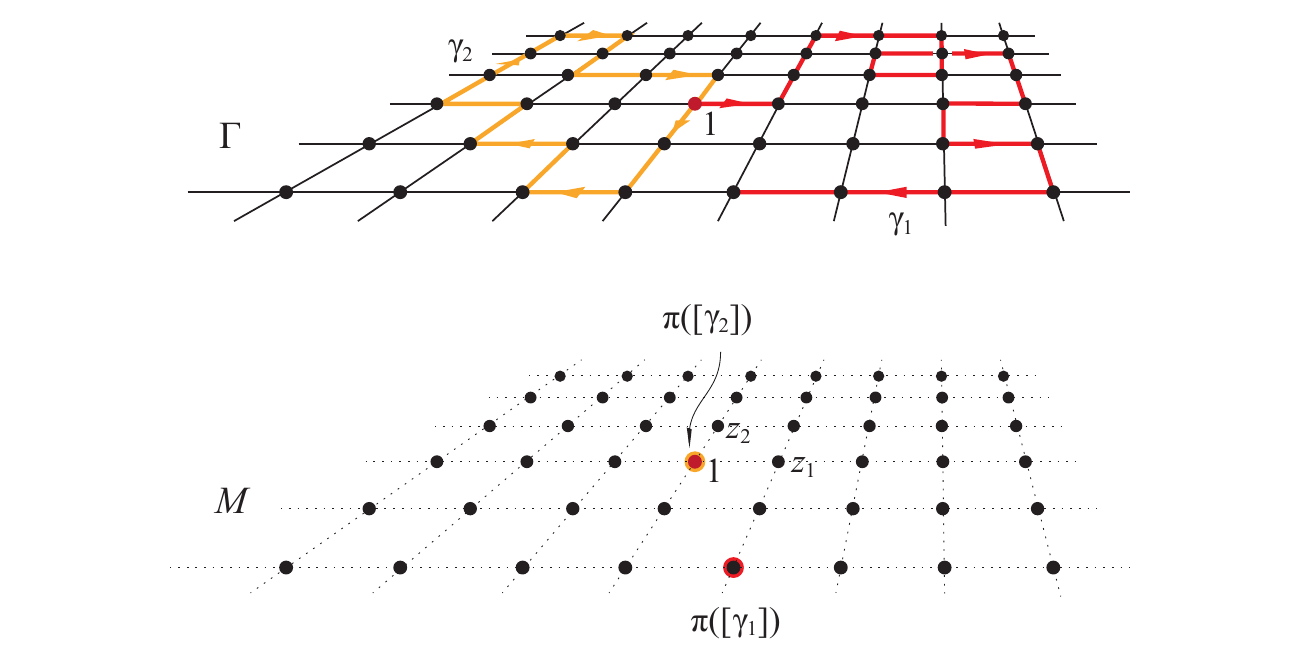}
  
  \bigskip
 \centerline{\parbox{13cm}{\small\baselineskip 12pt
  {\sc Figure}~1-1.
  The edge-path monoid $\breve{M}$ associated to a lattice $M$ (at the bottom) with generators indicated by the dashed grid.
  The rank $2$ case is illustrated here.
  All the edges of the Cayley graph $\Gamma$ (at the top) of $M$-with-generators are bi-directed.
  Examples of two elements of $\breve{M}$,
     represented by edge-paths
     $\gamma_1
	   =z_1z_2^3z_1^2z_2^{-2}z_1^{-1}z_2z_1^2z_2^{-2}z_1^{-1}z_2^{-1}z_1z_2^{-1}z_1^{-3}$
	   and
	 $\gamma_2
	  =z_2^{-2}z_1^{-1}z_2z_1^{-1}z_2z_1^{-1}z_2^3z_1z_2^{-2}z_1^2z_2^{-1}$
     at the identity $1\in\,\mbox{Vertex}(\Gamma)$,     and
   their image under the built-in monoid homomorphism $\pi:\breve{M}\rightarrow M$
   are indicated: $\pi([\gamma_1])=z_1z_2^{-2}$ and $\pi([\gamma_2])=1$.
  }}
\end{figure}	
	
 The complex torus $\mathbf{T}:={\Bbb T}^n := ({\Bbb C}^\times)^n$ action on the generators
  $$
    (z_1,\,\cdots\,\, z_n)\;\stackrel{g_\mathbf{t}}\longmapsto\; (t_1z_1,\,\cdots\,,\, t_nz_n)\,,
      \hspace{2em}\mathbf{t}=(t_1,\,\cdots\,, t_n)\,\in\,\mathbf{T}\,,	
  $$
 induces a ${\Bbb T}^n$-action on ${\Bbb C}\cdot M$ and ${\Bbb C}\cdot\breve{M}$,
   leaving each ${\Bbb C}$-factor of ${\Bbb C}\cdot M$ and ${\Bbb C}\cdot \breve{M}$ invariant,
 that render $\pi: {\Bbb C}\cdot\breve{M}\rightarrow {\Bbb C}\cdot M$ ${\Bbb T}^n$-equivariant.
 This extends to a ${\Bbb T}^n$-action on ${\Bbb C}\langle\breve{M}\rangle$ and ${\Bbb C}[M]$
  under which $\pi: {\Bbb C}\langle\breve{M}\rangle\rightarrow {\Bbb C}[M]$ is ${\Bbb T}^n$-equivariant.   	
}\end{sexample}

\bigskip

The lattice $M$  of rank $n$ in Example~1.7 is the starting point of constructions of the current work.
As a commutative monoid, there are occasions when it is more convenient for $M$ to be written additively,
   rather than multiplicatively.
We shall switch between these two pictures freely and use whichever suits better.

\bigskip
\bigskip

\section{Soft noncommutative toric schemes}

Soft noncommutative toric schemes are introduced in this section.
The setting and the terminology used indicate our focus on the function ring, rather than the point-set and topology,
  of a noncommutative space
 while striving to retain enough underlying geometric picture for the purpose of studying
 (dynamical) D-branes on a noncommutative target-space in string theory
 within the realm of Algebraic Geometry.
(Cf.~[L-Y1] (D(1)), [L-L-S-Y] (D(2)), and Sec.\ 4.2 of the current work.)

\bigskip

\subsection{The noncommutative affine $n$-space $\nc\mathbf{A}^n_{\Bbb C}$
        and its $0$-dimensional subschemes}

\begin{definition} {\bf [affine scheme, function ring, noncommutative affine $n$-space over ${\Bbb C}$]}\; {\rm
 Let $\mbox{\bf Ring}$ be the category of rings.
 An object $X$ in the opposite category $\mbox{\bf Ring}^\scriptsizeop$ of $\mbox{\bf Ring}$
  is called an {\it affine scheme}.
 The ring that underlies $X$ is called the {\it function ring} of $X$.
 If the function ring of $X$ is noncommutative (resp.\ commutative), then
  we say that $X$ is a {\it noncommutative affine scheme} (resp.\ {\it commutative affine scheme}).
 In particular,
 we define a {\it noncommutative affine $n$-space} over ${\Bbb C}$
   to be an object in $\mbox{\bf Ring}^\scriptsizeop$
   that corresponds to a free noncommutative associative ${\Bbb C}$-algebra generated by $n$ letters,
   in notation ${\Bbb C}\langle z_1,\,\cdots\,, z_n\rangle$.
 Since any two free noncommutative associative ${\Bbb C}$-algebras generated by $n$ letters are isomorphic
    for a fixed $n$,
  we shall denote a {\it noncommutative affine $n$-space} over ${\Bbb C}$ commonly by
    $\nc\mathbf{A}^n_{\Bbb C}$ or simply $\nc\mathbf{A}^n$.
 $\nc\mathbf{A}^n$ is {\it smooth} in the sense of the following extension/lifting property:
  \begin{itemize}
   \item[$\cdot$]
    For any
	 ${\Bbb C}$-algebra homomorphism ${\Bbb C}\langle z_1,\,\cdots\,,\,z_n\rangle\rightarrow A$  and
	  any ${\Bbb C}$-algebra epimorphism $B\rightaarrow A$,
    there exists a ${\Bbb C}$-algebra homomorphism
	  ${\Bbb C}\langle z_1,\,\cdots\,,\,z_n\rangle\rightarrow B$ that makes the following diagram commute
      $$
       \xymatrix{
	    & && B\ar@{->>}[d]  \\
	    &{\Bbb C}\langle z_1,\,\cdots\,,\,z_n\rangle \ar[rru]\ar[rr]    && A  &.
	   }
      $$	
  \end{itemize}
}\end{definition}

\medskip

\begin{definition} {\bf [master commutative subscheme of $\nc\mathbf{A}^n$]}\; {\rm
 The natural quotient
  $$
   {\Bbb C}\langle z_1,\,\cdots\,,\, z_n\rangle\;
    \longrightaarrow\;  {\Bbb C}[z_1,\,\cdots\,,\, z_n]\,,\:\:\:\:   z_i\;\longmapsto\; z_i
  $$
  of ${\Bbb C}$-algebras, with kernel the two-sided ideal $(z_iz_j-z_jz_i\,|\, 1\le i< j\le n)$,
 defines an embedding $\mathbf{A}^n\hookrightarrow \nc\mathbf{A}^n$,
 called the {\it master commutative subscheme} of $\nc\mathbf{A}^n$.
}\end{definition}

\medskip

\begin{definition} {\bf [closed subscheme of $\nc\mathbf{A}^n$]}\; {\rm
 A {\it closed subscheme} of $\nc\mathbf{A}^n$
  is a ${\Bbb C}$-algebra quotient ${\Bbb C}\langle z_1,\,\cdots\,,\,z_n\rangle \longrightaarrow A$.
 When $A$ is commutative, then the closed subscheme is called a {\it commutative closed subscheme}
  of $\nc\mathbf{A}^n$.
 Since when $A$ is commutative,  such a quotient always factors as
  $$
   \xymatrix{
	&{\Bbb C}\langle z_1, \,\cdots\,,\, z_n\rangle   \ar@{->>}[rr]   \ar@{->>}[rd] && A \\
	&& {\Bbb C}[z_1,\,\cdots\,,\, z_n] \ar@{->>}[ru]   &&,
	}
  $$
 a commutative closed subscheme of $\nc\mathbf{A}^n$
  is always contained in the master commutative subscheme
  $\mathbf{A}^n\subset\nc\mathbf{A}^n$.
 When $A$ is a finite-dimensional ${\Bbb C}$-algebra,
  the closed subscheme is called a {\it $0$-dimensional subscheme} of $\nc\mathbf{A}^n$.
   %
   %
 In particular,
 a {\it ${\Bbb C}$-point} on $\nc\mathbf{A}^n$ is a ${\Bbb C}$-algebra quotient
  ${\Bbb C}\langle z_1,\,\cdots\,,\,z_n\rangle \longrightaarrow {\Bbb C}$.
Since ${\Bbb C}$ is commutative,
  the set of ${\Bbb C}$-points of $\nc\mathbf{A}^n$ coincide naturally with the set of
  ${\Bbb C}$-points of $\mathbf{A}^n$
  via the built-in embedding $\mathbf{A}^n\hookrightarrow \nc\mathbf{A}^n$.
}\end{definition}

\medskip

\begin{definition} {\bf [punctual versus nonpunctual $0$-dimensional subscheme]}\; {\rm
  A $0$-dimen- sional subscheme
   ${\Bbb C}\langle z_1,\,\cdots\,,\,z_n\rangle \rightaarrow A$ of $\nc\mathbf{A}^n$ is called {\it punctual}
   if $A$ admits a ${\Bbb C}$-algebra quotient
   $$
     A\;\longrightaarrow\;  \Center(A)/\Nil_{\scriptsizeCenter(A)}\,.
   $$
 Here,
  $\Center(A)\subset A$ is the {\it center} of $A$ and $\Nil_{\scriptsizeCenter(A)}$
  is the ideal of nilpotent elements of $\Center(A)$.
 Else, the $0$-dimensional subscheme is called {\it nonpunctual} (i.e.\ ``fuzzy without core").
}\end{definition}

\medskip

\begin{example}
{\bf [Grassmann points, Azumaya/matrix points, other special points on $\nc\mathbf{A}^n$]}\; {\rm
 (1)
  All commutative $0$-dimensional subschemes of $\nc\mathbf{A}^n$ are punctual.
  
 (2)
  A {\it Grassmann point} (resp.\ {\it Azumaya} or {\it matrix point}) on $\nc\mathbf{A}^n$
    is a ${\Bbb C}$-algebra quotient ${\Bbb C}\langle z_1,\,\cdots\,,\,z_n\rangle  \rightaarrow A$
    such that $A$ is isomorphic to a Grassmann algebra over ${\Bbb C}$
    $$
      {\Bbb C}[\theta_1,\,\cdots\,,\theta_r]^\anticommuting\;
			  :=\; {\Bbb C}\langle \theta_1,\,\cdots\,,\,\theta_r\rangle/
			              (\theta_\alpha\theta_\beta+\theta_\beta\theta_\alpha  \,|\, 1\le \alpha, \beta\le r)
    $$
     for some $r$
    (resp.~a matrix algebra $M_r({\Bbb C})$ over $\Bbb C$ for some rank $r\ge 2$).
    Both are noncommutative points on $\nc\mathbf{A}^n$.
	However, {\it Grassmann points are punctual while Azumaya/matrix points are nonpunctual}
	  since there exists no ${\Bbb C}$-algebra homomorphism from $M_r({\Bbb C})$ to ${\Bbb C}$
	  for $r\ge 2$.
  		
 (3)
   Other special points on $\nc\mathbf{A}^n$ include points
     ${\Bbb C}\langle z_1,\,\cdots\,,\, z_n\rangle  \rightaarrow A$
     where $A$ is isomorphic to the ${\Bbb C}$-algebra $T^{\scriptsizeupper}_r({\Bbb C})$
	 of upper triangular matrix for some rank $r\ge 2$
     (resp.\  the ${\Bbb C}$-algebra $T^{\scriptsizelower}_r({\Bbb C})$
	  of lower triangular matrix for some rank $r\ge 2$.)
   They are called {\it upper-triangular matrix points} (resp.\ {\it lower-triangular matrix points}).
   Unlike Azumaya/matrix points on $\nc\mathbf{A}^n$, both are punctual
     under the correspondence that sends an upper-or-lower triangular matrix to its diagonal.
}\end{example}

\medskip

\begin{definition} {\bf [nested master $l$-commutative closed subscheme of $\nc\mathbf{A}^n$]}\; {\rm
 Given two two-sided ideals $I,\, J$ in ${\Bbb C}\langle z_1,\,\cdots\,,\, z_n\rangle$.
 Denote by $[I,J]$ the two-sided ideal in   ${\Bbb C}\langle z_1,\,\cdots\,,\, z_n\rangle$
  generated by elements of the form $[f,g]:= fg-gf$, where $f\in I$ and $g\in J$.
 Now let $\frak{m}_0$ be the two-sided ideal $(z_1,\,\cdots\,,\,z_n)$ of
  ${\Bbb C}\langle z_1,\,\cdots\,,\, z_n\rangle$.
 Then note that $[\frak{m}_0, \frak{m}_0]$ and $(z_iz_j-z_jz_i\,|\, 1\le i<j\le n )$  coincide.
 Consider the closed scheme $\mathbf{A}^n_{(l)}$ of $\nc\mathbf{A}^n$,
   $l=1, 2, \cdots$,
  associated the ${\Bbb C}$-algebra quotient
  $$
   {\Bbb C}\langle z_1,\,\cdots\,,\, z_n\rangle\;
      \longrightaarrow\;   {\Bbb C}\langle z_1,\,\cdots\,,\, z_n\rangle/[\frak{m}_0, \frak{m}_0^l]\,.
  $$
 Since
   $[\frak{m}_0, \frak{m}_0] \subset [\frak{m}_0, \frak{m}_0^2]
       \subset [\frak{m}_0, \frak{m}_0^3] \subset \cdots$,
  one has a nested sequence of closed subschemes in $\nc\mathbf{A}^n$
   $$
    \mathbf{A}^n=\mathbf{A}^n_{(1)}\; \hookrightarrow\; \mathbf{A}^n_{(2)} \;
	  \hookrightarrow\; \mathbf{A}^n_{(3)}\; \hookrightarrow\; \cdots\,.
   $$
 Any function on $\nc\mathbf{A}^n$ of the form $f_1\cdots f_l$, where $f_i\in \frak{m}_0$ for all $i$,
  becomes commutative when restricted to $\mathbf{A}^n_{(l^\prime)}$ for $l^\prime \le l$.
 Furthermore, any closed subscheme of $\nc\mathbf{A}^n$ with this property must be a closed subscheme of
   $\mathbf{A}^n_{(l)}$.
 Thus, we name $\mathbf{A}^n_{(l)}$
   the {\it master $l$-commutative subscheme} of $\nc\mathbf{A}^n$
   (with respect to the choice of coordinate functions $z_1,\,\cdots\,,\,z_n$ on $\nc\mathbf{A}^n$).
 The restriction of any function $f$ on $\nc\mathbf{A}^n$ of degree $d\ge 1$
   to any $\mathbf{A}^n_{(l)}$ with $l\le d$ is commutative on $\mathbf{A}^n_{(l)}$. 	
}\end{definition}

\medskip

\begin{example} {\bf [$0$-dimensional subscheme and $\mathbf{A}^n_{(l)}$]}\;  {\rm
 How a $0$-dimensional subscheme in $\nc\mathbf{A}^n$ intersects the nested sequence
   $\mathbf{A}^n_{(l)}$, $l=1,2,\ldots$, gives one a sense of a {\it depth of noncommutativity}
   the $0$-dimensional subscheme of $\nc\mathbf{A}^n$ is located at.\:\:
  (1)
   A commutative $0$-dimensional subscheme of $\nc\mathbf{A}^n$ lies in
            $\mathbf{A}^n=\mathbf{A}_{(1)}$ and hence in $\mathbf{A}^n_{(l)}$ for all $l$.

  (2)
   A punctual $0$-dimensional subscheme of $\nc\mathbf{A}^n$ contains a commutative subscheme
          and hence must have nonempty intersection with $\mathbf{A}^n$
		  and hence has nonempty intersection with $\mathbf{A}^n_{(l)}$ for all $l$.
   They thus lie in an infinitesimal neighborhood of $\mathbf{A}^n$	in $\nc\mathbf{A}^n$.	
		
  (3)
   Consider the Azumaya/matrix point on $\nc\mathbf{A}^{r^2}$
    $$
    {\Bbb C}\langle z_{ij}\;|\; 1\le i, j\le r \rangle\; \longrightaarrow\; M_r({\Bbb C})\,,\hspace{2em}
	 z_{ij}\;\longmapsto\; e_{ij}\,,
    $$			
    where $e_{ij}$ is the $r\times r$ matrix with $ij$-entry $1$ and elsewhere $0$.
   The kernel of this ${\Bbb C}$-algebra epimorphism is the two-sided ideal $I$ generated by
    $$
	   z_{ij}z_{i^\prime j^\prime}- \delta_{ji^\prime}z_{ij^\prime}\,,\;1\le i,j,i^\prime,j^\prime \le r\,,\;
	   \mbox{where $\delta_{\tinybullet}$ is the Kronecker delta.}
	$$
   In particular,
	 $z_{ij}- z_{ij^\prime_1}z_{j^\prime_1j^\prime_2}\cdots z_{j^\prime_{l-1}j^\prime_l}z_{j^\prime_lj}
	  \in I$
      for all $l=1,2,\ldots$.
   Recall the two-sided ideal $\frak{m}_0=(z_{ij}\,|\,1\le i,j\le r)$
     of ${\Bbb C}\langle z_{ij}\;|\; 1\le i, j\le r \rangle$.
   Then the two-sided ideal generated by $I+[\frak{m_0}, \frak{m}^l]$
         in $ {\Bbb C}\langle z_{ij}\;|\; 1\le i, j\le r \rangle$
    is the whole ${\Bbb C}\langle z_{ij}\;|\; 1\le i, j\le r \rangle$.	
	This shows that
	 this Azumaya point on $\nc\mathbf{A}^{r^2}$ has no intersection with $\mathbf{A}^{r^2}_{(l)}$
	  for all $l$.
	In a sense, this Azumaya point is located deep in the noncommutative part of $\nc\mathbf{A}^n$.	
}\end{example}

\medskip

\begin{definition} {\bf [connectivity and hidden disconnectivity of $0$-dimensional subscheme]}\; {\rm
 Given a $0$-dimensional subscheme
  ${\Bbb C}\langle z_1,\,\cdots\,, z_n\rangle\rightaarrow A$
  on $\nc\mathbf{A}^n$.
 We say that the $0$-dimensional subscheme is {\it connected}
    if $A$ is not a product $A_1\times A_2$ of two ${\Bbb C}$-algebras.
 Else, we say that the $0$-dimensional subscheme is {\it disconnected}.
 If $A$ is not a product, yet it admits a decomposition of $1$ by non-zero orthogonal idempotents
    $$
       1\;=\; e_1+\cdots+e_k\,,
    $$	
   for some $k$, where $e_i\ne 0$ and $e_i^2=e_i$ for all $i$ and $e_ie_j=e_je_i=0$ for all $i\ne j$,
 then we say that the connected $0$-dimensional subscheme has {\it hidden disconnectivity}.
}\end{definition}

\bigskip

\noindent
The phenomenon of hidden disconnectivity can occur for noncommutative points
   --- for example, Azumaya/matrix points, upper-triangular matrix points, lower-triangular matrix points ---
  on $\nc\mathbf{A}^n$ whether they are punctual or not.
(Cf.\ [L-Y1] (D(1)) for hidden disconnectivity of Azumaya schemes via the notion of surrogates.)

\bigskip

Any ${\Bbb C}$-algebra homomorphism
   $$
    u^\sharp\;:\; {\Bbb C}[t]\; \longrightarrow\; {\Bbb C}\langle z_1,\,\cdots\,,\, z_n\rangle
   $$
   with $u^\sharp(t)\not\in {\Bbb C}$ is a monomorphism and, by definition/tautology,
  induces a dominant morphism
   $$
     u\;:\; \nc\mathbf{A}^n\; \longrightarrow\; \mathbf{A}^1\,.
   $$
In this way,
 a $0$-dimensional subscheme
   ${\Bbb C}\langle z_1,\,\cdots\,,\,z_n\rangle \rightaarrow A$ on $\nc\mathbf{A}^n$
 is mapped to a $0$-dimensional, now commutative, subscheme on the affine line $\mathbf{A}^1$.
One may use such projections to get a sense as to how the $0$-dimensional subscheme sits in $\nc\mathbf{A}^n$
 and some of its properties.
(Cf.~Radom transformation in analysis.)
 
\bigskip

\begin{example}
{\bf [hidden disconnectivity of Azumaya/matrix point on $\nc\mathbf{A}^n$
          manifested via projection to $\mathbf{A}^1$]}\; {\rm
 Recall the matrix point on $\nc\mathbf{A}^{r^2}$ in Example~2.1.7 (3).
 Consider
  the projection $u:\nc\mathbf{A}^{r^2}\rightarrow \mathbf{A}^1$
  specified by the ${\Bbb C}$-algebra monomorphism
  $$
   u^\sharp\;:\; {\Bbb C}[t]\; \longrightarrow\; {\Bbb C}\langle z_{ij}\,|\, 1\le i, j\le r\rangle,\,\hspace{1em}
     t\;\longmapsto\; z_{i^\prime i^\prime}\:\:\mbox{for some $i^\prime$;}
  $$
 and let $\hat{u}^\sharp$ be a ${\Bbb C}$-algebra homomorphism from the composition
    $$
	 \xymatrix{
	  {\Bbb C}[t] \ar[r]_-{u^\sharp} \ar@/^1.4pc/[rr]^-{\hat{u}^\sharp}
	       & {\Bbb C}\langle z_{ij}\,|\, 1\le i, j\le r\rangle \ar[r]
		   &  M_r({\Bbb C})
	 }
	$$
    induced by the given matrix point on $\nc\mathbf{A}^{r^2}$.
  Then the matrix point, though connected, is projected under $u$ to two distinct closed points $\{p=0\,,\, p=1\}$
    on $\mathbf{A}^1$,
    described by the ideal $(t(t-1))$,
      	i.e.\ the kernel $\Ker(\hat{u}^\sharp)$ of $\hat{u}^\sharp$, of ${\Bbb C}[t]$.
  In this way, $u$ reveals the hidden disconnectivity of the given matrix point on $\nc\mathbf{A}^{r^2}$.
  As a (left) ${\Bbb C}[t]$-module through $\hat{u}^\sharp$,
   $M_r({\Bbb C})$ is pushed forward under $u$ to
   a torsion ${\cal O}_{\mathbf{A}^1}$-module on $\mathbf{A}^1$,
   with fiber-dimension $r^2-r$ at $p=0$ and $r$ at $p=1$.
  
 For reference, if one instead considers the projection
   $v:\nc\mathbf{A}^{r^2}\rightarrow \mathbf{A}^1$
  specified by the ${\Bbb C}$-algebra monomorphism
  $v^\sharp:{\Bbb C}[t]\rightarrow {\Bbb C}\langle z_{ij}\,|\, 1\le i, j\le r\rangle$
    with $t\mapsto z_{i^\prime j^\prime}$ for some $i^\prime\ne j^\prime$.
 Then the given matrix point is projected to a nonreduced point on $\mathbf{A}^1$
  associated to the ideal $(t^2)$ of ${\Bbb C}[t]$.
}\end{example}

%
%
%
%

\bigskip

\begin{flushleft}
{\bf $\nc\mathbf{A}^n$ as a smooth noncommutative affine toric scheme}
\end{flushleft}
Let
 $N$ be a lattice isomorphic to ${\Bbb Z}^n$, with a fixed basis $(e_1,\,\cdots\,,\, e_n)$,   and
 $M = \Hom_{\Bbb Z}(N,{\Bbb Z})$ be the dual lattice,
   with the evaluation pairing denoted by $\langle \,,\, \rangle: M\times N \rightarrow {\Bbb Z}$ and
    the dual basis $(e_1^\ast,\,\cdots\,,\, e_n^\ast)$.
Denote the associated ${\Bbb R}$-vector spaces
 $N_{\Bbb R}:= N\otimes_{\Bbb Z}{\Bbb R}$ and $M_{\Bbb R}:= M\otimes_{\Bbb Z}{\Bbb R}$.

Let $\sigma$ be the cone in $N_{\Bbb R}$ generated by $e_1,\,\cdots\,e_n$.
Then
   the commutative monoid $M_\sigma:= \sigma^\vee\cap M$ is generated by $\{e_1^\ast,\,\cdots\,,\,e_n^\ast\}$.
Let $\breve{M}$  be the  associated noncommutative monoid from edge-paths at $0$, cf.\ Definition~1.4.
Then the correspondence $e_i^\ast\mapsto z_i$, $i=1,\,\ldots\,, n$,
 specifies a ${\Bbb C}$ -algebra isomorphism
 from the monoid algebra ${\Bbb C}\langle \breve{M} _\sigma \rangle$
 to the function ring ${\Bbb C}\langle z_1,\,\cdots\,,\, z_n \rangle$ of $\nc\mathbf{A}^n$.
Under the built-in monoid homomorphism $\pi: \breve{M}_\sigma\rightarrow M_\sigma$,
 this isomorphism descends to
 an isomorphism ${\Bbb C}[M_\sigma]\rightarrow {\Bbb C}[z_1,\,\cdots\,,\, z_n]$
 that gives the standard realization of $\mathbf{A}^n$ as a toric variety.
Furthermore,
since the multiplicative group ${\Bbb C}^\times:={\Bbb C}-\{0\}$ in ${\Bbb C}$
  lies in the center of ${\Bbb C}\langle \breve{M}_\sigma \rangle$,
 the $({\Bbb C}^\times)^n$-action on ${\Bbb C}[M_\sigma]$ generated by
  $(z_1,\,\cdots,\,z_n)\mapsto (t_1z_1,\,\cdots\,,\,t_nz_n)$
   for $(t_1,\,\cdots\,,\, t_n)\in ({\Bbb C}^\times)^n$
  lifts canonically to a $({\Bbb C}^\times)^n$-action on ${\Bbb C}\langle \breve{M}_\sigma \rangle$
 that makes $\pi$ $({\Bbb C}^\times)^n$-equivariant.
  
\bigskip

\begin{definition} {\bf [$\nc\mathbf{A}^n$ as noncommutative affine toric scheme]}\; {\rm
 We shall call the above construction
 the realization of $\nc\mathbf{A}^n$ as a smooth {\it noncommutative affine toric scheme} over ${\Bbb C}$,
 associated to the cone $\sigma\subset N_{\Bbb R}$.
}\end{definition}
 
\bigskip

\noindent
By construction,
the restriction of the toric scheme structure on $\nc\mathbf{A}^n$ to the master commutative subscheme
 $\mathbf{A}^n\hookrightarrow \nc\mathbf{A}^n$
 recovers the ordinary realization of $\mathbf{A}^n$ as a toric variety.

\bigskip

\subsection{Soft noncommutative toric schemes associated to a fan}

Recall
 the lattice $N\simeq {\Bbb Z}^n$, with a fixed basis $(e_1,\,\cdots\,,\, e_n)$,
 the dual lattice $M = \Hom_{\Bbb Z}(N,{\Bbb Z})$,
 the evaluation pairing denoted by $\langle \,,\, \rangle: M\times N \rightarrow {\Bbb Z}$,
 the dual basis $(e_1^\ast,\,\cdots\,,\, e_n^\ast)$ of $M$, and
 the associated ${\Bbb R}$-vector spaces  $N_{\Bbb R}:= N\otimes_{\Bbb Z}{\Bbb R}$
    and $M_{\Bbb R}:= M\otimes_{\Bbb Z}{\Bbb R}$.

Let $\Delta$ be a {\it fan} in $N$.
(That is, $\Delta$ is a set of {\it rational strongly convex polyhedral cones} $\sigma$ in $N_{\Bbb R}$
    such that (1) Each {\it face} of a cone in $\Delta$ is also a cone in $\Delta$;
	  (2) The intersection of two cones in $\Delta$ is a face of each; cf. [Fu: Sec.\ 1.4].)
For $\sigma, \tau\in\Delta$,
 denote  $\tau\preccurlyeq\sigma$ or $\sigma\succcurlyeq\tau$  if $\tau$ is a face of $\sigma$; and
 denote  $\tau\prec\sigma$ or $\sigma\succ\tau$  if $\tau$ is a face of $\sigma$ and $\tau\ne \sigma$.
	
\bigskip

\begin{definition} {\bf [$\Delta$-system, inverse $\Delta$-system of objects in a category]}\; {\rm
 Let ${\cal C}$ be a category,
   with the set of object denoted $\Object{\cal C}$ and
          the set of morphisms of objects denoted $\Morphism{\cal C}$.
 Then a {\it $\Delta$-system}  (resp.\ {\it inverse $\Delta$-system}) {\it of objects in ${\cal C}$}
  is a correspondence $F: \Delta \rightarrow \ObjectCategory{\cal C}$ and a choice of morphisms
  $h_{\tau\sigma}:F(\tau)\rightarrow F(\sigma)$
  (resp.\ $h_{\sigma\tau}: F(\sigma)\rightarrow F(\tau)$)
   for all pairs of $(\sigma, \tau) \in\Delta\times \Delta$ with $\tau\prec\sigma$
  such that
   $h_{\tau\sigma}\circ h_{\rho\tau}= h_{\rho\sigma}$
   (resp.\ $h_{\tau\rho}\circ h_{\sigma\tau}=h_{\sigma\rho}$)
   for all $\rho\prec\tau\prec\sigma$.
}\end{definition}

\bigskip
	  	 	
Recall how a variety $Y(\Delta)$ can be associated to $\Delta$, (e.g.\ [Fu])
 as an inverse $\Delta$-system of monoid algebras or a $\Delta$-system of affine schemes:
Associated to each cone $\sigma\in\Delta$ is a commutative monoid
 $$
   M_\sigma\; :=\; \sigma^\vee\cap M\;
    =\; \{u\in M\,|\, \langle u, v\rangle \ge 0\:\:\mbox{for all $v\in\sigma$}  \}\,,.
 $$
By construction,
 there is a built-in monoid inclusion  $M_\sigma \hookrightarrow M_\tau$ for $\tau\prec\sigma$.
This gives rise to a $\Delta$-system of affine schemes
 $$
   \{\,U_\sigma := \Spec({\Bbb C}[M_\sigma])\,\}_{\sigma\in\Delta}\,,
 $$
 where ${\Bbb C}[M_\sigma]$ is the monoid algebra over ${\Bbb C}$ determined by $M_\sigma$,
that glue to a {\it toric variety} $Y(\Delta)$
through the system of inclusions of Zariski open sets
 $$
   \iota_{\tau\sigma}\; :\; U_\tau\; \hookrightarrow \; U_\sigma\,,
     \hspace{2em}\mbox{for all $\tau\prec\sigma\in\Delta$}\,,
 $$
 induced from the monoid inclusion $M_\sigma\hookrightarrow M_\tau$,
 that by construction satisfy the gluing conditions:
  $$
    \iota_{\tau\sigma} \circ \iota_{\rho\tau}\;=\; \iota_{\rho\sigma}\,,\hspace{2em}
     \mbox{for all $\rho \prec \tau \prec \sigma\in\Delta$}\,.
  $$
The built-in $({\Bbb C}^\times)^n$-action on $U_\sigma$ and
 the built-in torus embedding
  $$
   {\Bbb T}^n\; :=\; {\Bbb T}^n_{\Bbb C}\; :=\;  \Spec({\Bbb C}[M])\;
      \hookrightarrow\;  U_\sigma
  $$
  for each $\sigma\in\Delta$ glues to
 a $({\Bbb C}^\times)^n$-action on $Y(\Delta)$ and
 a torus embedding ${\Bbb T}^n \hookrightarrow Y(\Delta)$.
	  		 	  		
In this section, we shall construct
 a class of noncommutative spaces, named {\it soft noncommutative toric schemes}, associated to $\Delta$
 by constructing an inverse $\Delta$-system $\{\breve{M}_\sigma\}_{\sigma\in\Delta}$ of submonoids in
 $\breve{M}:= \langle e_1^\ast, {e_1^\ast}^{-1},\,\cdots\,,\, e_n^\ast, {e_n^\ast}^{-1}  \rangle$
 via lifting the above construction to the edge-path monoids of commutative monoids with generators
    and propose {\it soft gluings}
	to bypass the generally unresolvable issue of localizations of noncommutative rings when trying to gluing.
  
\bigskip

\begin{assumption} $[$on fan $\Delta$$]$\; {\rm
 Recall Sec.~2.1, theme: {\sl $\nc\mathbf{A}^n$ as a smooth noncommutative affine toric scheme}.
 To ensure that we have at least one uncomplicated chart in the atlas to begin with
   for the noncommutative toric scheme to be constructed, we shall assume that
   \begin{itemize}
    \item[$\cdot$]
    {\it The cone generated by $e_1,\,\cdots\,,\, e_n$ is in $\Delta$.}
   \end{itemize}
 The corresponding chart $\simeq \nc\mathbf{A}^n$ is called the {\it reference toric chart},
 cf.~Definition~2.2.13.
 Furthermore, to ensure that we have enough good fundamental charts in the atlas before gluings,
  we shall assume in addition that
 \begin{itemize}
  \item[$\cdot$]
  {\it All the maximal cones in $\Delta$ are $n$-dimensional, simplicial, and of index $1$.}
 \end{itemize}			
In particular, $Y(\Delta)$ is smooth.
}\end{assumption}
 
\medskip

\begin{remark} $[$other noncommutative toric variety$\,]$\; {\rm
 Based on various aspects of toric varieties and the language used,
 there are other notions and constructions of ``{\it noncommutative toric varieties}";
  for example, the work [K-L-M-V1], [K-L-M-V2] of
  {\it Ludmil Katzarkov}, {\it Ernesto Lupercio}, {\it Laurent Meersseman}, and {\it Alberto Verjovsky}.
 It is not clear if there is any connection among them.
 See Introduction for our motivation from D-branes in string theory.
}\end{remark}

\bigskip

\begin{flushleft}
{\bf Inverse $\Delta$-systems of $\Delta$-admissible submonoids of $\breve{M}$}
\end{flushleft}
Recall from Example~1.7
 the edge-path monoid
  $\breve{M}:= \langle e_1^\ast, {e_1^\ast}^{-1},\,\cdots\,,\, e_n^\ast, {e_n^\ast}^{-1}  \rangle$
  (cf.\ slightly different notation here) and
 the built-in monoid epimorphisms
   $\pi:  \breve{M}\rightarrow M$,
   $\pi:{\Bbb C}\cdot\breve{M}\rightarrow {\Bbb C}\cdot M$     and
 monoid-algebra epimorphism $\pi:{\Bbb C}\langle\breve{M}\rangle\rightarrow {\Bbb C}[M]$
 via commutatization.
 
\bigskip

\begin{definition} {\bf [submonoid of $\breve{M}$ admissible to cone in $\Delta$]}\; {\rm
 Let $\sigma\in \Delta$.
 Then a submonoid $\breve{M}^\prime$ of $\breve{M}$ is called {\it admissible to $\sigma$}
   if it satisfies:
   (0) $\breve{M}$ is finitely generated.
   (1) The monoid epimorphism $\pi:\breve{M}\rightarrow M$ induces a monoid epimorphism
              $\breve{M}^\prime \rightarrow M_\sigma$ by restriction.
   (2) An $\breve{m}\in \breve{M}^\prime$ is invertible in $\breve{M}^\prime$
               if $\pi(\breve{m})$ is invertible in $M_\sigma$.
}\end{definition}

\medskip

\begin{theorem} {\bf [existence of inverse $\Delta$-system of admissible submonoids of $\breve{M}$]}\;
 There exists an inverse $\Delta$-system
  $\{\breve{M}_\sigma\}_{\sigma\in\Delta}$ of submonoids of $\breve{M}$
  such that $\breve{M}_\sigma$ is admissible to $\sigma$ for all $\sigma\in\Delta$.
 {\rm We shall call such a system an} inverse
   $\Delta$-system of admissible submonoids of $\breve{M}$.
\end{theorem}

\begin{proof}
 We proceed in three steps.
 
 \bigskip

 \noindent
{\it Step $(a):$}\: {\it Submonoids of $\breve{M}$ admissible to maximal cones in $\Delta$}

 \medskip

 \noindent
 Let $\sigma$ be a maximal cone in $\Delta$.
 By assumption, $\sigma$ is $n$-dimensional, simplicial, and of index $1$.
 Thus, $\sigma^\vee$ is $n$-dimensional, simplicial, of index $1$ in $M_{\Bbb R}$
   and the generators $u_1,\,\cdots\,,\, u_n$ of the monoid $\sigma^\vee\cap M$ generates $M$ as well.
 
 \bigskip

 \begin{lemma} {\bf [submonoid admissible to maximal cone]}\;
  Let $\breve{u}_i\in \pi^{-1}(u_i)\subset \breve{M}$, for $i=1,\,\ldots\,,\,n$.
  Then the submonoid
    $\langle \breve{u},\,\cdots\,,\,\breve{u}_n \rangle \subset \breve{M}$
   is monoid-isomorphic to the free associative monoid of $n$ letters $\langle z_1,\,\cdots\,,\,z_n \rangle$.
 \end{lemma}

 \medskip
 
 \noindent
{\it Proof.}\:
 Recall the basis $(e_1^\ast,\,\cdots\,,\, e_n^\ast)$ for $M$.
 Since $(u_1,\,\cdots\,,\, u_n)$ generates $M$ as well, up to a relabelling of indices one may assume that
  $u_i = \sum_{j=1}^n a_{ij}e_j^\ast$ with the coefficient $a_{ii}\ne 0$, for $i=1\,\ldots\,,\,n$.
 Since $(\breve{u}_1,\,\cdots\,,\, \breve{u}_n)$ projects to $(u_1,\,\cdots\,,\, u_n)$ under $\pi$,
  as a word in $2n$ letters $e_1^\ast, {e_1^\ast}^{-1},\,\cdots\,,\,e_n^\ast, {e_n^\ast}^{-1}$,
  $\breve{u}_i$ must contain a letter
   $z_i^\prime:=e_i^\ast$, if $a_{ii}>0$, or ${e_i^\ast}^{-1}$, if $a_{ii}<0$.
 The correspondences
  $\breve{u}_i\mapsto z_i^\prime \mapsto z_i$ induce monoid-isomorphisms
  $\langle \breve{u}_1,\,\cdots\,,\, \breve{u}_n \rangle  \stackrel{\sim}{\longrightarrow}
    \langle z_1^\prime,\,\cdots\,,\, z_n^\prime \rangle  \stackrel{\sim}{\longrightarrow}
    \langle z_1,\,\cdots\,,\, z_n\rangle$.

 \hspace{38.2em}$\Box$

 \bigskip
 
 By construction, $\langle \breve{u}_1,\,\cdots\,,\, \breve{u}_n \rangle$
   is admissible to the maximal cone $\sigma$ and
 we may take\\   $\breve{M}_\sigma = \langle \breve{u}_1,\,\cdots\,,\, \breve{u}_n \rangle$.

 \bigskip

 \noindent
 {\it Step $(b):$}\: {\it Submonoids of $\breve{M}$ admissible to lower-dimensional cones in $\Delta$}

 \medskip

 \noindent
 For each maximal cone $\sigma\in\Delta$, fix a submonoid $\breve{M}_\sigma\subset \breve{M}$
   admissible to $\sigma$ as constructed in Step (a).
 For a cone $\tau\in\Delta$ of dimension $<n$, let
  $\Delta_\tau:= \{\sigma\,|\,\tau\prec \sigma,\,\mbox{$\sigma$ maximal cone}\}\subset \Delta$
 and consider the submonoid of $\breve{M}$ generated by all $\breve{M}_\sigma$,
   where $\sigma$ is a maximal cone in $\Delta$ that contains $\tau$:
  $$
   \breve{M}_\tau^\prime\;
    :=\; \langle \breve{M}_\sigma \,|\, \sigma\in\Delta_\tau \rangle\; \subset\; \breve{M}\,.
  $$
 Since each $\sigma\in\Delta_\tau$ is strongly convex,
  the submonoid $M_\tau := \tau^\vee\cap M$ of $M$ is generated by
   $M_\gamma$, $\gamma\in \Delta_\tau$.
 It follows that the built-in monoid epimorphism restricts to a monoid epimorphism
  $\pi:\breve{M}_\tau^\prime \rightarrow M_\tau$.
 The submonoid of $M_\tau$ that consists of invertible elements in $M_\tau$ is $\tau^\perp\cap M$.

 \bigskip

 \begin{lemma} {\bf [finite generatedness]}\;
  The submonoid $\pi^{-1}(\tau^\perp\cap M)\cap \breve{M}_\tau^\prime$ of $\breve{M}_\tau^\prime$
  is finitely generated.
 \end{lemma}

 \medskip

 \noindent
 {\it Proof.}\:\:
  For $\sigma$ a maximal cone in $\Delta$, let
    $u_{\sigma, 1},\,\cdots\,,\,u_{\sigma, n}$ be the generators of $M_\sigma$  and
    $\breve{u}_{\sigma, 1},\,\cdots\,,\,\breve{u}_{\sigma, n}$ be the generators of the monoid
      $\breve{M}_{\sigma}$ with $\pi(\breve{u}_{\sigma, i})=u_{\sigma, i}$.
  Then, by construction,\\
   $\{\breve{u}_{\sigma, i}\,|\, \sigma\in\Delta_\tau\,,\, i=1,\,\ldots\,,\,n \}\subset \breve{M}$
   generates the submonoid $\breve{M}_\tau^\prime$ of $\breve{M}$.
  Since $\tau^\perp\subset \partial \tau^\vee:= \cup_{\rho\prec\tau}\rho$  and
   $\tau^\vee = \sum_{\sigma\in\Delta_\tau}\sigma^\vee$ is convex
   with $\tau^\perp$ a linear stratum of $\partial\tau^\vee$,
  if $m_1, m_2\in M_\tau$ satisfy the condition $m_1+m_2\in\tau^\perp\cap M$,
    then both $m_1$ and $m_2$ must be in $\tau^\perp\cap M$.
  If follows that
   the submonoid
    $\pi^{-1}(\tau^\perp\cap M)\cap \breve{M}_\tau^\prime$ of $\breve{M}_\tau^\prime$
  is generated by
   $$
    \{\breve{u}_{\sigma, i}\,|\, \sigma\in\Delta_\tau\,,\, i=1,\,\ldots\,,\,n\,;\,
             u_{\sigma, i}\in \tau^\perp\}\,.
   $$
   
 \hspace{38.2em}$\Box$
 
 \bigskip

 Let
  $$
   \breve{S}_{\tau}\;:=\;
     \langle  	
 	  \breve{u}_{\sigma, i}\,,
	  \breve{u}_{\sigma, i}^{\:\:-1}\,|\, \sigma\in\Delta_\tau\,,\, i=1,\,\ldots\,,\,n\,;\,
             u_{\sigma, i}\in \tau^\perp	
	 \rangle
  $$
  and augment $\breve{M}_\tau^\prime$ to
  $$
    \breve{M}_\tau\;:=\; \langle \breve{M}_\tau^\prime\,,\, \breve{S}_{\tau} \rangle\;
	 \subset\; \breve{M}\,.
  $$
 Then, $\breve{M}_\tau$ is now a submonoid of $\breve{M}$ admissible to $\tau\in\Delta$.
 By construction, note that
  the submonoid $\breve{M}_\tau^\ast$ of $\breve{M}_\tau$
    that consists of all the invertible elements of the monoid $\breve{M}_\tau$
    coincides with $\breve{S}_\tau$;
  and that
   $\breve{M}_{\mathbf{0}}=\breve{M}$.
 
 \bigskip

 \noindent
 {\it Step $(c)$.}\: {\it The inverse $\Delta$-system $\{\breve{M}_\sigma\}_{\sigma\in\Delta}$}

 \medskip

 \noindent
 Consider
  the category of submonoids of $\breve{M}$
     with morphisms of objects given by inclusions of submonoids of $\breve{M}$    and
  the collection $\{\breve{M}_\sigma\}_{\sigma\in\Delta}$ of submonoids of $\breve{M}$
     that are admissible to cones in $\Delta$ constructed in Steps (a) and (b).
 For $\tau\prec\sigma$, $\Delta_\tau\supset \Delta_\sigma$ and hence
   $\breve{M}_\sigma\subset \breve{M}_\tau$.
 For $\rho\prec\tau\prec\sigma$,
   $\breve{M}_\sigma \subset \breve{M}_\tau \subset \breve{M}_\rho$ naturally.
 This shows that
  the collection $\{\breve{M}_\sigma\}_{\sigma\in\Delta}$
  is an inverse $\Delta$-system of submonoids of $\breve{M}$
  and concludes the proof of the theorem.

\end{proof}

\medskip

\begin{definition} {\bf [augmentation and diminishment]}\; {\rm
 Let $\{\breve{M}_\sigma\}_{\sigma\in\Delta}$    and
    $\{\breve{M}^\prime_\sigma\}_{\sigma\in\Delta}$
  be two inverse $\Delta$-systems of $\Delta$-admissible submonoids of $\breve{M}$
  such that $\breve{M}_\sigma \subset \breve{M}^\prime_\sigma$ for all $\sigma\in\Delta$.
 Then we say that
  $\{\breve{M}_\sigma\}_{\sigma\in\Delta}$ is a {\it diminishment} of
    $\{\breve{M}^\prime_\sigma\}_{\sigma\in\Delta}$    and
  $\{\breve{M}^\prime_\sigma\}_{\sigma\in\Delta}$ is an {\it augmentation} of
    $\{\breve{M}_\sigma\}_{\sigma\in\Delta}$.  	
 In notation,
   $\{\breve{M}_\sigma\}_{\sigma\in\Delta}
     \preccurlyeq  \{\breve{M}^\prime_\sigma\}_{\sigma\in\Delta}$	and
  $\{\breve{M}^\prime_\sigma\}_{\sigma\in\Delta}
    \succcurlyeq \{\breve{M}_\sigma\}_{\sigma\in\Delta}$.
}\end{definition}

\bigskip

Similar arguments to the proof of Theorem~2.2.5 give the following:

\bigskip

\begin{proposition} {\bf [completion to inverse $\Delta$-system]}\;
 Let $\Delta(n)$ be the set of $n$-dimensional cones in $\Delta$.
 Given $\{\breve{M}^\prime_\sigma\}_{\sigma\in\Delta(n)}$
     such that $\breve{M}^\prime_\sigma$ is a submonoid of $\breve{M}$ admissible to $\sigma$,
 there exists an inverse $\Delta$-system $\{\breve{M}_\tau\}_{\tau\in\Delta}$
  of admissible submonoids of $\breve{M}$
  such that $\breve{M}_\sigma = \breve{M}^\prime_\sigma$ for $\sigma\in\Delta(n)$.
 {\rm We shall call $\{\breve{M}_\tau\}_{\tau\in\Delta}$
  a {\it completion} of $\{\breve{M}^\prime_\sigma\}_{\sigma\in\Delta(n)}$
  to an inverse $\Delta$-system of admissible submonoids of $\breve{M}$.
    }
\end{proposition}

\begin{proof}
 The same proof as the proof of Theorem~2.2.5
   goes through
   as long as $\breve{M}_\sigma$, $\sigma\in\Delta(n)$, in Step (a) is admissible to $\sigma$.

\end{proof}

\medskip

\begin{proposition} {\bf [augmentation by $\Delta$-indexed submonoids]}\;
 Given an inverse $\Delta$-system $\{\breve{M}^\prime_\sigma\}_{\sigma\in\Delta}$
   of admissible submonoids of $\breve{M}$    and
  a collection $\{\breve{S}_\sigma\}_{\sigma\in\Delta}$ of finitely generated submonoids of $\breve{M}$
   such that $\pi(\breve{S}_\sigma)\subset M_\sigma$,
 there exists an augmentation $\{\breve{M}_\sigma\}_{\sigma\in\Delta}$  of
  $\{\breve{M}^\prime_\sigma\}_{\sigma\in\Delta}$
  such that $\breve{S}_\sigma \subset \breve{M}_\sigma$ for all $\sigma\in\Delta$.
\end{proposition}

\begin{proof}
 We begin with maximal cones $\sigma\in \Delta(n)$
   and define $\breve{M}_\sigma := \langle \breve{M}_\sigma, \breve{S}_\sigma\rangle$.
 Since $\sigma^\vee$ in this case is strongly convex in $M_{\Bbb R}$,
   $\breve{M}_\sigma$ must already be admissible to $\sigma$.
 We then proceed by induction by the condimension of cones in $\Delta$ and assume that
  $\breve{M}_\sigma$ for $\sigma\in\Delta(i)$, $1< k\le i\le n$, are constructed such that
  (1) $\breve{M}_\sigma$ is admissible to $\sigma$,
  (2) $\breve{M}_\sigma \supset \langle \breve{M}^\prime_\sigma, \breve{S}_\sigma \rangle$,
  (3) $\breve{M}_\tau\supset \breve{M}_\sigma$ if $\tau\prec\sigma$, 	
   for $\sigma, \tau\in\Delta(i)$, $1< k\le i\le n$.
 Let $\tau\in \Delta(k-1)$ and
  $$
    \breve{M}^{\prime\prime}_\tau \;
	 :=\;  \langle  \breve{M}^\prime_\tau, \breve{S}_\tau, \breve{M}_\sigma\,|\,
	               \sigma\in\Delta_\tau\cap\Delta(k)\rangle\,.
  $$
 Then
  $\breve{M}^{\prime\prime}_\tau\cap \pi^{-1}(\tau^\perp)$ is finitely generated
  and hence so is the submonoid
   $$
     \breve{S}^\prime_\tau\;
	  :=\; \{s \in \breve{M}\,|\,\mbox{$s$ or $s^{-1}$ is in
	                  $\breve{M}^{\prime\prime}_\tau\cap \pi^{-1}(\tau^\perp)$}\}
   $$
   of $\breve{M}$.
 Define
  $$
   \breve{M}_\tau\; :=\; \langle \breve{M}^{\prime\prime}_\tau, \breve{S}^\prime_\tau \rangle\,.
  $$
 Then
  (1) $\breve{M}_\tau$ is admissible to $\tau$,
  (2) $\breve{M}_\tau \supset \langle \breve{M}^\prime_\tau, \breve{S}_\tau \rangle$,
  (3) $\breve{M}_\tau\supset \breve{M}_\sigma$ if $\tau\prec\sigma$ for all $\sigma\in\Delta$.
 Finally we set $\breve{M}_{\mathbf{0}}=\breve{M}$.
 By construction,
  $\{\breve{M}_\sigma\}_{\sigma\in\Delta}$ is an augmentation of
    $\{\breve{M}^\prime_\sigma\}_{\sigma\in\Delta}$
    that satisfies $\breve{S}_\sigma \subset \breve{M}_\sigma$ for all $\sigma\in\Delta$.
 
\end{proof}

\medskip

\begin{corollary} {\bf [from $\Delta$-indexed to inverse $\Delta$-system]}\;
 Let $\{\breve{S}_\sigma\}_{\sigma\in\Delta}$ be a collection of finitely generated submonoids
   of $\breve{M}$ such that $\pi(\breve{S}_\sigma)\subset M_\sigma$.
 Then there exists an inverse $\Delta$-system $\{\breve{M}_\sigma\}_{\sigma\in\Delta}$
  of admissible submonoids of $\breve{M}$
  such that $\breve{S}_\sigma\subset \breve{M}_\sigma$ for all $\sigma\in\Delta$.
\end{corollary}

\begin{proof}
 Choose any inverse $\Delta$-system of admissible submonoids of $\breve{M}$ from Theorem~2.2.5
  and apply Proposition~2.2.10.

\end{proof}

\bigskip

\begin{flushleft}
{\bf Soft nocommutative toric schemes associated to $\Delta$}
\end{flushleft}
\begin{definition} {\bf [soft noncommutative toric schemes associated to $\Delta$]}\; {\rm
 (Continuing the previous notations.)
 Let
  $\{\breve{M}_\sigma\}_{\sigma\in\Delta}$ be an inverse $\Delta$-system of admissible submonoids
    of $\breve{M}$ such that $\breve{M}_\sigma\simeq \langle z_1,\,\cdots\,,\, z_n\rangle$ for
    all maximal cones $\sigma\in \Delta(n)$ and $\breve{M}_{\mathbf{0}}=\breve{M}$    and 	
  $\{\breve{R}_\sigma\}_{\sigma\in\Delta}
     := \{{\Bbb C}\langle\breve{M}_\sigma\rangle\}_{\sigma\in\Delta}$
	be the associated inverse $\Delta$-system of monoid algebras over ${\Bbb C}$.
 Then the corresponding $\Delta$-system of noncommutative affine schemes
   $$
     \breve{Y}(\Delta)\; :=\;
	  \breve{\cal U}(\Delta)\;:=\;  \{\breve{U}_\sigma\}_{\sigma\in\Delta}
   $$
   is called a {\it soft noncommutative toric scheme with smooth fundamental charts} associated to $\Delta$.
 Here,
  $\breve{U}_\sigma$ is the noncommutative affine scheme associated to $\breve{R}_\sigma$ (cf.\ Definition~2.1.1)
    for $\sigma\in\Delta$,
  the notation $\breve{\cal U}(\Delta)$ emphasizes that this is a gluing system of some particular charts, and
  the notation $Y(\Delta)$ emphasizes that this is treated as a noncommutative space
    as a whole.\footnote{Here,
                    we use the notation $Y(\Delta)$, rather than $X(\Delta)$ following [Fu], due to that
					such a space will be used as a target space for D-branes, whose world-volume is generally denoted
					$X^{\!\Aztiny}$.
					} 
					
 With an abuse of language,
  we shall call $\breve{Y}(\Delta)$ {\it $n$-dimensional}
  since the function ring of each fundamental chart is freely generated by $n$ coordinate functions.
 
 For $\sigma\in\Delta$,
   the monoid algebra ${\Bbb C}\langle\breve{M}_\sigma\rangle$ is called interchangeably
   the {\it function ring} or the {\it local coordinate rings} of the {\it chart} $\breve{U}_\sigma$.
 The assignment
   $$
     \breve{U}_\sigma\; \longmapsto\;
	   \breve{R}_\sigma := {\Bbb C}\langle \breve{M}_\sigma\rangle\,,\:\:
	  \sigma\in \Delta\,,
   $$	
   is called the {\it structure sheaf}\, of $\breve{Y}(\Delta)$    and
   is denoted ${\cal O}_{\breve{Y}(\Delta)}$ as in ordinary Commutative Algebraic Geometry.
 By construction, for cones $\tau\prec\sigma$ in $\Delta$, one has
  a dominant morphism of noncommutative affine schemes $\breve{U}_\tau\rightarrow \breve{U}_\sigma$
  and a monoid-algebra monomorphism\footnote{Though
                                         our focus in this work is on the rings and their homomorphisms for Noncommutative Algebraic Geometry,
						                 we try to keep the contravariant underlying geometry as in Commutative Algebraic Geometry in mind
										   whenever possible.
         								 This is why we use the notation $\iota_{\tau\sigma}^\sharp$ here   and
										   reserve $\iota_{\tau\sigma}$ for the corresponding morphism of underlying ``spaces".
			                             Here, there is no indication that
										     $\iota_{\tau\sigma}: \breve{U}_\tau \rightarrow \breve{U}_\sigma$
			                                 in any sense is an inclusion,
										 though it is true that, when restricted to $Y(\Delta)$,
				                         $\iota_{\tau\sigma}: U_\tau\rightarrow U_\sigma$ is an open set inclusion in the sense of
				                              affine schemes in Commutative Algebraic Geometry.
                                                                                          } 
    $$
	  \iota_{\tau\sigma}^\sharp\;:\;
	 {\cal O}_{\breve{Y}(\Delta)}(\breve{U}_\sigma)
	    = {\Bbb C}\langle \breve{M}_\sigma\rangle\;
	    \hookrightarrow\;
		{\cal O}_{\breve{Y}(\Delta)}(\breve{U}_\tau)
		    = {\Bbb C}\langle \breve{M}_\tau\rangle\,.
    $$
 These replace the role of `{\it open sets}' and `{\it restriction to open sets}'	 respectively
   in the usual definition of the `{\it structure sheaf}\,'
   of a {\it scheme} in Commutative Algebraic Geometry.
 
  An {\it ideal sheaf $\breve{\cal I}$} on $\breve{Y}$
    is an inverse $\Delta$-system $\{\breve{I}_\sigma\}_{\sigma\in\Delta}$,
    where $\breve{I}_\sigma$ is a two-sided ideal of $\breve{R}_\sigma$,
	such that $\iota_{\tau\sigma}^\sharp(I_\sigma)=I_\tau$.	
  Given an ideal sheaf $\breve{\cal I}=\{\breve{I}_\sigma\}_{\sigma\in\Delta}$ on $\breve{Y}(\Delta)$,
    one can form an inverse $\Delta$-system of ${\Bbb C}$-algebras
	$\{\breve{R}_\sigma/\breve{I}_\sigma\}_{\sigma\in\Delta}$,
	with the inclusion
	$\underline{\iota}_{\tau\sigma}^\sharp: \breve{R}_\sigma/\breve{I}_\sigma
	    \hookrightarrow \breve{R}_\tau/\breve{I}_\tau$, for $\tau\prec\sigma\in\Delta$,
	 naturally induced from $\iota_{\tau\sigma}^\sharp$.
  We will think of this inverse $\Delta$-system as defining a {\it soft noncommutative closed subscheme} $\breve{Z}$
   of $\breve{Y}(\Delta)$ and denoted its structure sheaf ${\cal O}_{\breve{Z}}$
   and write the quotient as ${\cal O}_{\breve{Y}(\Delta)}\rightarrow {\cal O}_{\breve{Z}}$.
 Redenote $\breve{\cal I}$ by ${\cal I}_{\breve{Z}}$. Then one has a short exact sequence
    $$
     0\; \longrightarrow\; {\cal I}_{\breve{Z}}\;  \longrightarrow\; {\cal O}_{\breve{Y}(\Delta)}\;
	      \longrightarrow\;  {\cal O}_{\breve{Z}}\; \longrightarrow\; 0\,.
    $$	
 
 Let $N^\prime$ be another lattice (isomorphic to ${\Bbb Z}^{n^\prime}$
      for some $n^\prime$ via a fixed basis $(e^\prime_1, \,\cdots\,,\, e^\prime_{n^\prime})$ of $N^\prime$),
  $\Delta^\prime$ be a fan in $N^\prime_{\Bbb R}$ that satisfies Assumption~2.2.2,   and
  $\breve{Y}(\Delta^\prime)$ be a soft noncommutative toric scheme associated to $\Delta^\prime$,
    with the underlying inverse $\Delta^\prime$-system of submonoids of $\breve{M}^\prime$
	  denoted $\{\breve{M}^\prime_{\sigma^\prime}\}_{\sigma^\prime\in\Delta^\prime}$.
 A {\it toric morphism}
   $$
     \breve{\varphi}\; :\; \breve{Y}(\Delta)\; \longrightarrow\;  \breve{Y}(\Delta^\prime)
   $$
   is a homomorphism $\varphi: N \rightarrow N^\prime$ of lattices that satisfies the conditions:
   \begin{itemize}
    \item[(1)]
	$\varphi$ induces a {\it map between fans} (in notation, $\varphi: \Delta\rightarrow \Delta^\prime$);
	 i.e., for each $\sigma\in\Delta$ there exists a $\sigma^\prime\in\Delta^\prime$ such that
	   $\varphi(\sigma)\subset \sigma^\prime$.
	
    \item[(2)]	
	The induced monoid-homomorphism $\breve{\varphi}^\sharp:\breve{M}^\prime\rightarrow \breve{M}$
      restricts to a monoid homomorphism
	  $\breve{\varphi}_{\sigma\sigma^\prime}^\sharp:
	         \breve{M}^\prime_{\sigma^\prime}\rightarrow \breve{M}_\sigma$
	  whenever $\varphi(\sigma)\subset \sigma^\prime$.
	 (Denote the induced
	    ${\Bbb C}\langle\breve{M}^\prime_{\sigma^\prime}\rangle
		  \rightarrow {\Bbb C}\langle\breve{M}_\sigma\rangle$
		also by $\breve{\varphi}_{\sigma\sigma^\prime}^\sharp$.)	
	 (Cf.\ Footnote~2.)
   \end{itemize}
  Thus, a toric morphism $\breve{\varphi}: \breve{Y}(\Delta)\rightarrow \breve{Y}(\Delta^\prime)$
    is a map of lattices $\varphi: N\rightarrow N^\prime$ that induces
	a $(\Delta^\prime, \Delta)$-system of monoid-algebra homomorphisms
	$\{\breve{\varphi}_{\sigma\sigma^\prime}^\sharp:	
	  {\Bbb C}\langle\breve{M}^\prime_{\sigma^\prime}\rangle
		  \rightarrow {\Bbb C}\langle\breve{M}_\sigma\rangle	
	      \}_{\sigma^\prime\in\Delta^\prime,\, \sigma\in\Delta,\, \varphi(\sigma)\subset \sigma^\prime}$.
  The nature of this collection justifies it be denoted
   $$
    \breve{\varphi}^\sharp\;:\;
	  {\cal O}_{\breve{Y}(\Delta^\prime)}\; \longrightarrow\;
	  {\cal O}_{\breve{Y}(\Delta)}\,.
   $$
}\end{definition}

\medskip

\begin{definition} {\bf [reference toric chart and fundamental toric charts of $\breve{Y}(\Delta)$]}\; {\rm
 (Continuing Definition~2.2.12.)
 A chart $\breve{U}_\sigma$ of $\breve{Y}(\Delta)$
     that is associated to a maximal cone $\sigma\in\Delta(n)$
   is called a {\it fundamental toric chart} of $\breve{Y}(\Delta)$.
 Among them,
   the one associated to the cone generated by the given basis $(e_1,\,\cdots\,,\,e_n)$ of the lattiice $N$
   is called the {\it reference toric chart} of $\breve{Y}(\Delta)$.
}\end{definition}

\bigskip

Fundamental toric charts are all isomorphic to the smooth noncommutative affine $n$-space $\nc\mathbf{A}^n$.
These charts should be thought of as the basic, good pieces to be glued via the subcollection
  $\{\breve{U}_\tau\}_{\tau\in \Delta-\Delta(n)}$.
The reference chart ensures that we have at least one fundamental toric chart $\breve{U}_\sigma$
  such that the morphism $\breve{U}_{\mathbf{0}}\rightarrow \breve{U}_\sigma$
   of noncommutative affine schemes resembles the inclusion of an Zariski open set.
Together, this allows one to think of $\breve{Y}(\Delta)$ as a ``partial compactification" of $\nc\mathbf{A}^n$
  via soft gluings of additional copies of $\nc\mathbf{A}^n$'s.
 
\bigskip

\begin{definition} {\bf [softening of noncommutative toric scheme]}\; {\rm
 Let $\breve{Y}(\Delta)$ and $\breve{Y}^\prime(\Delta)$ be two soft noncommutative toric schemes
  associated to a fan $\Delta$, with the underlying inverse $\Delta$-system of monoid algebras
   $\{{\Bbb C}\langle \breve{M}_\sigma \rangle\}_{\sigma\in\Delta}$ and
   $\{{\Bbb C}\langle\breve{M}^\prime_\sigma\rangle\}_{\sigma\in\Delta}$ respectively.
 Then $\breve{Y}^\prime(\Delta)$ is called a {\it softening} of $\breve{Y}(\Delta)$
  if $\breve{M}_\sigma=\breve{M}^\prime_\sigma$ for maximal cones $\sigma\in\Delta$
     (i.e.\ $\breve{Y}(\Delta)$ and $\breve{Y}^\prime(\Delta)$ have identical fundamental toric charts)
	  and
	 $\breve{M}_\sigma \subset \breve{M}^\prime_\sigma$ for nonmaximal cones $\sigma\in\Delta$
	 (i.e.\ the gluings in $\breve{Y}^\prime(\Delta)$ is {\it softer} than those in $\breve{Y}(\Delta)$).	
 By construction, there is a built-in toric morphism\\
   $\breve{\varphi}: \breve{Y}^\prime(\Delta)\rightarrow \breve{Y}(\Delta)$, called
   the {\it softening morphism}.
}\end{definition}
 
\medskip

\begin{remark} $[$meaning of softening: comparison to $\Isom$-functor for moduli stack$]$\; {\rm
 Introducing the notion of `softening' in the play is meant to remedy the fact that we don't have a good theory of localizations
  of  a noncommutative ring to apply to our construction of noncommutative schemes.
 Rather we keep a collection of charts as principal ones and the remaining charts serve for the purpose of gluing
  these principal charts. These `secondary charts' are then allowed to be `dynamically adjusted' to fulfill the purpose of
  serving as a medium between principal charts.
 One may compare this to the gluing in the construction of a moduli stack.
 There one cannot directly glue two charts for the moduli stack as schemes but rather has to create the medium chart
  via the $\Isom$-functor, cf.\ [Mu].
 See Sec.\ 3 how softening is used to construct invertible sheaves and twisted sections on a soft noncommutative toric scheme.
}\end{remark}

\bigskip

\noindent
$(c)$ {\it The ${\Bbb T}^n$-action}

\medskip

\noindent
The ${\Bbb T}^n$-action on ${\Bbb C}[M]$ leaves each ${\Bbb C}\cdot m$, $m\in M$, invariant.
As ${\Bbb C}= \Center ({\Bbb C}\langle \breve{M}\rangle)$,
 the ${\Bbb T}^n$-action on ${\Bbb C}[M]$ naturally lifts to a ${\Bbb T}^n$-action
  on ${\Bbb C}\langle\breve{M}\rangle$:
  $$
   \hspace{4em}
   \xymatrix{
    {\Bbb C}\langle\breve{M}\rangle  \ar[d]_-{\pi} \ar[rr]^-{\mathbf {t}}
	       &&  {\Bbb C}\langle\breve{M}\rangle   \ar[d]^-{\pi}\\
	 {\Bbb C}[M] \ar[rr]^-{\mathbf{t}}                                             &&{\Bbb C}[M]
	       &\hspace{-2.4em},  \hspace{2em}\mathbf{t}\in {\Bbb T}^n,
    }
  $$
  that leaves each ${\Bbb C}\cdot\breve{m}$, $\breve{m}\in \breve{M}$, invariant.
 It follows that
   ${\Bbb T}^n$ leaves each ${\Bbb C}\langle \breve{M}_\sigma\rangle$, $\sigma\in\Delta$, invariant
      and that
   $\iota_{\tau\sigma}^\sharp:
        {\Bbb C}\langle\breve{M}_\sigma\rangle \hookrightarrow {\Bbb C}\langle \breve{M}_\tau\rangle$,
	 $\tau\prec\sigma$,
    are ${\Bbb T}^n$-equivariant.
 This defines the ${\Bbb T}^n$-action on $\breve{Y}(\Delta)$.

\bigskip

\noindent
$(d)$ {\it The noncommutative toroidal morphism from the inclusion of the $0$-cone
        $\mathbf{0}\prec\sigma\in\Delta$}

\medskip

\noindent
Since
 $\mathbf{0}\preccurlyeq \sigma$, for all $\sigma\in\Delta$, and
 $\breve{M}_\mathbf{0}={\Bbb C}\langle\breve{M}\rangle$,
 one has a built-in ${\Bbb T}^n$-equivariant morphism
  $$
     \breve{\Bbb T}^n\, :=\, \Scheme ({\Bbb C}\langle\breve{M}\rangle)\;
      \longrightarrow\;  \breve{Y}(\Delta)\,.
  $$
Which plays the role of the toroidal embedding ${\Bbb T}^n\hookrightarrow Y(\Delta)$ in the commutative case.

\bigskip

\noindent
$(e)$ {\it The built-in ${\Bbb T}^n$-equivariant embedding $Y(\Delta)\hookrightarrow \breve{Y}(\Delta)$}

\medskip

\noindent
The monoid epimorphism $\pi:\breve{M}\rightarrow M$ restricts to monoid epimorphisms
 $\pi: \breve{M}_\sigma\rightarrow M_\sigma$, for all $\sigma\in\Delta$.
Which in turn gives a $\Delta$-collection of ${\Bbb T}^n$-equivariant monoid-algebra epimorphisms
 $\{{\Bbb C}\langle\breve{M}_\sigma\rangle \rightaarrow {\Bbb C}[M_\sigma]\}_{\sigma\in\Delta}$
  that commute with inclusions:
   $$
    \hspace{6em}
    \xymatrix{
     {\Bbb C}\langle\breve{M}_\sigma\rangle\;   \ar@{^{(}->}[rr]  \ar[d]_-{\pi}
	     && {\Bbb C}\langle\breve{M}_\tau\rangle   \ar[d]^-{\pi}     \\
	 {\Bbb C}[M_\sigma]\;  \ar@{^{(}->}[rr]
	   && {\Bbb C}[M_\tau]
	        &\hspace{-2.4em}, \hspace{2em}\tau\prec\sigma\in\Delta\,.
     }
   $$
This defines a built-in ${\Bbb T}^n$-equivariant embedding $Y(\Delta)\hookrightarrow \breve{Y}(\Delta)$.

\bigskip

\noindent
$(f)$ {\it The master commutative subscheme of\, $\breve{Y}(\Delta)$}

\medskip

\noindent
The inclusions
   ${\Bbb C}\langle\breve{M}_\sigma\rangle \hookrightarrow {\Bbb C}\langle\breve{M}_\tau\rangle   $,
   $\tau\prec\sigma\in\Delta$,
 induce ${\Bbb C}$-algebra homomorphisms
   $$
    \jmath_{\tau\sigma}^\sharp\; :\;
	   {\Bbb C}\langle\breve{M}_\sigma\rangle
	   /[{\Bbb C}\langle\breve{M}_\sigma\rangle,\,{\Bbb C}\langle\breve{M}_\sigma\rangle ]\;
	 \longrightarrow\;
	  {\Bbb C}\langle\breve{M}\rangle_\tau
	   /[{\Bbb C}\langle\breve{M}_\tau\rangle,\,{\Bbb C}\langle\breve{M}_\tau\rangle ]\,,
	      \hspace{2em} \tau\prec\sigma\in\Delta\,.	
   $$
This defines another inverse $\Delta$-system of commutative ${\Bbb C}$-algebras
 $$
  (\{{\Bbb C}\langle\breve{M}_\sigma\rangle
               /[{\Bbb C}\langle\breve{M}_\sigma\rangle,\,
			      {\Bbb C}\langle\breve{M}_\sigma\rangle]\}_{\sigma\in\Delta},\,
      \{\jmath_{\tau\sigma}^\sharp\}_{\tau\prec\sigma\in\Delta})
 $$
 and hence a (generally singular) commutative scheme $\breve{Y}(\Delta)^\tinyclubsuit$ over ${\Bbb C}$.
By construction
  any morphism from a commutative scheme $X$ to $\breve{Y}$ factors through a morphism\\
  $X\rightarrow \breve{Y}(\Delta)^\tinyclubsuit \hookrightarrow \breve{Y}(\Delta)$.
In particular, one has commutative diagrams
 $$
  \xymatrix{
   {\Bbb C}\langle\breve{M}_\sigma\rangle
	   /[{\Bbb C}\langle\breve{M}_\sigma\rangle,\,{\Bbb C}\langle\breve{M}_\sigma\rangle ]
	   \ar[rr]^-{\jmath_{\tau\sigma}^\sharp} \ar@{->>}[d]
	  &&
	  {\Bbb C}\langle\breve{M}\rangle_\tau
	   /[{\Bbb C}\langle\breve{M}_\tau\rangle,\,{\Bbb C}\langle\breve{M}_\tau\rangle ]
	       \ar@{->>}[d]      \\
  \;{\Bbb C}[M_\sigma]\;  \ar@{^{(}->}[rr]
     && {\Bbb C}[M_\tau]
	    & \hspace{-2em}, \hspace{2em} \tau\prec\sigma\in\Delta\,,	
   }
 $$
 and hence $Y(\Delta)\; \subset\, \breve{Y}(\Delta)^\tinyclubsuit \subset \breve{Y}(\Delta)$.
 Note that, by construction,\\
   ${\Bbb C}\langle\breve{M}_\sigma\rangle
	   /[{\Bbb C}\langle\breve{M}_\sigma\rangle,\,{\Bbb C}\langle\breve{M}_\sigma\rangle ]
	  \stackrel{\sim}{\longrightaarrow} {\Bbb C}[M_\sigma]	 $ for $\sigma\in\Delta(n)\cup\{\mathbf{0}\}$.
 However, such isomorphism on local charts may not holds for $\tau\in\Delta(k)$, $1\le k\le n-1$.
 
\bigskip

\begin{definition} {\bf [master commutative subscheme of $\breve{Y}(\Delta)$]}\; {\rm
 $\breve{Y}(\Delta)^\tinyclubsuit$ is called the {\it master commutative subscheme} of $\breve{Y}(\Delta)$.
}\end{definition}

\bigskip
\bigskip

\section{Invertible sheaves on a soft noncommutative toric scheme, twisted sections,
        and soft noncommutative schemes via toric geometry}

We construct in this section a class of noncommutative spaces, named `soft noncommutative schemes'
 from invertible sheaves and their twisted sections on $\breve{Y}(\Delta)$.

\bigskip

\begin{flushleft}
{\bf Modules and invertible sheaves on a soft noncommutative toric scheme}
\end{flushleft}
\begin{sdefinition} {\bf [sheaf on $\breve{Y}(\Delta)$ and ${\cal O}_{\breve{Y}(\Delta)}$-module]}\; {\rm
 $(1)$\:\:
 An inverse $\Delta$-system of objects in a category ${\cal C}$ (cf.\ Definition 2.2.1)
  is also called a {\it sheaf} (of objects in ${\cal C}$) on $\breve{Y}(\Delta)$.
 For example,
   ${\cal O}_{\breve{Y}(\Delta)}$ is a sheaf of ${\Bbb C}$-algebras on $\breve{Y}(\Delta)$.
  (Which justifies its name: the structure sheaf of $\breve{Y}(\Delta)$).
 
 \medskip
 
 $(2)$\:\:
 Let $\breve{Y}(\Delta) =\{\breve{U}_\sigma\}_{\sigma\in\Delta}$
  be a soft noncommutative toric scheme with the underlying inverse $\Delta$-system of monoid algebras
  $(\{\breve{R}_\sigma\}_{\sigma\in\Delta},
       \{\iota^\sharp_{\tau\sigma}\}_{\tau,\,\sigma\in\Delta,\,\tau\prec\sigma} )$.
 A {\it left ${\cal O}_{\breve{Y}(\Delta)}$-module}
  is an inverse $\Delta$-system
  $$
    \breve{\cal F}\;
    :=\; (\{\breve{F}_\sigma\}_{\sigma\in\Delta},
              \{h_{\sigma\tau}\}_{\tau,\, \sigma\in\Delta,\, \tau\prec\sigma})\,,
  $$
  where
    $\breve{F}_\sigma$ is a left $\breve{R}_\sigma$-module  and
	$h_{\sigma\tau}:
	    \iota_{\sigma\tau}^\sharp \breve{F}_\sigma
		   :=   \breve{R}_\tau\otimes_{\breve{R}_\sigma}\!\breve{F}_\sigma
	    \rightarrow \breve{F}_\tau$
  on $\breve{U}_\tau$ such that the following diagram commutes
  $$
   \xymatrix{
    \iota_{\sigma\rho}^\sharp \breve{F}_\sigma
	        \ar[rr]^-{\iota_{\tau\rho}^\sharp h_{\sigma\tau}   }
			\ar@/^2.6pc/[rrrr]^-{h_{\sigma\rho}}
	  && \iota_{\tau\rho}^\sharp \breve{F}_\tau
	       \ar[rr]^-{h_{\tau\rho}}
	  && \breve{F}_\rho
   }
  $$
  on $\breve{U}_\rho$ for $\rho\prec\tau\prec\sigma$.	
  
  An element of $\breve{F}_\sigma$  is called
    a {\it local section}  of $\breve{\cal F}$ over $\breve{U}_\sigma$.
  A {\it section}  (or {\it global section})of $\breve{\cal F}$ is
   a collection $\{\breve{s}_\sigma\}_{\sigma\in\Delta}$ of local sections
     $\breve{s}_\sigma\in \breve{F}_\sigma$, $\sigma\in\Delta$, of $\breve{\cal F}$
   such that $h_{\sigma\tau}(\breve{s}_\sigma)=\breve{s}_\tau$, $\tau\prec\sigma\in\Delta$.
  In the last expression, $\breve{s}_\sigma$ is identified as
   $1\otimes \breve{s}_\sigma
       \in \breve{R}_\tau\otimes_{\breve{R}_\sigma}\breve{F}_\sigma$.
  The ${\Bbb C}$-vector space of sections of $\breve{\cal F}$ on $\breve{Y}(\Delta)$
    is denoted $H^0(\breve{Y}(\Delta), \breve{\cal F})$ or simply $H^0(\breve{\cal F})$.
   
  A {\it homomorphism}
   $$
     f\; :\;
	  \breve{\cal F}\,
	    :=\, (\{\breve{F}_\sigma\}_{\sigma\in\Delta},
              \{h_{\sigma\tau}\}_{\tau,\, \sigma\in\Delta,\, \tau\prec\sigma})\;
	   \longrightarrow\;
	 \breve{\cal F}^\prime\,
	  :=\, (\{\breve{F}^\prime_\sigma\}_{\sigma\in\Delta},
              \{h^\prime_{\sigma\tau}\}_{\tau,\, \sigma\in\Delta,\, \tau\prec\sigma})		
   $$
   of left ${\cal O}_{\breve{Y}(\Delta)}$-modules
   is a collection $\{f_\sigma\}_{\sigma\in\Delta}$,
    where
	 $f_\sigma: \breve{F}_\sigma\rightarrow \breve{F}^\prime_\sigma$
	  is a homomorphism of left $\breve{R}_\sigma$-modules,
    such that the following diagrams commute    	
	$$
	  \xymatrix{	
	   \iota^\sharp_{\sigma\tau}\breve{F}_\sigma
	    \ar[rr]^-{h_{\sigma\tau}}    \ar[d]_-{\iota^\sharp_{\sigma\tau}(f_\sigma)}
	      && \breve{F}_\tau \ar[d]^-{f_\tau}  \\
	   \iota^\sharp_{\sigma\tau} \breve{F}^\prime_\sigma
	   \ar[rr]^-{h^\prime_{\sigma\tau}}
	      && \breve{F}^\prime_\tau
	  }
	$$
    for $\tau\prec\sigma\in\Delta$.
  $f$ is called
   {\it injective} (i.e.\ $f$ a {\it monomorphism}) if in addition $f_\sigma$ is injective for all $\sigma\in\Delta$;
   {\it surjective }(i.e.\ $f$ an {\it epimorphism}) if in addition $f_\sigma$ is surjective for all $\sigma\in\Delta$;
    an {\it isomorphism} if in addition $f_{\sigma}$ is an isomorphism of left $\breve{R}_\sigma$-modules
	      for all $\sigma\in\Delta$.
  
  \medskip
  
  $(3)$\:\:
  Recall the built-in ${\Bbb T}^n$-action $g_\bullet$ on $\breve{Y}(\Delta)$.
  Then $\breve{\cal F}$ is called {\it ${\Bbb T}^n$-equivariant}
    if there exist isomorphisms
	  $\phi_{\mathbf{t}}: g_{\mathbf{t}}^\ast\breve{\cal F}
		  \stackrel{\sim}{\longrightarrow }  \breve{\cal F}$,
	  $\mathbf{t}\in{\Bbb T}^n$,
	of ${\cal O}_{\breve{Y}(\Delta)}$-modules such that
	$\phi_{\mathbf{t}_2}\circ \phi_{\mathbf{t}_1}= \phi_{\mathbf{t}_1+\mathbf{t}_2}$\,:
   $$
    \xymatrix{
	  g_{\mathbf{t}_2}^\ast(g_{\mathbf{t}_1}^\ast \breve{\cal F})
	     \ar[rr]^-{\phi_{\mathbf{t}_1}} \ar@/^2.6pc/[rrrr]^-{\phi_{\mathbf{t}_1+\mathbf{t}_2}}		
	  && g_{\mathbf{t}_2}^\ast\breve{\cal F} \ar[rr]^-{\phi_{\mathbf{t}_2}}
	  && \breve{\cal F}	
	}
   $$
   for all $\mathbf{t}_1,\, \mathbf{t}_2\in{\Bbb T}^n$. 	
  Explicitly, note that
    the ${\Bbb T}^n$-action on $\breve{Y}(\Delta)$ leaves each chart
	$\breve{U}_\sigma$, $\sigma\in\Delta$ invariant and
  $\phi_\mathbf{t}$ restricts to an $\breve{R}_\sigma$-module homomorphism
    $$
	  \phi_\mathbf{t}\; :\; \breve{F}_\sigma^\mathbf{t}\; \longrightarrow\; \breve{F}_\sigma\,,\:\:
	   \mbox{with}\:\:
	   \breve{s}\; \longmapsto\; \phi_\mathbf{t}(\breve{s})\,=:\, g_\mathbf{t}\cdot\breve{s}
	$$
    on each chart $\breve{U}_\sigma$, $\sigma\in\Delta$.
  Here $\breve{F}_\sigma^\mathbf{t}:= g_\mathbf{t}^\ast\breve{F}_\sigma$,
    which is the same ${\Bbb C}$-vector space as $\breve{F}_\sigma$
	but with the $\breve{R}_\sigma$-module structure defined by
     $\breve{r}\stackrel{\mathbf{t}}{\cdot}\breve{s}
	   :=  g_\mathbf{t}^\sharp(\breve{r})\cdot \breve{s}$.		
  This defines a ${\Bbb T}^n$-action on $\breve{F}_\sigma$ that satisfies
	$\phi_\mathbf{t}(g_\mathbf{t}^\sharp(\breve{r})\cdot \breve{s})
	   = \breve{r} \cdot (g_\mathbf{t}\cdot \breve{s})$
	 by tautology, for  $\mathbf{t}\in{\Bbb T}^n$,  $\breve{r}\in \breve{R}_\sigma$, and
		$\breve{s}\in \breve{F}_\sigma$.
  The ${\Bbb T}^n$-action on $\breve{F}_\sigma$, $\sigma\in\Delta$, is equivariant under gluings:
	$$
      h_{\sigma\tau}(g_\mathbf{t}\cdot\breve{s})\;
	    =\;  g_\mathbf{t}\cdot h_{\sigma\tau}(\breve{s}) \,, 	
	$$
	for $\breve{s}\in \breve{F}_\sigma$,
	     $\mathbf{t}\in {\Bbb T}^n$, and
		 $\tau\prec\sigma\in\Delta$.

  \medskip
 
  $(4)$\:\:		
  Similarly, for {\it right} and {\it two-sided ${\cal O}_{\breve{Y}(\Delta)}$-modules}.
  
  \medskip
  
  $(5)$\:\:
  Let
     $\breve{\cal A}= (\{\breve{A}_\sigma\}_{\sigma\in\Delta},
	                              \{\iota_{\tau\sigma}^\sharp\}_{\tau\prec\sigma\in\Delta})$
	       be a sheaf of ${\Bbb C}$-algebras,
	 $\breve{\cal F}=(\{\breve{F}_\sigma\}_{\sigma\in\Delta},
	                            \{h^{\breve{\cal F}}_{\sigma\tau}\}_{\tau\prec\sigma\in\Delta})$
	      a right $\breve{\cal A}$-module,    and
	 $\breve{\cal G}=(\{\breve{G}_\sigma\}_{\sigma\in\Delta},
	                             \{h^{\breve{\cal G}}_{\sigma\tau}\}_{\tau\prec\sigma\in\Delta})$
		  a left $\breve{\cal A}$-module,
	 all three on $\breve{Y}(\Delta)$.
  Then
     define the {\it tensor product} $\breve{\cal F}\otimes_{\breve{\cal A}}\breve{\cal G}$
	   of $\breve{\cal F}$ and $\breve{\cal  G}$ over $\breve{\cal A}$ to be
     the following (right-on-$\breve{\cal F}$, left-on-$\breve{\cal G}$) $\breve{\cal A}$-module
	 on $\breve{Y}(\Delta)$
	 $$
	   \breve{\cal F}\otimes_{\breve{\cal A}}\breve{\cal G}\;
		   :=\;  (\{\breve{F}_\sigma
		                      \otimes_{\breve{A}_\sigma}\!\breve{G}_\sigma\}_{\sigma\in\Delta}\,,\,
		              \{h_{\sigma\tau}^{\breve{\cal F}}\otimes h_{\sigma\tau}^{\breve{\cal G}}
					              \}_{\tau\prec\sigma\in\Delta})\,.
	 $$
	 	  	
  \medskip
  
  $(6)$\:\:
  Let
    $$
	   \breve{\varphi}\; :\;
	    \breve{Y}(\Delta)\, =\, \{\breve{R}_\sigma\}_{\sigma\in\Delta}\;
	     \longrightarrow\;
		      \breve{Y}(\Delta^\prime)\,
		                  =\, \{\breve{S}_{\sigma^\prime\in\Delta^\prime}\}_{\sigma^\prime\in\Delta^\prime}
    $$
	 be a toric morphism (cf.\ Definition~2.2.12),
   $\breve{\cal F}:=(\{\breve{F}_\sigma\}_{\sigma\in\Delta},
	                            \{h^{\breve{\cal F}}_{\sigma\tau}\}_{\tau\prec\sigma\in\Delta})$
	  a left ${\cal O}_{\breve{Y}(\Delta)}$-module on $\breve{Y}(\Delta)$,     and
   $\breve{\cal G}:=(\{\breve{G}_{\sigma^\prime}\}_{\sigma^\prime\in\Delta^\prime},
	            \{h^{\breve{\cal G}}_{\sigma^\prime\tau^\prime}\}
				                            _{\tau^\prime\prec\sigma^\prime\in\Delta^\prime})$
		  a left ${\cal O}_{\breve{Y}(\Delta^\prime)}$-module on $\breve{Y}(\Delta^\prime)$.
  Recall the underlying homomorphism $N\rightarrow N^\prime$ of lattices     and
	   the induced ${\Bbb R}$-linear map $N_{\Bbb R}\rightarrow N^\prime_{\Bbb R}$,
	   both denoted still by $\varphi$, (cf.\ Definition 2.2.12).

  For each $\sigma^\prime\in\Delta^\prime$,
    $$
	   \Delta^{\sigma^\prime}\; :=\; \{\sigma\in\Delta\,|\, \varphi(\sigma)\subset \sigma^\prime\}
	$$
	  is a subfan of $\Delta$ and specifies a subsystem
	   ${\cal O}_{\breve{Y}(\Delta^{\sigma^\prime})}
	     := \{\breve{R}_{\sigma}\}_{\sigma\in\Delta^{\sigma^\prime}}$
	   of ${\cal O}_{\breve{Y}(\Delta)}$.
  Through the restriction of
	 $\breve{\varphi}^\sharp: {\cal O}_{\breve{Y}(\Delta^\prime)}
	           \rightarrow {\cal O}_{\breve{Y}(\Delta)}$ 		
      to $\breve{S}_{\sigma^\prime}\rightarrow {\cal O}_{\breve{Y}(\Delta^\prime)}$,
    the ${\Bbb C}$-vector space
	   $H^0(\breve{Y}(\Delta^{\sigma^\prime}), {\cal F}|_{\breve{Y}(\Delta^{\sigma^\prime})})$
	    is rendered a left $\breve{S}_{\sigma^\prime}$-module.
  The gluing data
	  $\{h^{\breve{\cal F}}_{\sigma\tau}\}_{\tau\prec\sigma\in\Delta}$ of $\breve{\cal F}$
     gives rise to a gluing data for the $\Delta^\prime$-collection
	  $\{H^0(\breve{Y}(\Delta^{\sigma^\prime}),
		                   {\cal F}|_{\breve{Y}(\Delta^{\sigma^\prime})})\}_{\sigma^\prime\in\Delta^\prime}$
     and turn the $\Delta^\prime$-collection into
	   a left ${\cal O}_{\breve{Y}(\Delta^\prime)}$-module,
	   denoted $\breve{\varphi}_\ast\breve{\cal F}$
	   and named the {\it direct image sheaf} or {\it pushforward} of $\breve{\cal F}$ under $\breve{\varphi}$.

  For each $\sigma\in\Delta$, there exists a unique $\rho^\prime_\sigma\in\Delta^\prime$
    such that $\varphi(\sigma)\subset\rho^\prime_\sigma$ and that $\rho^\prime_\sigma \preccurlyeq\sigma^\prime$
	 for all $\sigma^\prime\in\Delta^\prime$ with $\varphi(\sigma)\subset \sigma^\prime$.
  The property $\rho^\prime_\tau\preccurlyeq\rho^\prime_\sigma$ for $\tau\prec\sigma\in\Delta$   and
   the gluing data
	  $\{h^{\breve{\cal G}}_{\sigma^\prime\tau^\prime}\}_{\tau^\prime\prec\sigma^\prime\in\Delta^\prime}$
	  of $\breve{\cal G}$
	render the $\Delta$-collection $\{\breve{G}_{\rho^\prime_\sigma}\}_{\sigma\in\Delta}$
	 a sheaf of ${\Bbb C}$-vector spaces on $\breve{Y}(\Delta)$,
	 denoted $\breve{\varphi}^{-1}\breve{\cal G}$     and
	 named the {\it inverse image sheaf}  of $\breve{\cal G}$ under $\breve{\varphi}$.
   In particular, $\breve{\varphi}^{-1}{\cal }{\cal O}_{\breve{Y}(\Delta^\prime)}$
     is a sheaf of ${\Bbb C}$-algebras on $\breve{Y}(\Delta)$    and, by construction,
   $\breve{\varphi}^{-1}\breve{\cal G}$
      is a left $\breve{\varphi}^{-1}{\cal O}_{\breve{Y}(\Delta^\prime)}$-module.	
   Since
     $\breve{\varphi}^\sharp:{\cal O}_{\breve{Y}(\Delta^\prime)}
	    \rightarrow {\cal O}_{\breve{Y}(\Delta)}$
      renders ${\cal O}_{\breve{Y}(\Delta)}$
	  a two-sided $\breve{\varphi}^{-1}{\cal O}_{\breve{Y}(\Delta^\prime)}$-module,
	one can define further a left ${\cal O}_{\breve{Y}(\Delta)}$-module
	 $\breve{\varphi}^\ast\breve{\cal G}
	     :=  {\cal O}_{\breve{Y}(\Delta)}
		        \otimes_{\breve{\varphi}^{-1}{\cal O}_{\breve{Y}(\Delta^\prime)}}
				\breve{\varphi}^{-1}{\cal G}$
	 on $\breve{Y}(\Delta)$,
	named the {\it pullback} of $\breve{\cal G}$ under $\breve{\varphi}$.
	
  Similarly, for right modules and two-sded modules on $\breve{Y}(\Delta)$ and $\breve{Y}(\Delta^\prime)$.   
}\end{sdefinition}

\medskip

\begin{sexample} {\bf [ideal sheaf and structure sheaf of closed subscheme]}\; {\rm
 Recall Definition~2.2.12.
 Then,
   an ideal sheaf ${\cal I}_{\breve{Z}}$ on $\breve{Y}(\Delta)$ and
   the associated structure sheaf ${\cal O}_{\breve{Z}}$
   of a soft noncommutative closed subscheme $\breve{Z}$ of $\breve{Y}(\Delta)$
   are both two-sided ${\cal O}_{\breve{Y}(\Delta)}$-modules.
}\end{sexample}

\medskip

\begin{sdefinition}
{\bf [invertible ${\cal O}_{\breve{Y}(\Delta)}$-module/invertible sheaf/line bundle
          on $\breve{Y}(\Delta)$]}\;
{\rm
 (Continuing Definition~3.1.)
 A {\it left} (resp.\ {\it right}, {\it two-sided}\,)
   {\it invertible ${\cal O}_{\breve{Y}(\Delta)}$-module}
   (or {\it invertible sheaf} or {\it line bundle}) on $\breve{Y}(\Delta)$
  is a left (resp.\ right, two-sided) ${\cal O}_{\breve{Y}(\Delta)}$-module
   $$
     \breve{\cal L}\;
      =\;  (\{\breve{F}_\sigma\}_{\sigma\in\Delta},
            \{h_{\sigma\tau}\}_{\tau,\, \sigma\in\Delta,\, \tau\prec\sigma})
   $$
  such that
   $\breve{F}_\sigma\simeq \breve{R}_\sigma $ as left (resp.\ right, two-sided)
    $\breve{R}_\sigma$-modules for $\sigma\in\Delta$   and
   $h_{\sigma\tau}$ is an isomorphism of left (resp.\ right, two-sided) $\breve{R}_\tau$-modules
     for $\tau,\,\sigma\in\Delta$, $\tau\prec\sigma$.
}\end{sdefinition}

\medskip

\begin{slemma} {\bf [invertible ${\cal O}_{\breve{Y}(\Delta)}$-module: basic description]}\;
 {\rm (Continuing the notation from Sec.~2.2.)}
 Let $\breve{\cal L}$ be a (left) invertible ${\cal O}_{\breve{Y}(\Delta)}$-module.
 Then $\breve{\cal L}$ can be described by a collection
  $$
   \{c_{\sigma\tau}\breve{m}_{\sigma\tau}\}_{\tau\prec\sigma\in\Delta,\,\bullet}\;
    :=\;  \left\{
	         c_{\sigma\tau} \breve{m}_{\sigma\tau}
	                  \in {\Bbb C}^\times\!\breve{M}_\tau\,
	             \left|
				   \begin{array}{l}
				    \sigma,\, \tau\in\Delta,\,
	                        \tau\prec\sigma,\, \pi(\breve{m}_{\sigma\tau})\in\tau^\perp\cap M\,, \\
					 c_{\sigma\rho}\breve{m}_{\sigma\rho}
					    =   c_{\sigma\tau}c_{\tau\rho} \breve{m}_{\sigma\tau}\breve{m}_{\tau\rho}\;\;
					 \mbox{for}\; \rho\prec\tau\prec\sigma
				  \end{array}			
				\right.			
			\!\!\!\!\right\}\!.
  $$
 Here,
  $\breve{m}_{\sigma\sigma^\prime}$ acts on
    $\breve{F}_{\sigma\sigma^\prime}\simeq \breve{R}_{\sigma\sigma^\prime}$
	as a left $\breve{R}_{\sigma\sigma^\prime}$-module homomorphism
	given by `the multiplication by $\breve{m}_{\sigma\sigma^\prime}$ from the right'.
 In particular,
   $\breve{m}_{\tau\rho}\circ \breve{m}_{\sigma\tau}
      = \breve{m}_{\sigma\tau}\breve{m}_{\tau\rho}$.
 Two such collections
  $\{c_{\sigma\tau}\breve{m}_{\sigma\tau}\}_{\tau\prec\sigma\in\Delta,\,\bullet}$ and
  $\{c^\prime_{\sigma\tau}\breve{m}^\prime_{\sigma\tau}\}_{\tau\prec\sigma\in\Delta,\,\bullet}$
  define isomorphic invertible ${\cal O}_{\breve{Y}(\Delta)}$-modules
 if and only if there exists a collection
  $$
    \{c_\sigma\breve{m}_\sigma\}_{\sigma\in\Delta,\bullet}\;
	  := \{c_\sigma\breve{m}_\sigma \in{\Bbb C}^\times\breve{M}_\sigma\,|\,
	            \sigma\in\Delta,\, \pi(\breve{m}_\sigma) \in \sigma^\perp\cap M\}
  $$
  such that
   $$
     c^\prime_{\sigma\tau}\; =\;  c_{\sigma}^{-1}c_{\sigma\tau}c_\tau
	  \hspace{2em}\mbox{and}\hspace{2em}
      \breve{m}^\prime_{\sigma\tau}\;
	    =\; \breve{m}_\sigma^{-1}\breve{m}_{\sigma\tau}\breve{m}_\tau \,.
   $$
   
\end{slemma}	

\smallskip

\begin{proof}
 Observe
   that the multiplicative monoid of invertible elements of the monoid algebra ${\Bbb C}\langle\breve{M}\rangle$
     is given by ${\Bbb C}^\times\!\breve{M}$     and
   that the multiplicative monoid of invertible elements of the monoid algebra $\breve{R}_\sigma$, $\sigma\in\Delta$,
	 is contained in ${\Bbb C}^\times\breve{M}$  and
	 is given by ${\Bbb C}^\times\!\breve{M}_\sigma\cap\pi^{-1}(\sigma^\perp\cap M)$.
 It follows from Definition~3.3 that
  an invertible ${\cal O}_{\breve{Y}(\Delta)}$-module is determined by gluings of 	
  the collection $\{\breve{R}_\sigma\}_{\sigma\in\Delta}$ for each pair $\tau\prec\sigma$
  as $\breve{R}_\tau$-modules.
 This gives the gluing data
   $\{c_{\sigma\tau}\breve{m}_{\sigma\tau}\}_{\tau\prec\sigma\in\Delta,\bullet}$
   in the Statement.
 Different choices of isomorphisms $\breve{F}_\sigma\simeq \breve{R}_\sigma$
  (i.e.\ local trivializations of ${\cal L}$)
  define isomorphic invertible ${\cal O}_{\breve{Y}(\Delta)}$-modules.
   
\end{proof}
	
\medskip

\begin{sproposition} {\bf [existence of invertible sheaf after passing to softening]}\;
 Let $\breve{Y}(\Delta)$ be an ($n$-dimensional) noncommutative toric scheme associated to a fan $\Delta$.
 Recall the built-in inclusion $Y(\Delta)\subset \breve{Y}(\Delta)$ and
 let ${\cal L}$ be an invertible ${\cal O}_{Y(\Delta)}$-module on $Y(\Delta)$.
 Then,
  there exists a softening $\breve{Y}^\prime(\Delta)\rightarrow \breve{Y}(\Delta)$
      of $\breve{Y}(\Delta)$
   such that ${\cal L}$, now on $Y(\Delta)\subset \breve{Y}^\prime(\Delta)$,
    extends to an invertible ${\cal O}_{\breve{Y}^\prime(\Delta)}$-module
	on $\breve{Y}^\prime(\Delta)$.
\end{sproposition}
   
\medskip

\begin{proof}
 Under Assumption~2.2.2,
  $Y(\Delta)$ is smooth and thus ${\cal L}\simeq{\cal O}_{Y(\Delta)}(D)$
  for some ${\Bbb T}^n$-invariant Cartier divisor $D$ on $Y(\Delta)$   and
  can be specified by a collection $\{m_\sigma\in M\}_{\sigma\in\Delta(n)}$,
    where $m_\sigma\in M$ and
    $\Delta(n)$ is the collection of maximal cones in $\Delta$,
  that satisfies $m_\sigma^{-1} m_\tau\in (\sigma\cap\tau)^\perp\cap M$,
   for $\sigma, \tau\in\Delta(n)$.
 (Cf. [Fu: Sec.\ 3.4], with the monoid $M$ here presented multiplicatively for convenience.)
 For each $\tau\in\Delta(k)$, $k\le n-1$,
  fix an $\sigma\in\Delta(n)$ such that $\tau\prec\sigma$ and set $m_\tau = m_\sigma$.
 This extend $\{m_\sigma\}_{\sigma\in\Delta(n)}$ to
  a collection $\{m_\sigma\}_{\sigma\in\Delta}$
    such that $m_\sigma^{-1}m_\tau \in (\tau\cap \sigma)^\perp\cap M$.
    
 Let $\{\breve{M}_\sigma\}_{\sigma\in\Delta}$ be the underlying inverse $\Delta$-system of monoids
    associated to $\breve{Y}(\Delta)$     and,
 for each $\sigma\in\Delta$, fix an $\breve{m}_\sigma\in \breve{M}_\sigma$ 	
  such that $\pi(\breve{m}_\sigma)=m_\sigma$.
 Set
  $$
   \{\breve{m}_{\sigma\tau}\}_{\tau\prec\sigma\in\Delta}\;
      = \; \{\breve{m}_\sigma^{-1}\breve{m}_\tau \}_{\tau\prec\sigma\in\Delta}\,.
  $$
 Then,
  $$
   \pi(\breve{m}_{\sigma\tau})\;\in\; \tau^\perp\cap M\:\:\mbox{for $\tau\prec\sigma\in\Delta$}\,,
    \hspace{1em}\mbox{and}\hspace{1em}
   \breve{m}_{\sigma\tau}\,\breve{m}_{\tau\rho}\;=\; \breve{m}_{\sigma\rho}\:\:
    \mbox{for $\rho\prec\tau\prec\sigma\in\Delta$}\,.
  $$
 This is almost the gluing data in Lemma~3.4 for an invertible ${\cal O}_{\breve{Y}(\Delta)}$-module
  {\it except} that in general $\breve{m}_{\sigma\tau}\notin \breve{M}_\tau$.
 To remedy this, let
   $$
    \begin{array}{lclcl}
	 \breve{S}_\sigma   & =  & \mbox{the empty set} && \mbox{for $\sigma\in\Delta(n)$}\,,\\[1.2ex]
     \breve{S}_\tau       & =
	   & \langle\breve{m}_{\sigma\tau}\,|\, \tau\prec\sigma\in\Delta \rangle
	   && \mbox{for $\tau\in\Delta(k)$, $k\le n-1$}\,.	
    \end{array}
   $$
 This defines a collection of $\Delta$-indexed finitely generated submonoids of $\breve{M}$.
 It follows from Proposition~2.2.10 that
 one can augment $\{\breve{M}_\sigma\}_{\sigma\in\Delta}$
  to an inverse $\Delta$-system $\{\breve{M}^\prime_\sigma\}_{\sigma\in\Delta}$
  of submonoids of $\breve{M}$ such that $\breve{S}_\sigma\subset \breve{M}^\prime_\sigma$.
 Furthermore,
  since $\breve{S}_\sigma=$ empty set for $\sigma\in\Delta(n)$,
  it follows from the proof of Proposition~2.2.10 that
  it can be made that $\breve{M}^\prime_\sigma=\breve{M}_\sigma$ for $\sigma\in\Delta(n)$.
 Thus,
  $\{\breve{M}^\prime\}_{\sigma\in\Delta}$ defines a softening $\breve{Y}^\prime(\Delta)$
    of $\breve{Y}(\Delta)$    and
  ${\cal L}$ extends to an invertible ${\cal O}_{\breve{Y}^\prime(\Delta)}$-module
   $\breve{\cal L}$ on $\breve{Y}^\prime(\Delta)$.
 
\end{proof}
   
\bigskip

Since a soft noncommutative toric scheme associated to $\Delta$ always exists
   (cf.\  Theorem~2.2.5,  Proposition~2.2.9, Definition~2.2.12),
 Proposition~3.5 can be rephrased as:

\bigskip

\noindent
{\bf Proposition 3.5$^\mathbf{\prime}$. [existence of invertible sheaf]}\;
{\it
 Let ${\cal L}$ be an invertible sheaf on the smooth toric variety $Y(\Delta)$.
 Then there exists a soft noncommutative toric scheme $\breve{Y}(\Delta)\supset Y(\Delta)$ associated to $\Delta$
 such that ${\cal L}$ extends to an invertible sheaf $\breve{\cal L}$ on $\breve{Y}(\Delta)$.
} 

\bigskip

Note that the pullback of an invertible sheaf under a morphism of noncommutative toric schemes is an invertible sheaf.
In particular,
 any invertible sheaf on $\breve{Y}(\Delta)$ pulls back to invertible sheaves on softenings of $\breve{Y}(\Delta)$.

\bigskip

\begin{flushleft}
{\bf Twisted sections of an invertible sheaf on $\breve{Y}(\Delta)$}
\end{flushleft}
Let ${\cal L}$ be an invertible sheaf on the $n$-dimensional smooth toric variety $Y(\Delta)$.
Then ${\cal L}$ is isomorphic to a $\mathbf{T}$-Cartier divisor
  ${\cal O}_{Y(\Delta)}(D)={\cal O}_{Y(\Delta)}(\sum_{\tau\in\Delta(1)} a_\tau D_\tau)$,
	   where $\Delta(1)\subset \Delta$ is the set of rays (i.e. $1$-dimensional cones) in $\Delta$,
            $D_\tau$	is the $\mathbf{T}$-invariant Weil divisor on $Y(\Delta)$ associated to $\tau\in\Delta(1)$, and
	        $a_\tau\in{\Bbb Z}$.
The divisor $D$ determines an $m_\sigma\in M$ for each maximal cone $\sigma\in\Delta(n)$:
  $$
     \langle m_\sigma, v_\tau \rangle\;=\; -\,a_\tau\,,\hspace{1em}
	 \mbox{for all  $\tau\in\Delta(1)$ such that $\tau\prec\sigma$}\,.
  $$
 Here $v_\tau\in\tau\cap N$ is the first lattice point on the ray $\tau$.
Writing the commutative monoid $M$ multiplicatively as a multiplicatively closed subset of $R_\mathbf{0}={\Bbb C}[M]$
     and passing to the isomorphism ${\cal L}\simeq {\cal O}_{Y(\Delta)}(D)$,
  then
    ${\cal L}(U_\sigma)= R_\sigma\cdot m_\sigma = {\Bbb C}[\sigma^\vee\cap M]\cdot m_\sigma$,
          $\sigma\in\Delta(n)$.
It follows that, if defining the (possibly empty) {\it polytope} in $M_{\Bbb R}$ {\it associated to $D$}
 $$
   P_D\; :=\;  \bigcap_{\sigma\in\Delta(n)} \sigma^\vee\cdot m_\sigma\;
      =\;    \{u\in M_{\Bbb R}\,|\, \langle u, v_\tau \rangle \ge -a_\tau,\;\mbox{for all $\tau\in \Delta(1)$}\}\;
	  \subset\; M_{\Bbb R},
 $$
 then
 the ${\Bbb C}$-vector space $H^0(Y(\Delta), {\cal L})$ of sections of ${\cal L}$ is realized as
  $$
    H^0(Y(\Delta), {\cal O}_{Y(\Delta)}(D))\; =\;
	 \bigoplus_{u\in P_D\cap M}{\Bbb C}\cdot u\,.
  $$
(Cf.\ [Fu: Sec.~3.3 \& Sec.~3.4].)			

\bigskip

\begin{sdefinition} {\bf [twisted section of invertible sheaf]}\; {\rm
 Let
   $\breve{\cal L}
      = (\{\breve{F}_\sigma\}_{\sigma\in\Delta}, \{h_{\sigma\tau}\}_{\tau\prec\sigma\in\Delta} )$
   be a (left) invertible sheaf on a soft noncommutative toric scheme
   $\breve{Y}(\Delta)=\{\breve{R}_\sigma\}_{\sigma\in\Delta}$.
 Two local sections $\breve{s}_1,\,\breve{s}_2\in\breve{F}_\sigma$, $\sigma\in\Delta$,
   are said to be {\it equivalent}, in notation, $\breve{s}_1\sim \breve{s}_2$,
   if there exists an invertible element $\breve{r}\in \breve{R}_\sigma^{\,\ast}$
    such that $\breve{s}_1 = \breve{r}\breve{s}_2$.
  This is clearly an equivalence relation on $\breve{F}_\sigma$, $\sigma\in\Delta$.
 A {\it twisted section} of $\breve{\cal L}$ is a collection
    $$
	   \breve{s}= \{\breve{s}_\sigma\}_{\sigma\in\Delta}
	$$
	of local sections $\breve{s}_\sigma\in \breve{F}_\sigma$ of $\breve{\cal L}$
  such that
    $$
	  h_{\sigma\tau}(\breve{s}_\sigma)\; \sim\; \breve{s}_\tau\,,\;\;
	  \mbox{for $\tau\prec\sigma\in\Delta$}\,.
	$$
	
 Fix a local trivializing isomorphism $\breve{\cal F_\sigma}\simeq \breve{R}_\sigma$
   of left $\breve{R}_\sigma$-modules, $\sigma\in\Delta$.
 Then, $\breve{s}$ is specifoed uniquely by $\{\breve{r}_\sigma\}_{\sigma\in\Delta}$,
   with $\breve{r}\in\breve{R}_\sigma$.
 We will called $\{\breve{r}_\sigma\}_{\sigma\in\Delta}$
   the {\it presentation} of $\breve{s}$ {\it with respect to the local trivialization} of $\breve{\cal L}$.
 Different choices of local trivializations of $\breve{\cal L}$ give presentations of $\breve{s}$ that differ
  by a right multiplication of unit (i.e.\ invertible element) chart by chart.
}\end{sdefinition}

\medskip

\begin{slemma} {\bf [twisted section under pullback]}\;
 Let
   $\phi: \breve{Y}^\prime(\Delta)\rightarrow \breve{Y}(\Delta)$
       be a morphism of soft noncommutative toric schemes,
   $\breve{\cal L}$ be an invertible sheaf on $\breve{Y}(\Delta)$,    and
   $\breve{s}$ be a twisted section of $\breve{\cal L}$.
 Then,
   $\phi^\ast\breve{s}$ is a twisted section
     of the pullback $\phi^\ast\breve{\cal L}$ of $\breve{\cal L}$ to $\breve{Y}^\prime(\Delta)$.
   In particular, a twisted section pulls back to a twisted section under softening.
\end{slemma}

\medskip

\begin{proof}
 This follows from the fact that $\phi^\sharp$ takes an invertible element to an invertible element chart by chart.

\end{proof}

\medskip

\begin{sproposition} {\bf [extension to twisted section after passing to softening]}\;
 Let
   $\breve{\cal L}$ be an invertible sheaf on $\breve{Y}(\Delta)$,
   ${\cal L}$ be the restriction of $\breve{\cal L}$ to $Y(\Delta)\subset \breve{Y}(\Delta)$,    and
   $s_1,\,\cdots\,,\,s_k\in  H^0(Y(\Delta), {\cal L})$ be sections of ${\cal L}$ on $Y(\Delta)$.
 Then, there exists a softening $\breve{Y}^\prime(\Delta)\rightarrow \breve{Y}(\Delta)$
   such that all $s_1\,\cdots\,,\, s_k$ of ${\cal L}$ extend to twisted sections of the pullback $\breve{\cal L}^\prime$
   of $\breve{\cal L}$ to $\breve{Y}^\prime(\Delta)$.
\end{sproposition}

\medskip

\begin{proof}
 Present ${\cal L}$ as a $\mathbf{T}$-Cartier divisor ${\cal O}_{Y(\Delta)}(D)$,
   recall the polytope $P_D$ in $M_{\Bbb R}$ associated to $D$,    and
   denote the section of ${\cal L}$ corresponding to $m\in P_D\cap M$ by $s_m$.
 Since each $s_i$ is a ${\Bbb C}$-linear combination of finitely many $s_m$'s, $m\in P_D\cap M$,
  it follows from Lemma~3.7 that
  we only need to prove the proposition for a section of ${\cal L}$ in the form $s_{m^\prime}$
  for some $m^\prime\in P_D\cap M$.

 Write out more explicitly:
    $\breve{Y}(\Delta)=\{{\Bbb C}\langle \breve{M}_\sigma \rangle\}_{\sigma\in\Delta}$  and
	$\breve{\cal L}
	   = (\{ {\Bbb C}\langle \breve{M}_\sigma \rangle\}_{\sigma\in\Delta},
	           \{c_{\rho\tau}\breve{m}_{\sigma\tau}\}_{\tau\prec\sigma\in\Delta})$
	   under a local trivialization (cf.\ Lemma~3.4).
 Recall the collection $\{m_\sigma\}_{\sigma\in \Delta(n)}\subset M$ associated to $D$
   from the beginning of the theme.
 For any $\tau\in \Delta-\Delta(n)$, fix an $\sigma\in\Delta(n)$ such that $\tau\prec\sigma$ and
   assign $m_\tau$ to be $m_\sigma$.
 This enlarges the collection $\{m_\sigma\}_{\sigma\in\Delta(n)}$
   to a collection $\{m_\sigma\}_{\sigma\in\Delta}$
   with the property that
   $$
     m_\sigma\,m_\tau^{-1}\;\in\; \tau^\perp\cap M\,,\:\:\mbox{for $\tau\prec\sigma\in\Delta$}\,.
   $$
 Here, we write the operation of the commutative monoid $M$ multiplicatively.
   
 Choose $\breve{m}_\sigma,\,\breve{m}^\prime_\sigma\in \breve{M}$
   such that
    $$
	  \pi(\breve{m}_\sigma)\;=\; m_\sigma, \hspace{2em}
	  \pi(\breve{m}^\prime_\sigma)\;=\; m^\prime, \hspace{2em}
	  \breve{m}^\prime_\sigma\;\in\; \breve{M}_\sigma\cdot \breve{m}_\sigma,
	$$
   for all $\sigma\in\Delta$.
 Since $\pi:\breve{M}\rightarrow M$ restricts to a surjection
   $\breve{M}_\sigma\cdot \breve{m}_\sigma\rightarrow M_\sigma\cdot m_\sigma$  and
   $m^\prime\in M_\sigma\cdot m_\sigma$, such a choice always exists.
 It follows that
   $\breve{m}_\sigma^\prime= \breve{r}^\prime_\sigma\cdot \breve{m}_\sigma$
   for some $\breve{r}_\sigma^\prime\in\breve{M}_\sigma\subset \breve{R}_\sigma$, $\sigma\in\Delta$.
 Consider now the collection
   $$
       \breve{s}^\prime :=\{\breve{r}_\sigma^\prime\}_{\sigma\in\Delta}
   $$
   of local sections of $\breve{\cal L}$.
 By construction, it recovers the section $s_{m^\prime}$ of ${\cal L}$
   when restricted to $Y(\Delta)\subset \breve{Y}(\Delta)$.
 For $\tau\prec\sigma\in\Delta$,
   $$
     h_{\sigma\tau}(\breve{r}_\sigma^\prime)\;
	   =\; \breve{r}_\sigma^\prime\cdot c_{\sigma\tau}\breve{m}_{\sigma\tau}\,
	   \in\, \breve{R}_\tau .
   $$
 As elements in $\breve{M}$,
   $$
     \breve{r}_\sigma^\prime\, \breve{m}_{\sigma\tau}\;=\;
	   (\breve{r}_\sigma^\prime\,\breve{m}_{\sigma\tau}\,
	                        \mbox{$\breve{r}_\tau^\prime$}^{-1})\cdot \breve{r}_\tau^\prime\,.
   $$
  Since
   $$
      \pi(\breve{r}_\sigma^\prime\,\breve{m}_{\sigma\tau}\,\mbox{$\breve{r}_\tau^\prime$}^{-1})\;
	 =\; (m_\sigma\, m_\tau^{-1})\cdot \pi(\breve{m}_{\sigma\tau})\;
	 \in\;\tau^\perp\cap M\,,
   $$
   $\pi(\breve{r}_\sigma^\prime\,\breve{m}_{\sigma\tau}\,\mbox{$\breve{r}_\tau^\prime$}^{-1})$
    is invertible in ${\Bbb C}[M_\tau]$.
 Thus, $\breve{s}$ is almost a twisted section of $\breve{\cal L}$ {\it except} that
   the twisting factor
      $\breve{r}_\sigma^\prime\,\breve{m}_{\sigma\tau}\,\mbox{$\breve{r}_\tau^\prime$}^{-1}$
	may not yet be in $\breve{M}_\tau$.
 The same argument  as in the proof of Proposition~3.5 and applying Proposition~2.2.10
  show that
     there exists a softening\\     $\phi: \breve{Y}^\prime(\Delta)\rightarrow \breve{Y}(\Delta)$
	 such that the pullback collection $\phi^\ast \breve{s}^\prime$ is indeed a twisted section
	 of the pullback invertible sheaf $\phi^\ast\breve{\cal L}$ on $\breve{Y}^\prime(\Delta)$.
 This proves the proposition.

\end{proof}

\bigskip

Since a soft noncommutative toric scheme associated to $\Delta$
   with an invertible sheaf that extends ${\cal L}$ on $Y(\Delta)$ exists
   (cf.\ Proposition~3.5$^\prime$),
Proposition~3.8 is equivalent to:

\bigskip
				
\noindent
{\bf Proposition 3.8$^\mathbf{\prime}$. [extension to twisted section]}\; {\it				
 Let
   ${\cal L}$	be an invertible sheaf on $Y(\Delta)$   and
   $s_1,\,\cdots\,,\, s_k\in H^0(Y(\Delta), {\cal L})$ be sections of ${\cal L}$.
 Then, there exists a soft noncommutative toric scheme $\breve{Y}(\Delta)$
    with an invertible sheaf $\breve{\cal L}$ that restricts to ${\cal L}$ on $Y(\Delta)\subset \breve{Y}(\Delta)$
  such that $s_i$ extends to a twisted section $\breve{s}_i$ of $\breve{\cal L}$, for $i=1,\,\ldots\,,\, k$.
} 

\bigskip

\begin{flushleft}
{\bf Soft noncommutative closed subschemes of $\breve{Y}(\Delta)$
           associated to twisted sections of invertible sheaves on $\breve{Y}(\Delta)$}
\end{flushleft}
Continuing the notation and discussion of the previous theme.
Let $\breve{s}_i=\{\breve{s}_{i,\,\sigma}\}_{\sigma\in\Delta}$
  be a (non-zero) twisted section of an invertible sheaf $\breve{\cal L}_i$ on $\breve{Y}(\Delta)$,
  for $i=1,\,\ldots\,,\, k$.
With respect to a local trivialization of each $\breve{\cal L}_i$,
 $\breve{s}_i$ has a presentation $\{\breve{r}_{i,\,\sigma}\}_{\sigma\in\Delta}$,
 with $\breve{r}_{i,\,\sigma}\in\breve{R}_\sigma$, $\sigma\in\Delta$.
A different local trivialization of  $\breve{\cal L}_i$ gives rise to a different presentation of $\breve{s}_i$
 of the form $\{\breve{r}_{i,\,\sigma} \breve{u}_{i,\,\sigma}\}_{\sigma\in\Delta}$,
 where $\breve{u}_{i,\,\sigma}$ is a unit (i.e.\ invertible element) of $\breve{R}_\sigma$.
Denote $\Gamma = \{\breve{s}_1,\,\cdots\,,\,\breve{s}_k\}$.
It follows that the two-sided ideal
 $$
   \breve{I}_{\Gamma,\,\sigma}\; :=\; (\breve{r}_{1,\,\sigma},\,\cdots\,,\,\breve{r}_{k,\,\sigma})
 $$
 of $\breve{R}_\sigma$, $\sigma\in\Delta$, associated to $\Gamma$
 is independent of the local trivialization of $\breve{\cal L}_i$, $i=1,\,\ldots\,,\,k$.
This defines an ideal sheaf
 ${\cal I}_\Gamma :=\{\breve{I}_{\Gamma,\,\sigma}\}_{\sigma\in\Delta}$ on $\breve{Y}(\Delta)$
 and hence a soft noncommutative closed subscheme $\breve{Z}_\Gamma$ of $\breve{Y}(\Delta)$.

\bigskip

\begin{sdefinition} {\bf [soft noncommutative closed subscheme associated to twisted sections]}\; {\rm
 $\breve{Z}_\Gamma$ is called
     the {\it soft noncommutative closed subscheme of $\breve{Y}(\Delta)$ associated to $\Gamma$}.
 When $\Gamma=\{\breve{s}\}$,
   $\breve{Z}_{\breve{s}}:=  \breve{Z}_\Gamma$
    is called the {\it soft noncommutative hypersurface} of $\breve{Y}(\Delta)$ associated to $\breve{s}$.
}\end{sdefinition}

\bigskip

\begin{sdefinition} {\bf [soft noncommutative scheme/space (via toric geometry)]}\; {\rm
 A soft noncommutative toric scheme $\breve{Y}(\Delta)$ associated to a fan $\Delta$ or
   a soft noncommutative closed subscheme $\breve{Z}$ of a $\breve{Y}(\Delta)$ for some $\Delta$
  will be called a {\it soft noncommutative scheme via toric geometry},
    or simply a {\it soft noncommutative scheme}, or even a {\it soft noncommutative space}.
 The {\it structure sheaf} ${\cal O}_{\breve{Y}}$ of such a  noncommutative space $\breve{Y}$
    is defined to the inverse $\Delta$-system of ${\Bbb C}$-algebras that describes $\breve{Y}$.   	
}\end{sdefinition}

\bigskip

It is this class of noncommutative spaces that we will study further in sequels.
Definition~2.2.12 and Definition~3.1 can be generalized routinely to soft noncommutative schemes and their modules.
A straightforward generalization of Proposition 3.8$^\prime$
     to a finite collection of invertible sheaves with sections
    $\{({\cal L}_1, s_1),\,\,\cdots,\, ({\cal L}_k, s_k)\}$ on $Y(\Delta)$
  implies that:			
			
\bigskip

\begin{scorollary} {\bf [commutative complete intersection subscheme]}\;
 Any complete intersection subscheme $Z$ of a smooth toric variety $Y(\Delta)$  embeds
   as a commutative closed subscheme of a soft noncommutative scheme $\breve{Z}$
   in a $\breve{Y}(\Delta)\,(\supset Y(\Delta))$
   such that $\breve{Z}\cap Y(\Delta)=Z$.
  In particular, this applies to complete intersection Calabi-Yau spaces in a smooth toric variety.
\end{scorollary}

\bigskip			

Indeed it is the attempt to make this corollary holds that the notion of soft noncommutative schemes via toric geometry is discovered.
Let us spell out a guiding question behind the NCS spin-off of the D-project before leaving this section:

\bigskip

\begin{squestion} {\bf [soft noncommutative Calabi-Yau space and mirror]}\; {\rm
 What is the correct notion/definition of {\it soft noncommutative Calabi-Yau spaces}
   in the current context?
 From pure mathematical generalization of the commutative case?
 From world-volume conformal invariance or supersymmetry of D-branes (cf.\ Sec.~4.2)?
 What is the {\it mirror symmetry} phenomenon in this context?																	
}\end{squestion}

\bigskip
\bigskip

\section{(Dynamical) D-branes on a soft noncommutative space}

The notion of
  morphisms from an Azumaya scheme with a fundamental module over ${\Bbb C}$
  to a soft noncommutative scheme over ${\Bbb C}$
 is developed in this section.

\bigskip
		
\subsection{Review of Azumaya/matrix schemes with a fundamental module over ${\Bbb C}$}
		
The notion of {\it Azumaya schemes} (or equivalently {\it matrix schemes})  {\it with a fundamental module}
   and {\it morphisms} therefrom were developed
 in [L-Y1]  (D(1)) and [L-L-S-Y] (D(2)) in the realm of Algebraic Geometry   and
 in [L-Y2] (D(11.1)) and [L-Y3] (D(11.2)) in the realm of Differential Geometry
 to capture the Higgsing/unHiggsing feature of D-branes when they coincide or separate     and
 to address dynamical D-branes moving in a space-time, [L-Y4] (D(13.1))and [L-Y5] (D(13.3)).
The generalization of such notion to cover fermionic D-branes in Superstring Theory has been a major focus
 in more recent years; cf. the SUSY spin-off of the D-project, e.g.\ [L-Y6] (SUSY(2.1)).
The key setup of Azumaya/matrix schemes in the realm of Algebraic Geometry that is relevant to the current work
 is reviewed below for the terminology and notation.
Readers are referred to the above works for more details and
 to [Liu] for a review of the first four years (fall 2007 -- fall 2011) of the D-project.

\bigskip

\begin{definition}  {\bf [D-brane world volume: Azumaya/matrix scheme over ${\Bbb C}$]}\; {\rm
 Let
   $X$ be a (Noetherian) scheme over ${\Bbb C}$     and
   ${\cal E}$ be a locally free ${\cal O}_X$-module of rank $r$.
 Then
  the sheaf
   $\Endsheaf_{{\cal O}_X}({\cal E}):= \Homsheaf_{{\cal O}_X}({\cal E},\,{\cal E})$
   of endomorphisms of ${\cal E}$ is  a noncommutative ${\cal O}_X$-algebra
   with center ${\cal O}_X$.
 The fiber of $\Endsheaf_{{\cal O}_X}({\cal E})$ over each ${\Bbb C}$-point on $X$
   is isomorphic to the {\it Azumaya algebra} (or interchangeably {\it matrix algebra})
   of $r\times r$ matrices over ${\Bbb C}$.
 ${\cal E}$ is in the fundamental representation of $\Endsheaf_{{\cal O}_X}({\cal E})$.
 The new ringed space from the enhanced structure sheaf, with its fundamental representation encoded,
  $$
    X^{\!A\!z}:= (X, {\cal O}_X^{A\!z}:= \Endsheaf_{{\cal O}_X}({\cal E}), {\cal E})\,.
  $$
  is called an {\it Azumaya scheme} (or interchangeably {\it matrix scheme}) {\it with a fundamental module}.
 
 The ringed space $(X,{\cal O}_X^{A\!z})$ serves as the {\it world-volume} of a D-brane    and
 the fundamental (left) ${\cal O}_X^{A\!z}$-module ${\cal E}$ serves as the {\it Chan-Paton sheaf}
   on the world-volume.
 For a complete description of the data on the world-volume of a D-brane,
   there is also a {\it connection $\nabla$} on ${\cal E}$    and, in this case,
   one may introduce a {\it Hermitian structure} on ${\cal E}$ and require that $\nabla$ be unitary.
 These two additional structures on $X^{\!A\!z}$ are irrelevant to the discussion of the current work
   and hence will be dropped.
}\end{definition}
  
\bigskip

As the category of (left) ${\cal O}_X^{A\!z}$-modules is equivalent to the category of
 ${\cal O}_X$-modules (cf.\ Morita equivalence),
 the noncommutativity of $X^{\!A\!z}$ may look only of a mild kind
 --- yet important for the correct mathematical description of dynamical D-branes in string theory.

\bigskip

\begin{example} {\bf [Azumaya/matrix point with fundamental module$/{\Bbb C}$]}\; {\rm
 To get a feeling of such a mildly noncommutative space, consider the simplest case:
 $X =$ a ${\Bbb C}$-point $p$ and
  $$
   p^{A\!z}\;:= (p, M_r({\Bbb C}), {\Bbb C}^{\oplus r})\,,
  $$
  where $M_r({\Bbb C})$ is the algebra of $r\times r$ matrices over ${\Bbb C}$.
 One may attempt to assign some sensible topology ``$\Space(M_r({\Bbb C}))$"
  to the noncommutative ${\Bbb C}$-algebra, in the role of $\Spec(R)$ to a commutative ring $R$.
 However, all the existing methods only lead to that ``$\Space(M_r({\Bbb C}))$" $=\{p\}$.
 Thus, underlying topology is the wrong direction to understand the geometry
   behind the Azumaya point (interchangeably, matrix point) $p^{A\!z}$.

 Now, the D-brane world-volume is only half of the story.
 The other half is how the world-volume gets mapped into a target space-time $Y$
 (i.e.\ {\it morphisms} from Azumaya schemes to $Y$).
 With this in mind,
  if $A\subset M_r({\Bbb C})$ is a ${\Bbb C}$-subalgebra  and
   we do know how to make sense of $\Space(A)$ and morphisms $\Space(A)\rightarrow Y$
  --- for example, $A$ is commutative ${\Bbb C}$-subalgebra of $M_r({\Bbb C})$ and
      $Y$ is an ordinary Noetherian commutative scheme in Commutative Algebraic Geometry ([Ha]) ---
 then we should expect a morphism $p^{A\!z}\rightarrow Y$ from the composition
  $$
    ``\Space(M_r({\Bbb C}))"\; \longrightaarrow\; \Space(A)\; \longrightarrow Y\,.
  $$
 Here, $\longrightaarrow$ indicates a {\it dominant morphism} since $A\hookrightarrow M_r({\Bbb C})$.

 It follows that
  the correct way to unravel the geometry behind $p^{A\!z}$ is through ${\Bbb C}$-subalgebras
  of $M_r({\Bbb C})$.
 For example, for $r\ge 2$, $M_r({\Bbb R})$ contains ${\Bbb C}$-subalgebra in product form
  $A=A_1\times A_2$.
  This means $p^{A\!z}$ has hidden disconnectivity.
 Indeed, the consideration of commutative ${\Bbb C}$-subalgebras of $M_r({\Bbb C})$ alone
  is enough to indicate that $p^{A\!z}$ has a very rich geometry behind.
 Cf.\ [Liu: {\sc Figure} 1-2, {\sc Figure} 3-1, and {\sc Figure} 3-2].
}\end{example}

\bigskip

Example~4.1.2 motivates the following definition:

\bigskip
		
\begin{definition} {\bf [surrogate of $X^{\!A\!z}$]}\; {\rm		
 (Continuing Definition~4.1.1.)
 Let ${\cal A}\subset {\cal O}_X^{A\!z}$ be an ${\cal O}_X$-subalgebra.
 Then, the ringed space
  $$
     X_{\!\cal A}\; :=\;  (X,\,{\cal A} )
  $$
 is called a {\it surrogate} of $X^{\!A\!z}$.
 When ${\cal A}$ is commutative, $X_{\!\cal A}$ is simply $\boldSpec({\cal A})/X$.
 The inclusion ${\cal O}_X\subset {\cal A}$ specifies
   a dominant morphism $X_{\!\cal A}\rightarrow X$ over $X$
  while the inclusion ${\cal A}\subset {\cal O}_X^{A\!z}$ specifies a dominant morphism
   $X^{\!A\!z}\rightarrow X_{\!\cal A}$ over $X$.
 Cf.\ {\sc Figure} 4-1-1.\footnote{Reproduced from:\,
			 {\sl Chien-Hao Liu}  and {\sl Shing-Tung Yau},
			 {\it D-branes and synthetic/$C^\infty$-algebraic symplectic/ calibrated geometry, I.
			         Lemma on a finite algebraicness property of smooth maps from Azumaya/matrix manifolds},
			arXiv:1504.01841 [math.SG] (D(12.1)), {\sc Figure} 1-2.
                           } 
 %
 
 \bigskip
 
\begin{figure}[htbp]
 \bigskip
  \centering
  \includegraphics[width=0.80\textwidth]{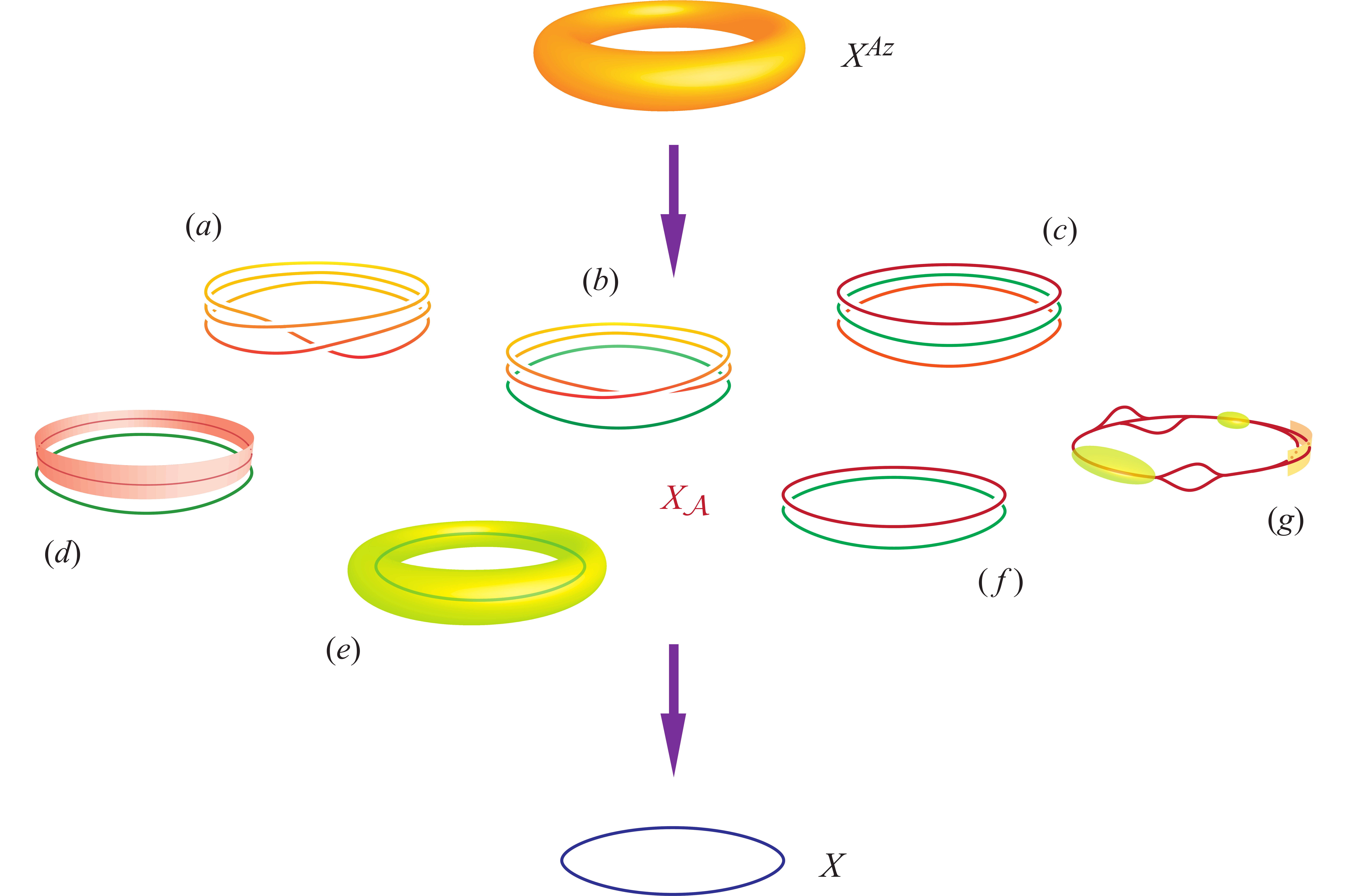}
  
  \bigskip
  \bigskip
 \centerline{\parbox{13cm}{\small\baselineskip 12pt
  {\sc Figure}~4-1-1.
 The Azumaya/matrix scheme $X^{\!A\!z}$ is indicated by a noncommutative cloud sitting over  $X$.
  In-between are illustrated examples of surrogates (a), (b), (c), (d), (e), (f), (g)  of $X^{\!A\!z}$.
  The vertical arrows
    $X^{\!A\!z}\rightarrow X_{\!\cal A}$ and $X_{\!\cal A}\rightarrow X$
	are the built-in dominant morphisms associated to the inclusions
	${\cal O}_X\subset {\cal A}\subset {\cal O}_X^{A\!z}$ of ${\cal O}_X$-algebras.	
  }}
\end{figure}	

}\end{definition}

\bigskip
\bigskip
 
\subsection{Morphisms from an Azumaya scheme with a fundamental module
       to a soft noncommutative scheme $\breve{Y}$: (Dynamical) D-branes on $\breve{Y}$}

We study in the subsection
  the notion of morphisms from an Azumaya scheme with a fundamental module
   to a soft noncommutative scheme $\breve{Y}$.
The design is guided by the requirement to reproduce the Higgsing/unHiggsing phenomenon
  when D-branes collide/coincide or fall apart/separate in a space-time.

\bigskip

\begin{flushleft}
{\bf Idempotents and gluing of quasi-homomorphisms}
\end{flushleft}
\begin{definition} {\bf [quasi-homomorphism of ${\Bbb C}$-algebras]}\; {\rm
 Given two (generally noncommutative) ${\Bbb C}$-algebras $R$ and $A$,
  a {\it quasi-homomorphism} from $R$ to $A$ is a ${\Bbb C}$-vector-space homomorphism $f:R\rightarrow A$
  such that $f(r_1r_2)= f(r_1)f(r_2)$.
 Note that it is not required that $f(1_R)=1_A$, where $1_R$  and $1_A$ are the identity of $R$ and $A$ respectively.
 Since $f(1_R)f(1_R)=f(1_R)$, $f(1_R)$ is an idempotent of $A$.
 Since $[f(1_R), f(r)]=0$ for all $r\in R$, $f(R)\subset \Centralizer(f(1_R))$.
 A ${\Bbb C}$-algebra quasi-homomorphism $f: R\rightarrow S$ with $f(1_R)=1_S$
  is called as usual a ${\Bbb C}$-algebra homomorphism.
 In the other extreme, the zero-map $R\rightarrow 0\in A$ is a quasi-homomorphism.
}\end{definition}

\bigskip

Let $A,\,R,\,S$ be ${\Bbb C}$-algebras with a fixed $\Bbb C$-algebra homomorphism $h:R\rightarrow S$.
It is standard that two ${\Bbb C}$-algebra homomorphisms $f_R:R\;\rightarrow A$ and $f_S: S \rightarrow A$
 are defined to be glued under $h$ if $f_R=f_S  \circ h$.
But if, instead, $f_R$ and $f_S$ are only ${\Bbb C}$-algebra quasi-homomorphisms, it takes some thought
 as to what the ``correct" definition of `$f_R$ and $f_S$ glue under $h$' should be.
As a generally noncommutative ${\Bbb C}$-algebra,
 it can happen that $A$ is not a product of ${\Bbb C}$-algebras and yet still contains idempotents other than $0$ and $1$.
In terms of the objects in the opposite category of the category of ${\Bbb C}$-algebras,
the geometry $\Scheme(A)$ behind the ${\Bbb C}$-algebra $A$ could have hidden disconnectivity
  (cf.\ Definition~2.1.8 and Example~2.1.9)
  and in terms of geometry
 it can happen that some connected components in the hidden disconnectivity of $\Scheme(A)$ is mapped to
   the complement of the image of $\Scheme(S)$ in $\Scheme(R)$ under $h$.
Since such hidden disconnectivity of $\Scheme(A)$ is inherited from idempotents of $A$ other than $0$ and $1$,
 we need to keep track of $f_R(1_R)$ and $f_S(1_S)$,
   allowing the situation $f_R(1_R)\ne f_S(1_S)$, as well
 when considering gluing of ${\Bbb C}$-algebra quasi-homomorphisms $f_R$ and $f_S$ under $h$.

\bigskip

\begin{definition} {\bf [subordination relation of idempotents]}\; {\rm
 Let $A$ be a ${\Bbb C}$-algebra and $e_1,\, e_2\in A$ are idempotents: $e_1^2=e_1$, $e_2^2=e_2$.
 We say that
   {\it $e_1$ is subordinate to $e_2$}, in notation $e_1\prec e_2$, or equivalently
  {\it $e_2$ is superior to $e_1$}, in notation $e_2\succ e_1$,
  if
   $$
     e_1\;\ne\; e_2   \hspace{2em}\mbox{and}\hspace{2em}	
     e_1e_2\; =\; e_2e_1\; =e_1\,.
   $$
 Denote `$e_1\prec e_2$ or $e_1=e_2$' by $e_1\preccurlyeq e_2$ or equivalently $e_2\succcurlyeq e_1$.
 Note that $0\preccurlyeq e \preccurlyeq 1$ for any idempotent $e$ of $A$.
}\end{definition}

\bigskip

When
 $e_1\prec e_2$, $e_2^\prime:= e_2-e_1\in A$ is also an idempotent subordinate to $e_2$ and
   it satisfies $e_2^\prime e_1=e_1e_2^\prime=0$.
 Thus, $e_2=e_1 + e_2^\prime$ gives an orthogonal decomposition of $e_2$ into subordinate idempotents.
 
\bigskip

\begin{definition} {\bf [gluing  of quasi-homomorphisms]}\; {\rm
 Let
  $A,\,R,\,S$ be ${\Bbb C}$-algebras with a fixed $\Bbb C$-algebra homomorphism $h:R\rightarrow S$    and
  $f_R:R\;\rightarrow A$ and $f_S: S \rightarrow A$ be quasi-homomorphisms of ${\Bbb C}$-algebras.
 We say that {\it $f_R$ and $f_S$ glue under $h$}, denoted $f_R \stackrel{h}{\rightsquigarrow}f_S$,   if
  $$
	f_S(1_S)\preccurlyeq f_R(1_R)\,,\;\;
    f_R(R)\subset \Centralizer(f_S(1_S))\,,\;\; \mbox{and}\;\;
	f_S(1_S)\cdot f_R     = f_R \cdot f_S(1_S)    =  f_S  \circ h\,.
  $$
 In terms of a commutative diagram:
  $$
    f_S(1_S) \preccurlyeq f_R(1_R)\,,\;\;
    f_R(R) \subset \Centralizer(f_S(1_S))\,,\;\; \mbox{and}
	\hspace{-1.6em}
    \xymatrix{
	 &  R \ar[d]_-h \ar[rrr]^-{f_R}&&& A \ar[d]^-{f_S(1_S)\,\cdot\,\bullet}_-{\bullet\,\cdot\,f_S(1_S)\,,} \\
	 &  S \ar[rrr]^-{f_S}                   &&& A     & \hspace{1em}.
	}
  $$
}\end{definition}

\medskip

\begin{explanation} {\bf [gluing of quasi-homomorphisms - geometry behind]} {\rm
 Refine the commuting diagram in Definition~4.2.3 to
  $$
   \xymatrix{
    &  R \ar[d]_-h \ar[rrr]^-{f_R}
	  &&& f_R(R)\; \ar[d]^-{=\; f_S(1_S)\,\cdot\,\bullet\;=:\; g}_-{\bullet\,\cdot\,f_S(1_S)}
	                                 \ar@{^{(}->}[rr]
      &&    A   \\
	&  S \ar[rrr]^-{f_S}                   &&& f_S(S)\; \ar@{^{(}->}[rr]         && A      & \hspace{-4em}.
   }
  $$
 Then, in addition to $h$,
  $f_S$ and $f_S$ are now homomorphisms of ${\Bbb C}$-algebras and
  so is $g$, since $f_S(1_S)\preccurlyeq f_R(1_R)$ in $A$,
  while the two horizontal inclusions are only quasi-homomorphisms of ${\Bbb C}$-algebras.
 Let
  $$
     e_1\; :=\; f_R(1_R) \hspace{2em}\mbox{and}\hspace{2em}
	 e_2\; :=\; f_S(1_S)\,.
  $$
 Then $e_1=e_1^\prime+ e_2$ is an orthogonal decomposition of $e_1$ by idempotents.
 Since $f_R(R)\subset \Centralizer(e_2)$,
   $$
     f_R(R)\;=\; f_R(R)\cdot e_1^\prime + f_R(R)\cdot e_2 \hspace{2em}\mbox{in $A$}
   $$
  is an orthogonal decomposition of the ${\Bbb C}$-algebra $f_R(R)$; i.e.
   $f_R(R)\simeq f_R(R)\cdot e_1^\prime \times f_R(R)\cdot e_2$ as abstract ${\Bbb C}$-algebras.
 Thus,
  when two quasi-homomorphisms $f_R:R\rightarrow A$ and $f_S: S\rightarrow A$ glue under $h$
    in the sense of Definition~4.2.3,
  one has a commuting diagram
  $$
   \xymatrix{
    &  R \ar[d]_-h \ar[rrr]^-{f_R}
	  &&& f_R(R)\cdot e_1^\prime \times f_R(R)\cdot e_2\;
	             \ar[d]^-{=\; e_2\,\cdot\,\bullet\;=:\; g}_-{\bullet\,\cdot\, e_2}
	                                 \ar@{^{(}->}[rr]               &&    A   \\
	&  S \ar[rrr]^-{f_S}                   &&& f_S(S)\; \ar@{^{(}->}[rr]           && A        & \hspace{-4em}.
   }
  $$
 Or, equivalently in terms of the objects in the opposite category of the category of ${\Bbb C}$-algebras,
  a diagram of morphisms of (generally noncommutative) schemes:
  $$
   \xymatrix{
    \hspace{4em}
    &  \Scheme(R)
	   & \Scheme(f_R(R)\cdot e_1^\prime)  \amalg \Scheme( f_R(R)\cdot e_2)\hspace{9em}  \ar[l]    \\
	&  \Scheme(S)\ar[u]
	   & \Scheme(f_S(S))   \ar[l] \ar[u]      & \hspace{-22em}.
   }
  $$
 
 Assume in addition that the idempotent $e_3:= 1-e_1$ is also orthogonal to $e_2$ \\
 (i.e.\ $e_2\, e_3=e_3\, e_2=0$ in $A$)
   and let
   $$
      A_e\; :=\; \{a\in A\;|\; ae=ea=a\}
   $$
     for an idempotent $e\in A$.
 Note that
   $A_e$ is a ${\Bbb C}$-algebra with a built-in ${\Bbb C}$-algebra quasi-homomorphism $A_e\hookrightarrow A$.	
 Then, $1=e_1^\prime + e_2 +e_3$ is an orthogonal decomposition of $1$ by idempotents in $A$ and
   $$
     A_{e_1^\prime}\times A_{e_2}\times A_{e_3}\;
      \stackrel{\sim}{\longrightarrow}\;
	 A_{e_1^\prime}+ A_{e_2}+ A_{e_3}\; \hookrightarrow\;  A
   $$
	is now an honest ${\Bbb C}$-subalgebra inclusion.
 Since $f_R(R)\subset A_{e_1^\prime}+ A_{e_2}$	and $f_S(S)\subset A_{e_3}$, one has now
  a commuting diagram of ${\Bbb C}$-algebra homomorphisms
  $$
   \xymatrix{
    & &&&  A_{e_1^\prime} \times A_{e_2}\times A_{e_3} \;
	                   \ar@{->>}[d]^-{\pi_{12}}  \ar@{^{(}->}[rr]     &&  A      \\
    &  R \ar[d]_-h \ar[rrr]^-{f_R}
	  &&& \:\:A_{e_1^\prime} \times A_{e_2}\;
	             \ar@{->>}[d]^-{\pi_2}       \\
	&  S \ar[rrr]^-{f_S}                   &&& \;\;A_{e_2}       &&         & \hspace{-4em}.
   }
  $$
 Here
  $\pi_{12}$ and $\pi_2$ are the projection homomorphisms of product ${\Bbb C}$-algebras
   to the indicated components.
 The corresponding commuting diagram of morphisms of noncommutative affine schemes in the opposite category is then:
   ($\amalg$ = disjoint union)
  $$
   \hspace{-1em}
   \xymatrix{
    & &  \Scheme(A_{e_1^\prime}) \amalg \Scheme(A_{e_2})\amalg \Scheme(A_{e_3}) \;
	    &&  \Scheme(A) \ar@{->>}[ll]     \\
    &  \Scheme(R)
	  & \;\rule{0ex}{1.2em}\Scheme(A_{e_1^\prime}) \amalg \Scheme(A_{e_2})
           \ar@{^{(}->}[u]^-{\iota_{12}}  \ar[l]	
                      	  \\
	& \Scheme(S)\ar[u]
  	   &\;\rule{0ex}{1.2em}\Scheme(A_{e_2})  \ar[l]  \ar@{^{(}->}[u]^-{\iota_2}
	   &&         & \hspace{-4em}.
   }
  $$
 Here, $\iota_{12}$ and $\iota_2$ are the inclusion morphism of the connected components as indicated.
  
 Thus,
   a quasi-homomorphism as defined in Definition~4.2.1 is nothing but a roof of ordinary homomorphisms    and,
  with an additional assumption of the completeness of idempotents involved,
  a gluing of quasi-homomorphisms in the sense of Definition~4.2.3
    is nothing but the ordinary gluing of homomorphisms after the replacement by roofs.
}\end{explanation}

\medskip

\begin{remark} {$[$transitivity of subordination relation of idempotents$]$}\; {\rm
 Caution that for a general noncommutative ${\Bbb C}$-algebra $A$,
   idempotents $e_1\prec e_2$ and $e_2\prec e_3$ do not imply $e_1\prec e_3$.
 (E.g.\
    $A={\Bbb C}\langle z_1,z_2,z_2  \rangle
	             /(z_1^2-z_1, z_2^2-z_2, z_3^2-z_3, z_1z_2-z_1, z_2z_1-z_1, z_2z_3-z_2, z_3z_2-z_2)$.)
 However, transitivity of $\prec$ holds for idempotents of the Azumaya/matrix algebra
     $M_r({\Bbb C})$ of rank $r$ over ${\Bbb C}$, for all $r$.
 (For idempotents $e_1\prec e_2$, $e_2\prec e_3$ in $M_r({\Bbb C})$,
     the three $e_1$, $e_2$, $e_3$ must be commuting and hence simultaneously diagonizable.
	 Which implies $e_1\prec e_3$ in the end.)
}\end{remark}

\bigskip

\begin{flushleft}
{\bf Homomorphisms from a $\Delta$-system of ${\Bbb C}$-algebras to a ${\Bbb C}$-algebra $A$}
\end{flushleft}
Let
 $\Delta$ be a fan in $N_{\Bbb R}$ that satisfies Assumption~2.2.2    and
 $A$ a ${\Bbb C}$-algebra.

\bigskip

\begin{definition} {\bf [$\Delta$-system of idempotents in $A$]}\; {\rm
 Let $\mathbf{e}_\Delta^A := \{e_\sigma\}_{\sigma\in \Delta}$ be a collection of idempotents in $A$
  labelled by cones in $\Delta$.
 $\mathbf{e}_\Delta^A$ is said to be
  \begin{itemize}
   \item[(i)]
   a {\it weak $\Delta$-system of idempotents in $A$}
      if $e_\tau\preccurlyeq e_\sigma$ for all $\tau\prec\sigma\in\Delta$;
	
   \item[(ii)]
   a {\it strong $\Delta$-system of idempotents in $A$}
      if $e_\sigma e_{\sigma^\prime} = e_{\sigma\cap \sigma^\prime}$
	   for all $\sigma,\, \sigma^\prime\in\Delta$.
  \end{itemize}	
 Note that
   since $\sigma\cap\sigma^\prime=\sigma^\prime\cap\sigma$,
    all the idempotents  in a strong $\Delta$-system of idempotents in $A$ commute with each other.
 Furthermore, since $\tau\cap \sigma=\tau$ for $\tau\prec\sigma\in\Delta$,
   a strong $\Delta$-system of idempotents in $A$ is naturally a weak $\Delta$-system of idempotents in $A$.
}\end{definition}

\medskip

\begin{lemma-definition} {\bf [reduced idempotents in strong $\Delta$-system]}\;
 Let $\mathbf{e}_\Delta^A=\{e_\sigma\}_{\sigma\in\Delta}$ be a strong $\Delta$-system of idempotents in $A$.
 Then, there exists a unique collection $\{\underline{e}_\sigma\}_{\sigma\in\Delta}$ of orthogonal
     (i.e.\ $\underline{e}_\sigma\,\underline{e}_{\sigma^\prime}=0$ for $\sigma\ne\sigma^\prime$)
	idempotents in $A$
  such that
    $$
      e_\sigma\;=\; \sum_{\tau\preccurlyeq \sigma}\underline{e}_\tau\,.
    $$	
 {\rm	
 $\underline{e}_\sigma$ is called the {\it reduced idempotent} associated to $\sigma\in\Delta$
    from the strong $\Delta$-system $\mathbf{e}_\Delta^A$.
 It is the contribution to $e_\sigma$ after being trimmed away all the boundary contributions.
	}
\end{lemma-definition}	

\medskip

\begin{proof}
 Explicitly, the reduced idempotent $\underline{e}_\sigma$, for $\sigma\in\Delta(k)$, is given by:
  \begin{eqnarray*}
     \lefteqn{
		\underline{e}_{\,\sigma}\;
		:=\;
             e_\sigma\,
			 - \sum_{\tau\in \Delta(k-1),\,\tau\prec\sigma }  e_\tau\,
			 + \sum_{\tau\in\Delta(k-2),\,\tau\prec\sigma}e_\tau\,
			 - \sum_{\tau\in\Delta(k-3),\,\tau\prec\sigma} e_\tau       }\\[1.2ex]
	 && \hspace{2em}
			 +\,\cdots\,
			 + (-1)^{k-1} \sum_{\tau\in\Delta(1),\,\tau\prec\sigma} e_\tau\,
			 + (-1)^k e_\mathbf{0}\,. \hspace{12em}
  \end{eqnarray*}	  	   		
  Equivalently,
    let $\{\tau_1,\,\cdots\,,\,\tau_k\}$ be the set of facets (i.e.\ codimension-$1$ faces) of $\sigma\in\Delta(k)$,
    then
     $$
	   \underline{e}_{\,\sigma}\;
	    =\;
         e_\sigma - \sum_i e_{\tau_i}	+ \sum_{i_1<i_2}e_{\tau_{i_1}} e_{\tau_{i_2}}	
			 -\sum_{i_1<i_2<i_3} e_{\tau_{i_1}}e_{\tau_{i_2}} e_{\tau_{i_3}}
			 +\,\cdots\,
			 + (-1)^k e_{\tau_1}\cdots e_{\tau_k}\,.
	 $$	  	     	
  Such inclusion-exclusion formula is akin and applies to simplicial cones.	
  
  That $\underline{e}_\sigma^2=\underline{e}_\sigma$, for $\sigma\in\Delta$,  and
   that $\underline{e}_\sigma\,\underline{e}_{\sigma^\prime}=0$, for $\sigma\ne\sigma^\prime$,
   follow by induction on `up to $k$', $1\le k\le n-1$.

\end{proof}

\medskip

\begin{definition} {\bf [complete strong $\Delta$-system of idempotents]}\; {\rm
 Let
  $\mathbf{e}_\Delta^A:= \{e_\sigma\}_{\sigma\in\Delta}$ be a strong $\Delta$-system of idempotents in $A$
    and
  $\underline{\mathbf{e}}_\Delta^A := \{\underline{e}_\sigma\}_{\sigma\in\Delta}$
   be the associated collection of reduced idempotents.
 We say that $\mathbf{e}_\Delta^A$ is {\it complete}
   if $\sum_{\sigma\in\Delta}\underline{e}_\sigma =1$.
}\end{definition}

\bigskip

With all these preparations, we are now ready for a key definition:

\bigskip

\begin{definition} {\bf [homomorphisms from inverse $\Delta$-system of ${\Bbb C}$-algebras to $A$]}\; {\rm
 Let
   ${\cal Q} := (\{Q_\sigma\}_{\sigma\in\Delta},
                  \{\iota_{\tau\sigma}^\sharp: Q_\sigma\rightarrow Q_\tau\}_{\tau\prec\sigma\in\Delta})$
   be an inverse $\Delta$-system of ${\Bbb C}$-algebras.
 Denote the identity element of $Q_\sigma$ by $1_{Q_\sigma}$ and that of $A$ by $1$.
 Then
   a collection
     $$
	   \varphi^\sharp\;  :=\;   \{f_\sigma^\sharp: Q_\sigma\rightarrow A\}_{\sigma\in\Delta}
	 $$
	 of ${\Bbb C}$-algebra quasi-homomorphisms indexed by cones in $\Delta$
	 is called a {\it homomorphism} from ${\cal Q}$ to $A$, denoted now $\varphi^\sharp: {\cal Q}\rightarrow A$,
   if
   \begin{itemize}
    \item[(i)] ({\it Inverse $\Delta$-system} of quasi-homomorphisms of ${\Bbb C}$-algebras)
	  $$
	    f_\sigma\;\stackrel{\iota_{\tau\sigma}^\sharp}{\rightsquigarrow}\; f_\tau\,,\;\;
		\mbox{for $\tau\prec\sigma\in\Delta$}\,,
	  $$
      Cf.\ Definition~4.2.3.
      In particular,
	   $\mathbf{e}_{\varphi^\sharp}
	     :=\{e_\sigma := f_\sigma^\sharp(1_{Q_\sigma})\}_{\sigma\in\Delta}$
	   is a weak $\Delta$-system of idempotents in $A$, called the
	   {\it $\Delta$-system of idempotents in $A$ associated to $\varphi^\sharp$}.
	
	\item[(ii)]   ({\it Completeness})\hspace{2em}
	 $\mathbf{e}_{\varphi^\sharp}$ is a complete strong $\Delta$-system of idempotents in $A$.		
   \end{itemize}
  For the simple uniformness of terminology but with slight abuse, we shall call such homomorphisms also
    {\it ${\Bbb C}$-algebra homomorphisms}
	when need to distinguish with homomorphisms of other algebraic structures.
  The set of homomorphisms from ${\cal Q}$ to $A$ is thus denoted
    $\Hom_{{\Bbb C}\dashAlgscriptsize}({\cal Q}, A)$, or simply $\Hom({\cal Q}, A)$.
}\end{definition}

\medskip

\begin{explanation} {\bf [gluings of $\Delta$-system of (local) quasi-homomorphisms to}  \\
              {\bf a ``global" homomorphism]}\;
{\rm
 Given a homomorphism $\varphi^\sharp: {\cal Q}\rightarrow A$ as defined in Definition~4.2.9.
 Let
    $\underline{\mathbf{e}}_{\varphi^\sharp}:= \{\underline{e}_\sigma\}_{\sigma\in\Delta}$,
    $e_\sigma := f_\sigma^\sharp(1_{Q_\sigma})$, $\sigma\in\Delta$,
      be the collection of the reduced idempotents from $\mathbf{e}_{\varphi^\sharp}$    and
    $$
	   A_{\underline{e}_\sigma}\;
	   :=\; \{a\in A\,|\, a e_{\underline{e}_\sigma}\,=\, e_{\underline{e}_\sigma}a\, =\, a \}\,,
	    \hspace{1em}\mbox{$\sigma\in\Delta$}\,.
	$$	
 Then,
   a generalization of Explanation~4.2.4 gives the inclusions of ${\Bbb C}$-algebras
    %
	\begin{eqnarray*}
	  f_\sigma^\sharp(Q_\sigma)
	   &  \subset &  \oplus_{\tau\preccurlyeq \sigma}A_{\underline{e}_\tau}\;\;
	                                 \mbox{(as a multiplicatively closed ${\Bbb C}$-vector subspace of $A$)}          \\[1.2ex]
	   & \simeq   & \times_{\tau\preccurlyeq\sigma} A_{\underline{e}_\tau}\;\;
	                                 \mbox{(as an abstract ${\Bbb C}$-algebra)}\,,
	                                 \hspace{1em}\mbox{$\sigma\in\Delta$}\,,	
	\end{eqnarray*}
    and commuting diagrams of {\it homomorphisms} of ${\Bbb C}$-algebras:
   $$
    \hspace{8em}
	\xymatrix{									
	 && A_\varphi\, :=\, \times_{\sigma\in\Delta}A_{\underline{e}_\sigma}
	         \ar@{^{(}->}[rr]^-{\pi_{A_\varphi}^\sharp}    \ar@{->>}[d]^-{\jmath_\sigma^\sharp}
	     && A                      \\
	 Q_\sigma \ar[rr]^-{f_\sigma^\sharp} \ar[d]^-{\iota_{\tau\sigma}^\sharp}
	    &&  A_{\varphi,\,\sigma}\, :=\,\times_{\rho\preccurlyeq\sigma\in\Delta}A_{\underline{e}_\rho}
                  \ar@{->>}[d]^-{\jmath_{\tau\sigma}^\sharp}		\\
	 Q_\tau \ar[rr]^-{f_\tau^\sharp}
	    &&  A_{\varphi,\,\tau}\, :=\,\times_{\rho\preccurlyeq\tau\in\Delta} A_{\underline{e}_\rho}
	                       &&  & ,\:\:\mbox{for $\tau\prec\sigma\in\Delta$}\,.
    }										
   $$
  Here
     $\jmath_\sigma^\sharp: A_\varphi\rightarrow A_{\varphi,\sigma}$, $\sigma\in\Delta$, and
	 $\jmath_{\tau\sigma}^\sharp: A_{\varphi,\,\sigma}\rightarrow A_{\varphi,\,\tau}$
       are the projection maps onto factors of product ${\Bbb C}$-algebras as indicated.
  The underlying geometry of the diagrams is revealed in terms of commuting diagrams of morphisms of
    noncommutative affine schemes in the opposite category: ($\amalg$ = disjoint union)
   $$
    \hspace{1.6em}
	\xymatrix{									
	 && \Scheme(A_\varphi) = \amalg_{\sigma\in\Delta}\Scheme(A_{\underline{e}_\sigma})
	     && \Scheme(A) \ar@{->>}[ll]_-{\pi_{A_\varphi}}                      \\
	 \Scheme(Q_\sigma)
	    &&  \Scheme(A_{\varphi,\,\sigma})\rule{0ex}{1.2em}
		         = \amalg_{\rho\preccurlyeq\sigma\in\Delta}\Scheme(A_{\underline{e}_\rho})
				    \ar@{^{(}->}[u]^-{\jmath_\sigma}       \ar[ll]_-{f_\sigma}        \\
	 \Scheme(Q_\tau)\ar[u]^-{\iota_{\tau\sigma}}
	    &&  \Scheme(A_{\varphi,\,\tau})\rule{0ex}{1.2em}
		         = \amalg_{\rho\preccurlyeq\tau\in\Delta} \Scheme(A_{\underline{e}_\rho})
				    \ar@{^{(}->}[u]^-{\jmath_{\tau\sigma}}  \ar[ll]_-{f_\tau}
	                       &&  & \hspace{-3.6em},
    }										
   $$	
   for $\tau\prec\sigma\in\Delta$.
 Think of ${\cal Q}$ as defining contravariantly a noncommutative space ``$\Space({\cal Q})$".
 $\varphi^\sharp$ then defines a morphism $\varphi:\Scheme(A)\rightarrow \Space(\cal Q)$
   of noncommutative spaces.
}\end{explanation}

\bigskip

\begin{flushleft}
{\bf (Dynamical, complex algebraic) D-branes on a soft noncommutative scheme}
\end{flushleft}
Let
  $X$ be a (commutative) scheme over ${\Bbb C}$,
  ${\cal E}$ be a locally free ${\cal O}_X$-module of rank $r$,
  $$
   X^{\!\Azscriptsize}\;
     =\;(X,\, {\cal O}_X^\Azscriptsize:= \Endsheaf_{{\cal O}_X}({\cal E}),\, {\cal E})
  $$
     be the Azumaya/matrix scheme over ${\Bbb C}$
	 with the underlying topology $X$ and the fundamental module ${\cal E}$,     and
  $\breve{Y}$ be a soft noncommutative scheme via toric geometry,
   with the structure sheaf
   $$
     {\cal O}_{\breve{Y}}\;
      =\; \breve{\cal R}_\Delta\;	
      :=\; (\{\breve{R}_\sigma\}_{\sigma\in\Delta},\,
	         \{\iota_{\tau\sigma}^\sharp\}_{\tau\prec\sigma\in\Delta})\,.
  $$
  an inverse $\Delta$-system of ${\Bbb C}$-algebras,
  cf.\ Definition~3.10.
Consider the sheaf of sets
  $$
    \Homsheaf_{{\Bbb C}\dashAlgscriptsize}({\cal O}_{\breve{Y}},\, {\cal O}_X^\Azscriptsize)
	 \hspace{1em}\mbox{over $X$}
  $$
  associated to the presheaf defined by the assignment
  $U  \mapsto
     \Hom_{{\Bbb C}\dashAlgscriptsize}(\breve{R}_\Delta,\, {\cal O}_X^\Azscriptsize(U))$
   for $U$ open in $X$    and
  restriction maps
  $
    \Hom_{{\Bbb C}\dashAlgscriptsize}(\breve{R}_\Delta,\, {\cal O}_X^\Azscriptsize(U))
	 \rightarrow
	    \Hom_{{\Bbb C}\dashAlgscriptsize}(\breve{R}_\Delta,\, {\cal O}_X^\Azscriptsize(V))$
   via the post-composition with the ${\Bbb C}$-algebra homomorphism
   ${\cal O}_X^\Azscriptsize(U)\rightarrow {\cal O}_X^\Azscriptsize(V)$
  for open sets $V\subset U\subset X$.
  
\bigskip

\begin{definition} {\bf [D-brane on $\breve{Y}$: morphism $X^{\!\Az}\rightarrow \breve{Y}$]}\; {\rm
 A {\it morphism} $\varphi: X^{\!\Azscriptsize}\rightarrow \breve{Y}$ is by definition a global section, in notation
  $$
    \varphi^\sharp\;:\; {\cal O}_{\breve{Y}}\; \longrightarrow\; {\cal O}_X^\Azscriptsize
  $$
  of
  $\Homsheaf_{{\Bbb C}\dashAlgscriptsize}({\cal O}_{\breve{Y}},\, {\cal O}_X^\Azscriptsize)$.
 Explicitly, $\varphi^\sharp$ is represented by the following data
  $$
    ({\cal U}:= \{U_\alpha\}_{\alpha\in I},
	      \varphi^\sharp_{\cal U}
		    :=   \{\varphi^\sharp_\alpha: \breve{R}_\Delta \rightarrow
			              \End_{{\cal O}_X}({\cal E}|_{U_\alpha})\}_{\alpha\in I})\,,
  $$
  where
    ${\cal U}$ is an open covering of $X$    and
	$\varphi^\sharp_{\cal U}$ a collection of ${\Bbb C}$-algebra homomorphisms
   such that the following diagrams commute
  $$
    \xymatrix{
	   & \End_{{\cal O}_X}({\cal E}|_{U_\alpha}) \ar@{^{(}->}[ld]                      \\
	  \End_{{\cal O}_X}({\cal E}|_{U_\alpha\cap U_\beta})
         & &&&&&	\breve{R}_\Delta  \ar[lllllu]_-{\varphi^\sharp_\alpha}    \ar[llllld]^-{\varphi^\sharp_\beta}\\
	   & \End_{{\cal O}_X}({\cal E}|_{U_\beta})     \ar@{_{(}->}[lu]
	        &&&&& & \hspace{-3.8em},
	}
  $$
  for all $\alpha,\,\beta\in I$.
  
 With the additional data of a connection $\nabla$ on ${\cal E}$
 (the {\it Chan-Paton sheaf with a gauge field})
  this gives a mathematical model for dynamical, complex algebraic {\it D-branes}
  as an extended object moving on the soft noncommutative target-space $\breve{Y}$ over ${\Bbb C}$.
 (One may also introduce a Hermitian structure on ${\cal E}$     and
     consider only unitary connections $\nabla$ on ${\cal E}$.)
}\end{definition}

\medskip

\begin{definition} {\bf [surrogate of $X^{\!A\!z}$ associated to $\varphi$]}\; {\rm
 The completeness condition on idempotents (cf.\ Definition~4.2.9: Condition (ii))
  implies that the inclusion\\
     ${\cal A}_\varphi
	   := {\cal O}_X\langle \varphi^\sharp({\cal O}_{\breve{Y}})   \rangle
	    \hookrightarrow {\cal O}_X^{A\!z}$
	 is an ${\cal O}_X$-algebra homomorphism
	 (rather than only a quasi-homomorphism of ${\cal O}_X$-algebras).
 The corresponding noncommutative scheme $X_\varphi$ over $X$ is called
    the {\it surrogate}   of $X^{\!A\!z}$ associated to $\varphi$.
 The sequence of ${\cal O}_X$-algebra inclusions
   ${\cal O}_X \stackrel{\pi_\varphi^\sharp}{\hookrightarrow}   {\cal A}_\varphi
      \hookrightarrow {\cal O}_X^{A\!z}$
  gives a sequence of dominant morphisms
   $X^{\!A\!z}\longrightaarrow X_\varphi \stackrel{\pi_\varphi}{\longrightaarrow} X$.
 By construction, $\varphi$ factors to a composition of
   $X^{\!A\!z}\rightaarrow X_\varphi \stackrel{f_\varphi}{\longrightarrow}\breve{Y}$.
}\end{definition}

\medskip

\begin{definition} {\bf [image of $\varphi$ and push-forward Chan-Paton sheaf]}\; {\rm
 The two-sided ideal sheaf $\Ker(\varphi^\sharp)$ of ${\cal O}_{\breve{Y}}$ defines
  a soft closed subscheme of $\breve{Y}$, denoted $\Image\varphi$ and called the {\it image} of $\varphi$.
 $\varphi^\sharp$ renders ${\cal E}$ an ${\cal O}_{\breve{Y}}$-module,
   denoted $\varphi_\ast{\cal E}$ and called the {\it push-forward Chan-Paton sheaf} on $\breve{Y}$.
 By construction, $\varphi_\ast{\cal E}$ is supported on $\Image\varphi$.
}\end{definition}

\bigskip

\noindent
{\sc Figure}~4-2-1.\footnote{Revised from:\,
			 {\sl Chien-Hao Liu}  and {\sl Shing-Tung Yau},
			 {\it More on the admissible condition on differentiable\\   maps
			         $\varphi:(X^{\!A\!z}, E; \nabla)\rightarrow Y$ in the construction of the non-Abelian
					 Dirac-Born-Infeld action $S_{DBI}(\varphi, \nabla)$},\\
			arXiv:1611.09439 [hep-th] (D(13.2.1)), {\sc Figure} 2-1.
                           } 


\begin{figure}[htbp]
 \bigskip
  \centering
   \includegraphics[width=0.80\textwidth]{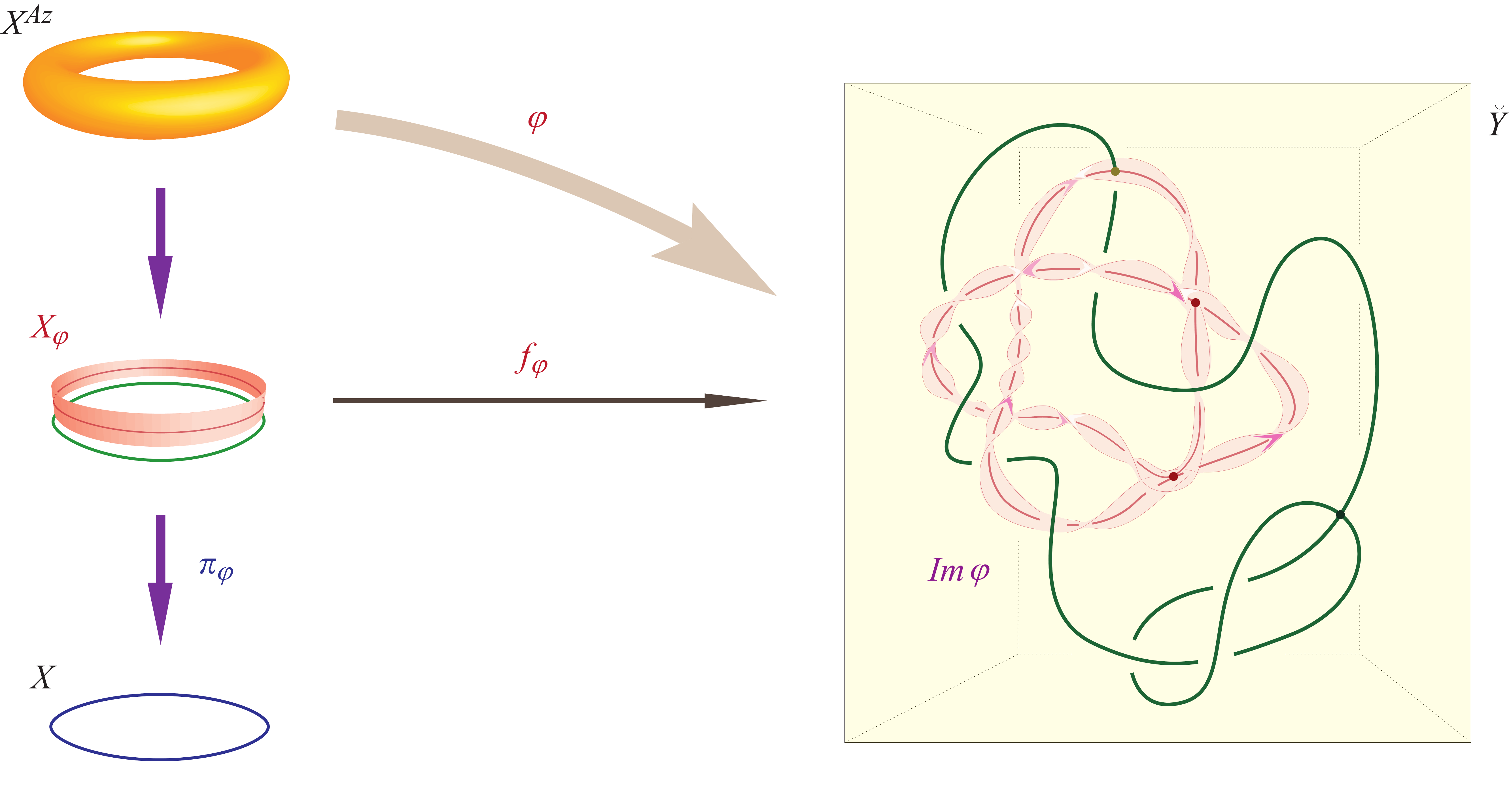}

  \bigskip
  
  \bigskip
  
 \centerline{\parbox{13cm}{\small\baselineskip 12pt
  {\sc Figure}~4-2-1.
  A morphism $\varphi: X^{\!A\!z}\rightarrow \breve{Y}$ is illustrated.
  Note that the image $\varphi(X^{\!A\!z})$ of $\varphi$ in $Y$ can have very rich variations of behavior
   as a closed subscheme of $\breve{Y}$.
  }}
\end{figure}		

\bigskip

We conclude this subsection with a second guiding question:

\bigskip

\begin{question} {\bf [generalized matrix model]}\; {\rm
 For
   $X$ a ${\Bbb C}$-point $\pt$,
   ${\cal E}={\Bbb C}^r$, and
   $\Delta=\{\sigma\}$, where $\sigma$ is a simplicial cone of $N_{\Bbb R}$ of index $1$,
  one has that $\breve{Y}(\Delta)=\nc\mathbf{A}^n$     and
   morphisms from $\pt^{A\!z}$ to $\nc\mathbf{A}^n$ are given by tuples
   $\mathbf{m}=(m_1,\,\cdots\,,\, m_n)$, $m_i\in M_r({\Bbb C})$.
 Such $\mathbf{m}$'s (with possible Hermitian constraints) are fields in a matrix model.
 From this aspect,
  a theory with fields morphisms from $\pt^{A\!z}$ to a soft noncommutative toric scheme $\breve{Y}(\Delta)$
   for a fan $\Delta$ satisfying Assumption~2.2.2
   gives a generalization of matrix models
     via gluing simple matrix models on morphisms $\pt^{A\!z}\rightarrow \nc\mathbf{A}^n$.
 Details of some examples? Consequences? Relations to Topological String Theory?
}\end{question}

\bigskip

\subsection{Two equivalent descriptions of a morphism $\varphi: X^{\!A\!z}\rightarrow \breve{Y}$}

The notion of a soft noncommutative toric scheme $\breve{Y}(\Delta)$ over ${\Bbb C}$
    associated to a fan $\Delta$    can be generalized to
  the notion of a {\it soft noncommutative toric scheme $\breve{Y}_X(\Delta)$ over $X$}
   associated to $\Delta$
  by generalizing
    the notion of an inverse $\Delta$-system of ${\Bbb C}$ monoid algebras
	   $\{{\Bbb C}\langle \breve{M}_\sigma\rangle\}_{\sigma\in\Delta}$
    to the notion of an {\it inverse $\Delta$-system of ${\cal O}_X$ monoid algebras}
	    $\{{\cal O}_X \langle \breve{M}_\sigma\rangle\}_{\sigma\in\Delta}$.
 Similarly, for close subschemes $\breve{Z}$ of $\breve{Y}$.
This defines the {\it product space} $X\times \breve{Y}$
   of the commutative scheme $X$ and the soft noncommutative scheme $\breve{Y}$
   and its structure sheaf ${\cal O}_{X\times\breve{Y}}$.
The built-in inclusion
  $$
     {\cal O}_X\;
	   \hookrightarrow\;  \Center({\cal O}_{X\times \breve{Y}})\;
	   \subset\;  {\cal O}_{X\times\breve{Y}}
  $$
   of ${\cal O}_X$-algebras defines the projection map
  $\pr_X: X\times\breve{Y}\rightarrow X$.
In particular, an ${\cal O}_{X\times\breve{Y}}$-module $\widetilde{\cal F}$ on $X\times\breve{Y}$
  is automatically an ${\cal O}_X$-module,
  denoted $\pr_{X,\,\ast}\,\widetilde{\cal F}$ or simply $\widetilde{\cal F}$.
The built-in inclusion
 $$
   {\cal O}_{\breve{Y}}\; \hookrightarrow\; {\cal O}_{\breve{Y}}\cdot 1\;
    \subset {\cal O}_{X\times\breve{Y}}
 $$
 of ${\Bbb C}$-algebras defines the projection map
 $\pr_{\breve{Y}}: X\times \breve{Y}\rightarrow \breve{Y}$.
 
Let $\widetilde{\cal F}$ be a two-sided ${\cal O}_{X\times \breve{Y}}$-module.
 The {\it support} of $\widetilde{\cal F}$, in notation $\Supp(\widetilde{\cal F})$, is by definition
  the closed subscheme of $X\times\breve{Y}$
   associated to the two-sided ideal sheaf of ${\cal O}_{X\times\breve{Y}}$
   generated by both left annihilators and right annihilators of $\widetilde{\cal F}$
  $$
    \Ann(\widetilde{\cal F})\:
	:=\; (f\in {\cal O}_{X\times\breve{Y}}\,|\,
	                \mbox{$f\cdot \widetilde{\cal F}=0$ or $\widetilde{\cal F}\cdot f=0$})\,.
  $$
 We say that $\widetilde{\cal F}$ is {\it of relative dimension $0$ over $X$}
  if the restriction of $\widetilde{\cal F}$ to each $\{p\}\times\breve{Y}$ for $p$ a ${\Bbb C}$-point on $X$
        is a finitely dimensional ${\Bbb C}$-vector space  and
   ${\cal O}_{\scriptsizeSupp(\widetilde{\cal F})}
	  := {\cal O}_{X\times\breve{Y}}/\Ann(\widetilde{\cal F})$
        is a coherant ${\cal O}_X$-module under $\pr_{X,\,\ast}$.

\bigskip

\begin{definition-lemma} {\bf [graph of morphism $\varphi: X^{\!A\!z}\rightarrow \breve{Y}$]}\; {\rm
 Let $\varphi: X^{\!A\!z}\rightarrow \breve{Y}$ be a morphism, defined via
  $\varphi^\sharp:{\cal O}_{\breve{Y}}
    \rightarrow {\cal O}_X^{A\!z}:=\Endsheaf_{{\cal O}_X}({\cal E})$.
 Then,
   $\varphi^\sharp$ determines a homomorphism of ${\cal O}_X$-algebras
   $$
     \widetilde{\varphi}^\sharp\;:\; {\cal O}_{X\times\breve{Y}}\;
	   \longrightarrow\; {\cal O}_X^{A\!z}\,,\hspace{1em}
	 \mbox{with
	   $f\cdot \breve{r}\longmapsto f\cdot \varphi^\sharp(\breve{r})$
	    for all $f\in {\cal O}_X$ and $\breve{r}\in {\cal O}_{\breve{Y}}$}\,.
   $$
  This defines the morphism $\widetilde{\varphi}:X^{\!A\!z}\rightarrow X\times\breve{Y}$.
  The two-sided ideal sheaf $\Ker(\widetilde{\varphi}^\sharp)$ of ${\cal O}_{X\times\breve{Y}}$
   defines a closed subscheme $\Gamma_\varphi$ of $X\times\breve{Y}$ that is isomorphic to the surrogate
   $X_\varphi$ of $X^{\!A\!z}$ associated to $\varphi$, as noncommutative schemes over $X$.
  By construction,
   $\widetilde{\cal E}:= \widetilde{\varphi}_\ast{\cal E}$
     is an ${\cal O}_{X\times\breve{Y}}$-module supported on $\Gamma_\varphi$.
  The ${\cal O}_{X\times\breve{Y}}$-module $\widetilde{\cal E}$  is called the {\it graph} of $\varphi$.
  It has the property that
   \begin{itemize}
    \item[$\cdot$]
	 $\widetilde{\cal E}$ is of relative dimension $0$ over $X$  and
	 $\pr_{X,\,\ast}\widetilde{\cal E}={\cal E}$\;
	 (i.e.\  $\pr_{X,\,\ast}\widetilde{\cal E}\simeq {\cal E}$ canonically).
   \end{itemize}
 
  The converse also holds: {\it
  Let
   $\widetilde{\cal E}$
     be an ${\cal O}_{X\times\breve{Y}}$-module of relative dimension $0$ over $X$
 	 such that ${\cal E}:= \pr_{X,\,\ast}\widetilde{\cal E}$ is a coherent locally free
	 ${\cal O}_X$-module, say of rank $r$      and\\
   $X^{\!A\!z}:= (X, {\cal O}_X^{A\!z}:= \Endsheaf_{{\cal O}_X}({\cal E}), {\cal E})$
     be the Azumaya scheme$/{\Bbb C}$ with the fundamental module ${\cal E}$.
  Then, $\widetilde{\cal E}$ specifies a morphism $\varphi: X^{\!A\!z}\rightarrow \breve{Y}$.
   } 
 
   Cf.~[L-L-S-Y: {\sc Figure} 2-2-1] (D(2)).
}\end{definition-lemma}

\medskip

\begin{proof}
 Via $\pr_{\breve{Y}}^\sharp$,
 the ${\cal O}_{\breve{Y}}$-module structure on $\widetilde{\cal E}$
   induces a ${\cal O}_{\breve{Y}}$-module structure on ${\cal E}$.
 This defines a homomorphism
   $\varphi^\sharp: {\cal O}_{\breve{Y}}\rightarrow {\cal O}_X^{A\!z}$.

\end{proof}

\bigskip

From Definition/Lemma~4.3.1, one realizes that
   
\bigskip

\begin{lemma}
{\bf [$\varphi$ and moduli stack of $0$-dimensional ${\cal O}_{\breve{Y}}$-modules]}\;
  Let
    $\frak{M}^{0;r}(\breve{Y})$
	be the moduli stack of $0$-dimensional ${\cal O}_{\breve{Y}}$-modules
	 that are of complex dimension $r$ as ${\Bbb C}$-vector spaces.
 Then
   a morphism from an Azumaya scheme over $X$ with the fundamental module of rank $r$ to $\breve{Y}$
    gives a morphism $X\rightarrow \frak{M}^{0;r}(\breve{Y})$.\;
 {\rm   Cf.~[L-L-S-Y: {\sc Figure} 3-1-1] (D(2)).}
\end{lemma}

\newpage
\baselineskip 13pt
{\footnotesize

\vspace{1em}

\noindent
chienhao.liu@gmail.com, 
chienliu@cmsa.fas.harvard.edu; \\
yau@math.harvard.edu

}

\end{document}